\documentclass[10pt]{amsart}
\usepackage[utf8]{inputenc}

\usepackage[margin=1in]{geometry}
\usepackage[expansion=false]{microtype}

\usepackage{comment, amsmath, amsfonts, amssymb, amsthm, amsopn, color, xcolor, graphicx, amscd, latexsym, psfrag}

\usepackage{tikz}
\usetikzlibrary{diagrams, cd, decorations.markings, matrix}

\tikzset
{
  ->-/.style={
    decoration={markings, mark=at position .57 with {\arrow{Latex}}},postaction={decorate}
  }
}

\usepackage[pdfusetitle,unicode,hidelinks]{hyperref}
\usepackage[all]{xy}
\usepackage{eucal}
\usepackage{verbatim}

\newcommand{\<}{\discretionary{}{}{}}

\newcommand{\cA}{\mathcal{A}}
\newcommand{\cC}{\mathcal{C}}
\newcommand{\cD}{\mathcal{D}}
\newcommand{\cE}{\mathcal{E}}
\newcommand{\A}{\mathbb{A}}
\newcommand{\bE}{\mathbb{E}}
\newcommand{\cF}{\mathcal{F}}
\newcommand{\cK}{\mathcal{K}}
\newcommand{\cL}{\mathcal{L}}

\newcommand{\cM}{\mathcal{M}}
\newcommand{\cN}{\mathcal{N}}
\newcommand{\cO}{\mathcal{O}}
\newcommand{\cP}{\mathcal{P}}
\newcommand{\cQ}{\mathcal{Q}}
\newcommand{\R}{\mathrm{R}}
\newcommand{\cS}{\mathcal{S}}
\newcommand{\cT}{\mathcal{T}}
\newcommand{\Z}{\mathbb{Z}}

\newcommand{\Aut}{\mathrm{Aut}}
\newcommand{\CAlg}{\mathrm{CAlg}}
\newcommand{\Cat}{\mathrm{Cat}}
\newcommand{\Ch}{\mathrm{Ch}}
\newcommand{\coev}{\mathrm{coev}}

\newcommand{\dR}{\mathrm{dR}}

\newcommand{\Tate}{\mathrm{Tate}}
\newcommand{\dual}{\mathrm{dual}}

\newcommand{\ev}{\mathrm{ev}}
\newcommand{\Fr}{\mathrm{Fr}}
\newcommand{\Fun}{\mathrm{Fun}}
\newcommand{\Hom}{\mathrm{Hom}}
\newcommand{\id}{\mathrm{id}}
\newcommand{\iHom}{\underline{\mathrm{Hom}}}

\newcommand{\Map}{\mathrm{Map}}
\newcommand{\Mod}{\mathrm{Mod}}
\newcommand{\Mon}{\mathrm{Mon}}
\newcommand{\op}{\mathrm{op}}
\newcommand{\oplax}{\mathrm{oplax}}
\newcommand{\Prst}{\mathcal{P}\mathrm{r}^{\mathrm{St}}}
\newcommand{\bPrst}{\mathbf{Pr}^{\mathrm{St}}}
\newcommand{\pt}{\mathrm{pt}}

\newcommand{\QCoh}{\mathrm{QCoh}}
\newcommand{\Ind}{\mathrm{Ind}}
\newcommand{\Coh}{\mathrm{Coh}}
\newcommand{\dmod}{\mathrm{\textnormal{-}mod}}
\newcommand{\Sch}{\mathrm{Sch}}
\newcommand{\Sm}{\mathrm{Sm}}

\newcommand{\IndCoh}{\mathrm{IndCoh}}
\newcommand{\rig}{\mathrm{rig}}

\newcommand{\tens}{\otimes}
\newcommand{\Td}{\mathrm{Td}}
\newcommand{\tTate}{\mathrm{tTate}}
\newcommand{\Tr}{\mathrm{Tr}}

\newcommand{\bu}{\mathbf{1}}
\newcommand{\Var}{\mathrm{Var}}
\newcommand{\Bl}{\mathrm{Bl}}
\newcommand{\Sym}{\mathrm{Sym}}

\newcommand{\twocell}[1]{\ar@{}[#1]^(.37){}="a"^(.63){}="b" \ar@{=>} "a";"b"}
\newcommand{\twocelllabel}[2]{\ar@{}[#1]^(.37){}="a"^(.63){}="b" \ar@{=>}^{#2} "a";"b"}
\newcommand{\defterm}[1]{\textbf{\emph{#1}}}
\newcommand{\adj}[2]{
\xymatrix{
#1 \ar@<.5ex>[r] & #2 \ar@<.5ex>[l]
}
}

\numberwithin{equation}{section}

\newtheorem{theorem}[equation]{Theorem}
\newtheorem{corollary}[equation]{Corollary}
\newtheorem{proposition}[equation]{Proposition}
\newtheorem{lemma}[equation]{Lemma}
\theoremstyle{definition}
\newtheorem{definition}[equation]{Definition}
\theoremstyle{remark}
\newtheorem{remark}[equation]{Remark}
\newtheorem{example}[equation]{Example}

\def\HH{\mathrm{HH}}
\newcommand{\ch}{\mathrm{ch}}
\def\Perf{\mathrm{Perf}}

\def\Mot{\mathrm{Mot}}
\def\MOT{\mathbb M\mathrm{ot}}
\def\mon{\mathrm{mon}}
\def\sat{\mathrm{sat}}
\def\Sp{\mathrm{Sp}}
\newcommand{\cB}{\mathcal{B}}

\def\H{\mathrm H}
\def\K{\mathbb{K}}

\let\scr=\mathcal

  \numberwithin{equation}{section}

\theoremstyle{plain}
\newtheorem{innercustomthm}{Theorem}
\newenvironment{customthm}[1]
{\renewcommand\theinnercustomthm{#1}\innercustomthm}
  {\endinnercustomthm}
\newtheorem*{theorem*}{Theorem}

\begin{document}
\title{The categorified Grothendieck--Riemann--Roch theorem}

\begin{abstract}
In this paper we prove a categorification of the Grothendieck--Riemann--Roch theorem.  Our result  implies in particular  a Grothendieck--Riemann--Roch theorem for To\"en and 
Vezzosi's secondary Chern character.    
As a main application, we establish a comparison between the To\"en--Vezzosi Chern character and the classical Chern character, and show that  the 
categorified Chern character  recovers the classical de Rham realization. 
\end{abstract}

\author{Marc Hoyois}
\email{marc.hoyois@ur.de}
\address{Fakultät für Mathematik\\ Universität Regensburg\\ 93040 Regensburg\\ Germany}

\author{Pavel Safronov}
\email{pavel.safronov@math.uzh.ch}
\address{Institut f\"{u}r Mathematik\\ Universit\"{a}t Z\"urich\\ Winterthurerstrasse 190\\ 8051 Zurich\\ Switzerland}

\author{Sarah Scherotzke}
\email{sarah.scherotzke@uni.lu}
\address{Department of Mathematics\\
Universit\'{e} du Luxembourg \\
Maison du Nombre\\
6, Avenue de la Fonte\\
L-4364 Esch-sur-Alzette, Luxembourg}

\author[Nicol\`o Sibilla]{Nicol\`o Sibilla}
\email{ nsibilla@sissa.it, N.Sibilla@kent.ac.uk}
\address{Sissa\\ Via Bonomea 265, 34136 Trieste TS
Italy\\
and
University of Kent\\ 
Canterbury, Kent CT2 7NF\\
UK}

\maketitle

\vspace{-1em}
{\small \tableofcontents}

\section{Introduction}
In this paper we prove a Grothendieck--Riemann--Roch theorem for the categorified Chern 
character defined in \cite{TV2} and in \cite{HSS}. Our result yields in particular a Grothendieck--Riemann--Roch theorem for To\"en and 
Vezzosi's secondary Chern character, thus answering a question raised in \cite{TV1}. Our main applications include 
\begin{enumerate}
\item a proof that the To\"en--Vezzosi's Chern character \cite{TV1,TV2}  matches the classical Chern character \cite{mccarthy,keller1999cyclic}, %In \cite{BN3} Ben--Zvi and Nadler 
\item a proof that, in the geometric  setting, the categorified Chern character recovers the de Rham realization of smooth   algebraic varieties. \end{enumerate}  Throughout the paper we will work over a 
 fixed   connective $\mathbb{E}_\infty$ ring spectrum $k.$
 
\subsection{The categorified Chern character}
Categorified invariants arise naturally   
in homotopy theory.  
Over the last thirty years  a rich picture  relating 
the chromatic hierarchy of cohomology theories to  categorification has emerged.  On the first step of the chromatic ladder, 
 $K$-theory  classifies vectors bundles, which are one-categorical objects.   
%the $K$-theory of a space 
%$X$ classifies vectors bundles over $X,$ which are one-categorical objects.   
Cohomology theories of higher chromatic depth, %higher up on the chromatic ladder, 
such as elliptic cohomology,  are expected to classify 
 higher-categorical geometric structures.  The literature on these aspects is vast:  
we refer the reader, for instance, to 
\cite{BDR} for an account of 
the connections between  elliptic cohomology %via 
and %a categorification of the theory of vector bundles. 
the theory of 2-vector bundles.  %In a different but related development, 
%Stolz and Teichner have been pursuing a program where cycles in elliptic cohomology should classify parametrized higher TQFT-s.
%From a related 
%perspective  the Stolz-Teichner program is directed at realizing elliptic cohomology as classifying higher categorical gadgets related to TQFT. 

A different source of motivations for 
studying categorified invariants  
comes from representation theory.  
From a modern perspective, the Deligne--Lusztig theory of  character sheaves can be viewed as an early pointer to the existence 
of an interesting 
picture of categorified characters.  
In the last decade   categorical actions have become a mainstay of geometric representation theory. 
As shown by Khovanov--Lauda and Rouquier \cite{khovanov2009diagrammatic, khovanov2011diagrammatic, rouquier20082}, categorical actions encode subtle positivity properties.   
Further, they   
play a key role in recent approaches to the geometric Langlands program 
due to  Ben-Zvi and Nadler \cite{BNlsc, BN3, BN2}, Gaitsgory, Arinkin, and Rozenblyum \cite{Ga1, GR}. A comprehensive character theory of categorical actions of finite groups was developed by Ganter and Kapranov 
in \cite{Ganter20082268}, extending  earlier results  
of Hopkins, Kuhn, and Ravenel \cite{hopkins1992morava}.

  In this paper we build  on the theory of categorified invariants of  stacks developed  by To\"en and Vezzosi in \cite{TV1} and \cite{TV2}.   
If $X$ is a scheme (or stack), the  %ordinary 
Chern character  is an assignment mapping vector bundles over $X$ to classes  in the Chow group  %, or in  any other algebraic incarnation of the cohomology of $X$,  such as the Hochschild homology $\HH(X)$. 
or in  any other  incarnation of its cohomology,  such as its Hochschild homology $\HH(X)$.  
 To\"en and Vezzosi's \emph{secondary Chern character} is a categorification of the ordinary Chern character. It takes as input  a type of categorified bundles,   given by %sufficiently 
fully dualizable sheaves of categories over $X$ locally tensored over $\Perf(X),$ and lands in a higher version of Hochschild homology.

  More precisely,   
  fully dualizable 
 sheaves of categories  over $X$ %locally tensored over $\QCoh(X),$categorified bundles 
form an  $\infty$-category denoted by $\mathrm{ShvCat}^{\mathrm{sat}}(X).$ The  
To\"en--Vezzosi's secondary Chern character is a morphism 
\begin{equation}
\label{secochernchar}
\ch^{(2)}\colon \iota_0(\mathrm{ShvCat}^{\mathrm{sat}}(X)) \rightarrow \cO(\cL^2 X)
\end{equation}
where $\iota_0(\mathrm{ShvCat}^{\mathrm{sat}}(X))$ is  the maximal 
$\infty$-subgroupoid of $\mathrm{ShvCat}^{\mathrm{sat}}(X),$ and the  target $\cO(\cL^2 X)$  is the \emph{secondary Hochschild homology} of $X.$

As explained in \cite{HSS}  the secondary Chern character is %$\ch^{(2)}$ is in fact 
the shadow of a much richer categorified  
character theory %, which takes the shape of  
encoded in a symmetric monoidal functor of $\infty$-categories 
\begin{equation}
\label{catchcatchfirst}
 \Ch\colon 
\mathrm{ShvCat}^{\mathrm{sat}}(X) \longrightarrow \Perf(\cL X).   
\end{equation}  
 %which is called   the categorified Chern character.  
 We can recover  the %To\"en--Vezzosi's  
 secondary Chern character %can be recovered from (\ref{catchcatchfirst}) %can be captured through $\Ch$ 
  by taking maximal $\infty$-subgroupoids in (\ref{catchcatchfirst}), and then applying the ordinary Chern character
$$
\xymatrix{
\iota_0\mathrm{ShvCat}^{\mathrm{sat}}(X) 
\ar@/^ 18pt/@{->}[rr]^ -{\ch^{(2)}}  \ar[r]^-{\iota_0(\Ch)} &  \iota_0\Perf(\cL X) 
\ar[r]^-{\ch} & \HH(\cL X) \simeq \cO(\cL^2 X). 
}
$$
%Our previous paper \cite{HSS} was devoted to the  construction of the categorified Chern character %(\ref{catchcatchfirst})  
% and to the investigation of some of its fundamental properties.  

In this paper we carry forward the investigation of the categorified Chern character. 
Our main result is  a categorified   Grothendieck--Riemann--Roch  theorem for $\Ch$.% for the categorified Chern character. 

\subsection{The categorified GRR  theorem} 
We actually work in a setting which differs slightly from  (\ref{catchcatchfirst}). Technical issues  compel us to restrict to the \emph{affine} context,  where categorical sheaves are captured globally through the action of a symmetric monoidal category. As we explain in Section \ref{inmanycases}  below, however, affineness in the categorified setting is a much less severe restriction  than in   ordinary algebraic geometry.  

Our formalism %we shall work with  %We will work with a  
%formalism which 
applies to any presentable and  stable symmetric monoidal $\infty$-category $\cC$. Let $\cL \cC:=S^1 \otimes \cC \,$ be the \emph{loop space} of $\cC.$  %As explained for instance in \cite{BFN},  
The loop space $\cL \cC \simeq \cC \otimes_{\cC \otimes \cC}\cC$ is  in a precise sense  the  Hochschild homology  of the commutative algebra $\cC.$  %As in ordinary homological algebra, 
It is therefore  the natural 
receptacle of Chern classes of \emph{dualizable $\cC$-modules}. 
These are presentable categories carrying an action of $\cC,$ which are $\cC$-linearly 
dualizable.  % as $\cC$-linear categories.  
They form an $\infty$-category denoted by $\Mod_\cC^\dual.$ The categorified Chern character is a symmetric  monoidal functor    
\begin{equation}
\label{catchsymmmon}
\mathrm{Ch}\colon \Mod_\cC^\dual \to  \cL\cC.
\end{equation}

 In classical homological algebra, 
the Chern character  factors through the fixed locus for the canonical  $S^1$-action  on the Hochschild complex. %, which calculates negative cyclic homology.   
This is a manifestation of    general  
 rotation invariance  properties 
 of  trace maps,  which are themselves a     special instance  %best understood in the context of %are themselves an instance of  %encapsulated in  %just an aspect 
 of the vast array of symmetries encoded  
in a TQFT,  
see \cite{TV2}, \cite{HSS} and \cite{BN3}.  This feature persists at the categorified level,  and  we will consider the % consider the   
%there is an 
$S^1$-equivariant refinement of the categorified Chern character    
\begin{equation}
\label{catchsymmmons1}
\mathrm{Ch}^{S^1}\colon \Mod_\cC^\dual \longrightarrow  (\cL\cC)^{S^1}.
\end{equation}  

 \subsubsection{The main theorem}
\label{subsubmainthrm}
The  Grothendieck--Riemann--Roch (GRR) theorem 
%is a statement 
encodes the compatibility between  Chern character and pushforward. As explained in \cite{M},   the  classical GRR theorem 
 can be  
viewed as the conflation of two distinct commutativity 
statements.  It is the first of these  two statements  
which is especially relevant for the purposes of categorification. 
 To clarify this, let us briefly review  
 the setting of 
the classical  GRR theorem. 

Let $f\colon X \to Y$ be a proper map between smooth and quasi-projective schemes over a field, and let 
$ \, 
f_*\colon \Perf(X) \to \Perf(Y)
\,$  
be the pushforward. 
The first half of the %statement encapsulated in the GRR theorem  is the claim that the diagram    %first statement encapsulated in the 
GRR theorem consists of the claim that the diagram 
\begin{equation}
\label{classicalGRR}
\begin{gathered}
\xymatrix{
%\iota_0\Perf(X) \ar[r] \ar[d]_-{f_*} & 
%\mathrm{K}(X) = \mathrm{K}(\Perf(X)) 
\iota_0\Perf(X) \ar[r]^-{\ch} \ar[d]_-{f_*} &  \cO(\cL X) \simeq \HH(\Perf(X)) \ar[d]^-{f_*} \\
\iota_0\Perf(Y) 
%\ar[r] & 
%\mathrm{K}(Y) = \mathrm{K}(\Perf(Y)) 
\ar[r]^-{\ch}  &  \cO(\cL Y) \simeq  \HH(\Perf(Y))}
\end{gathered}
\end{equation}
commutes.   Next  
 we %take  the additional step of 
 %reformulating the commutativity 
can reformulate the commutativity of (\ref{classicalGRR})  in terms of differential forms 
%of translating  (\ref{classicalGRR})  
%this picture in terms of differential 
%forms via 
%
%full statement of the classical GRR theorem can be recovered %in full %from (\ref{classicalGRR}) 
%by  composing on the right the Chern character maps  with 
via the HKR isomorphism  
$$
 \HH(\Perf(X)) \xrightarrow{\simeq}  \H^*(X, \oplus_{i \geq 0} \Omega^{i}_X), \quad 
  \HH(\Perf(Y)) \xrightarrow{\simeq}  \H^*(Y, \oplus_{i \geq 0} \Omega^{i}_Y).  
 $$
The second half of the  GRR theorem is  about the interplay between the pushforward and the HKR equivalence: they fail to commute, but this can be obviated by turning on a correction term given by the Todd class.

Our main theorem is a categorification  %counterpart 
of the \emph{first half} of the GRR theorem. We refer the reader to Remark \ref{secondhalf} for a discussion the second half of the GRR theorem in the categorified setting. Let $f\colon \cD \to \cC$ be a \emph{rigid} symmetric monoidal functor between presentable and stable symmetric monoidal categories. The map  
$f$ induces a functor between loop spaces $ \, \cL f\colon \cL \cD \to \cL \cC \, $ with right-adjoint $\, 
\cL f^\R.$ 
 Rigidity  implies that there is a well-defined pushforward of dualizable modules $$
f_*\colon \Mod_\cC^\dual \longrightarrow \Mod_\cD^\dual. 
$$ 
 We are ready to state our main result. 
\begin{customthm}{A}[The categorified GRR Theorem, Theorem \ref{thm:GRR}]
\label{main1}
There is a commutative square   of $\infty$-categories 
\begin{equation}
\label{eqmainmainmainmainmain}
\begin{gathered}
\xymatrix{  \Mod_\cC^\dual \ar[r]^-{\Ch^{S^1} } 
\ar[d]_-{f_*} 
& (\cL \cC)^{S^1}    \ar[d]^-{\cL f^\R}  \\  
 \Mod_\cD^\dual \ar[r] ^-{\Ch^{S^1}} 
&  (\cL \cD)^{S^1}.  }
\end{gathered}
\end{equation}
\end{customthm} 
In Section \ref{grrandmotives} and \ref{inmanycases}  we reformulate  Theorem \ref{main1} in the more specialized settings of non-commutative motives and monoidal categories of geometric origin. Next, in Section \ref{applicationsoftheGRR} we explain some of its applications. % of Theorem \ref{main1}. 

\subsubsection{GRR and motives} 
\label{grrandmotives}
In addition to  $S^1$-invariance  the  Chern character inherits a second important property of trace maps: it is  \emph{additive}, so it factors through K-theory.  One categorical level up, Verdier localizations of $\cC$-linear categories replace short exact sequences,  
 % short exact sequences  correspond to  Verdier localizations of $\cC$-linear categories, 
 and the category of  \emph{noncommutative  $\mathcal{\cC}$-motives}  takes up the role 
%categoriefies 
of K-theory.  The category of 
$\cC$-motives was introduced in \cite{HSS}, building on the work of Cisinski--Tabuada and Blumberg--Gepner--Tabuada \cite{CT, BGT}, see also \cite{robalo} for closely related constructions. % in the absolute case. %The categorification of  short exact sequence of modules are  Verdier localizations of $\cC$-linear categories, while the role of K-theory is played by  the category of  \emph{non-commutative (nc) $\mathcal{\cC}$-motives}. 
%This  is  
%a relative version of the  category of 
%nc motives  of  Cisinski--Tabuada  and Blumberg--Gepner--Tabuada \cite{CT, BGT}, which was introduced in \cite{HSS}. 
%The theory requires %in fact % to define nc motives we have to 
%to assume 
The theory applies in a more limited  generality than Theorem \ref{main1}, as we require $\cC$   %is required 
to be generated by its subcategory of compact objects $\cC^\omega.$ %of compact objects. % \emph{compactly generated}: $\, \cC \simeq \mathrm{Ind}(\cC^\omega)$.  

Let $f\colon \cD \to \cC$ be a rigid functor of compactly generated symmetric monoidal categories. There is a well-defined pushforward functor between categories of localizing noncommutative motives  
$$ 
f_*\colon\MOT(\cC^\omega) \longrightarrow \MOT(\cD^\omega).
$$ 
Then Theorem \ref{main1} specializes to the following statement, which closely parallels the classical K-theoretic formulation of the GRR theorem.  
\begin{customthm}{B}[Theorem \ref{Motivic GRR}]
\label{mainB}
There is a  commutative square   
of $\infty$-categories 
$$ 
\xymatrix{ 
%\mathrm{Mot}(\cC^\omega) \ar[r]^-{\Ch^{S^1} } 
%\ar[d]_-{f_*} 
%& (\cL \cC)^{S^1}    \ar[d]^-{\cL f^\R} && 
\MOT(\cC^\omega) \ar[r]^-{\Ch^{S^1} } 
\ar[d]_-{f_*} 
& (\cL \cC)^{S^1}    \ar[d]^-{\cL f^\R} \\ 
%\mathrm{Mot}(\cD^\omega) \ar[r] ^-{\Ch^{S^1}} 
%&  (\cL \cD)^{S^1} && 
\MOT(\cD^\omega) \ar[r] ^-{\Ch^{S^1}} 
&  (\cL \cD)^{S^1}.}
$$ 
\end{customthm} 

 \subsubsection{The geometric setting}
\label{inmanycases} 
Our results hold in the categorified  affine setting:  we consider  modules over symmetric monoidal categories   rather than general categorical sheaves on stacks. Surprisingly, however, this encompasses many  examples of geometric interest.  
In fact the global sections functor  $$ 
\Gamma\colon \mathrm{ShvCat}^{\mathrm{dual}}(X) \rightarrow \Mod^{\mathrm{dual}}_{\QCoh(X)} \,,   
$$ 
although not an equivalence in general, is an equivalence for a large class of derived stacks   called \emph{1-affine stacks}.   Gaitsgory proves in \cite{Ga1} that quasi-compact and   
quasi-separated schemes and semi-separated
Artin stacks of  finite type (in characteristic zero) are all  examples of  1-affine stacks. 
For 1-affine stacks  Theorem \ref{main22} below captures the full geometric picture of the categorified GRR theorem. %, but work still needs to be done to push our methods beyond the 1-affine setting. 

If $X$ is a  derived stack 
there is a natural $S^1$-equivariant map 
$ \, 
\cL \QCoh(X) \to  \QCoh(\cL X) \, , 
 $  
 where $\cL X$ is the free loop stack of $X$.
%If we assume additionally that $X$ is 1-affine the composite 
%$$
%\mathrm{ShvCat}^\dual(X) \xrightarrow{\Gamma(-)} \Mod^\dual _{\QCoh(X)}\xrightarrow{\Ch} \cL \QCoh(X) \longrightarrow \QCoh(\cL X)
%$$ 
%coincides with the geometric categorified Chern character defined  in \cite{HSS}:  in particular, restricting to fully dualizable objects recovers  the functor  (\ref{catchcatchfirst}) considered at the beginning of this introduction. 
%
%
If $f: X \to Y$ is a map of derived stacks, we denote by $f$ %use the same symbol  $f$ %in the same way 
%to denote 
also the symmetric monoidal pullback functor
$ \, 
f: \QCoh(Y) \to \QCoh(X) \, .
$   
Following Gaitsgory,  we introduce  \emph{passable}   maps of stacks. Passability is a relatively minor assumption, and is satisfied in most cases of  geometric interest.  

In Proposition \ref{passableimpliesrigid} of the main text we show that  pullback   functors along  passable morphisms are rigid. This together with  Theorem \ref{main1} immediately implies the following statement. \begin{customthm}{C}%[%The geometric categorified GRR, 
%Theorem \ref{main33}]
\label{main22}
Let $X \stackrel{f} \to Y$ be a passable morphism of derived stacks. Then there is a commutative diagram of $\infty$-categories   
$$ 
\xymatrix{
\Mod_{\QCoh(X)}^{\mathrm{dual}} \ar[r]^-{\Ch^{S^1} } 
\ar[d]_-{f_*} 
& \QCoh(\cL X)  ^{S^1} \ar[d]^-{\cL f_*} \\ 
\Mod_{\QCoh(Y)}^{\mathrm{dual}} \ar[r] ^-{\Ch^{S^1}  } 
&  \QCoh(\cL Y) ^{S^1}.
}
$$
%Further, if $X$ and $Y$ are 1-affine there is a commutative diagram 
%$$ 
%\xymatrix{ 
%\mathrm{ShvCat}^{\mathrm{dual}}(X) \ar[r]^-{\Ch^{S^1}  } 
%\ar[d]_-{f_*} 
%& \QCoh(\cL X)  ^{S^1} \ar[d]^-{\cL f_*} \\ 
%\mathrm{ShvCat}^{\mathrm{dual}}(Y) \ar[r] ^-{\Ch^{S^1}  } 
%&  \QCoh(\cL Y) ^{S^1}    
%}
%$$
\end{customthm} 
  \begin{remark}
\label{secondhalf}
The statement which we called %to which we referred 
 in Section \ref{subsubmainthrm}  the  \emph{second half} of the  GRR theorem can also be categorified, 
but becomes essentially trivial.    
Let $X$ be a semi-separated derived Artin stack in characteristic zero. By \cite[Theorem 6.9]{BNlsc}  the HKR isomorphism lifts to an equivalence  
$\exp$ of formal stacks  between 
\begin{itemize}
\item the shifted tangent complex of $X$ completed at the zero section, $\widehat{\mathbb{T}_X[-1]} $  
\item and  the loop stack of 
$X$ completed at the constant loops, $\widehat{\cL X}$.
\end{itemize}
%These are equivalent via the exponential map 
%$ \, 
%\mathrm{exp}:\widehat{\mathbb{T}_X[-1]} \to \widehat{\cL X} 
%\, .$  
%The classical HKR theorem encodes the induced isomorphism between the algebras of regular functions.  % on $\widehat{\mathbb{T}_X[-1]}$ and $\widehat{\cL X}.$ 
%Similarly, 
The categorified HKR consists of the statement that $\exp$ induces an equivalence   
\begin{equation}
\label{categorifiedhkr}
\exp^*\colon\QCoh(\widehat{\cL X}) \stackrel{\simeq} \longrightarrow \QCoh(\widehat{\mathbb{T}_X[-1]}).
\end{equation}
As $\exp^*$ is a pullback, it is  compatible with  pullbacks and pushforwards along maps of stacks: contrary to the classical setting, incorporating the HKR equivalence (\ref{categorifiedhkr}) does not alter the commutativity of the GRR diagram (\ref{eqmainmainmainmainmain}), which stays commutative on the nose. \end{remark}
 \subsection{Applications of the categorified GRR}
 \label{applicationsoftheGRR} 
 Theorem \ref{main1},  \ref{mainB} and \ref{main22} have several interesting consequences. They provide powerful tools to establish  %delicate 
 comparison results for the ordinary and categorified Chern character. Our applications fall into  three main areas: 
 \begin{enumerate}
 \item \emph{The ordinary  Chern character}. To\"en and Vezzosi give an alternative construction of the  Chern character, which is the one we use throughout the paper. Theorem \ref{main1} implies 
 that it matches the classical definition. % different constructions of the  Chern character agree. 
 %It also recovers a more general form of the classical GRR theorem  that applies to sheaves of categories. %Also, it yields a generalization of the  classical  GRR theorem that applies to sheaves of  categories.  
 \item  \emph{The secondary Chern character}. Theorem \ref{mainB} implies a GRR statement for the  %To\"en--Vezzosi's 
 secondary Chern character. This yields a comparison  between  secondary Chern character and motivic character maps that had already appeared in  the literature. 
 \item \emph{The de Rham realization}.  Theorem \ref{main22}  implies that  in the geometric setting   $\Ch^{S^1}$ %coincides with 
 matches the de Rham realization. This shows  in particular  that the Gauss--Manin connection is  of non-commutative origin.   
 \end{enumerate}
 
 \subsubsection{The ordinary Chern character}
 The classical definition of the 
 Chern character for $k$-linear  categories  is due to McCarthy \cite{mccarthy} and Keller \cite{keller1999cyclic}, and rests on %with values in Hochschild homology hinges on 
 the naturality   of Hochschild homology. Let 
 $\cA$ be a stable $k$-linear category. If $x$ is an object of $\cA,$ let $ \, \phi_x\colon 
 \Perf(k)  \to \cA
 \, $ be the unique $k$-linear functor 
mapping $k$ to $x.$ Then the Chern character   is defined by the formula 
\begin{equation}
\label{classicalchernchar}
 x \in \mathrm{Ob}(\cA) \mapsto \mathrm{HH}(\phi_x)(1) \in \mathrm{HH}_0(\cA). 
\end{equation}  
In \cite{BN3} Ben--Zvi and Nadler  % give a modern account of 
revisit  (\ref{classicalchernchar}) from the vantage point of the functoriality properties of traces in symmetric monoidal $(\infty,2)$-categories. Let $\Mod_k$ be the symmetric monoidal 
$\infty$-category of $k$-modules, and let $\cC$ be a dualizable $\Mod_k$-module. The Hochschild homology  of $\cC$ %can be identified 
coincides with the \emph{trace of} $\cC$  as a dualizable $\Mod_k$-module %object in $\mathrm{Mod}^{\mathrm{dual}}_{\Mod_k}$ %the symmetric monoidal  $\infty$-category of $\mathrm{Mod}_k$-modules %is identified with its Hochschild homology 
$$ \mathrm{HH}(\cC) \simeq \Tr(\cC) \in \mathrm{Mod}_k .$$   
The trace is functorial and thus, under standard identifications, it yields a map of $\infty$-groupoids $\,\cC^\dual \to \HH(\cC) \, $:  %
%Chern character is the map of $\infty$-groupoids $\iota_0 \cC^{\mathrm{dual}} \to \HH(\cC)$ determined 
% via functoriality of  traces 
\begin{equation}
\begin{gathered}
\label{bznchbzn}
\xymatrix{
\cC^{\mathrm{dual}} \simeq  \Hom_{\mathrm{Mod}^{\mathrm{dual}}_{\Mod_k}}(\mathrm{Mod}_k, \cC) \ar[d]^-{\mathrm{Tr}(-)} \\   \mathrm{Hom}_{\Mod_k}(\Tr(\Mod_k), \Tr(\cC)) \simeq  \mathrm{Hom}_{\Mod_k} (k, \HH(\cC))  \simeq \HH(\cC). }
\end{gathered}
\end{equation}
Ben-Zvi and Nadler take (\ref{bznchbzn}) as the definition of the Chern character. Passing to sets of connected components in  (\ref{bznchbzn})  gives back  (\ref{classicalchernchar}).  
 
To\"en and Vezzosi  give a different     definition of  the  Chern character, %on $\iota_0\cC^{\mathrm{dual}}.$ 
which requires additionally that $\cC$ carries a symmetric monoidal structure.  %, which we denote respectively $\ch$ and $\Ch,$  
%in terms of \emph{monodromy}: coherent and categorical sheaves 
The objects of $\cC,$ pulled back to the loop space via the map 
$$  
\cC \to S^1 \otimes \cC =  \cL \cC \quad \text{induced by the inclusion} \quad \pt \to S^1 \, ,
$$ %induced by the inclusion of the base point $\pt \in S^1$
acquire a canonical auto-equivalence, called \emph{monodromy}. %  If $E$ is a bundle or a sheaf of categories over a derived stack $X,$ its pullback to the free loop space $\cL X$ has a canonical self-map  %and this is 
%which is 
%called the monodromy automorphism. 
To\"en and Vezzosi define 
the  Chern character  as the trace of the monodromy auto-equivalence, and this is the definition we use throughout the paper. The reader can find in \cite{TV1} an explanation of the beautiful geometric heuristics motivating 
To\"en and Vezzosi's approach.  Their  construction   yields a map of $\infty$-groupoids landing in the endomorphisms of the unit  object of $\cL \cC$  \begin{equation}
\label{TVuncatChern}
\ch: \cC^\dual \longrightarrow \Omega \cL \cC.
\end{equation}

\begin{customthm}{D}[Theorem \ref{TVchern=BZNchern}]
\label{mainD}
Under the canonical identification 
$ \, 
\Omega \cL \cC \simeq \HH(\cC)
\, ,$ 
To\"en and Vezzosi's Chern character (\ref{TVuncatChern}) coincides with (\ref{bznchbzn}). \end{customthm}
%Theorem \ref{mainD} recovers and extends earlier results by Markarian \cite{M} and C{\u{a}}ld{\u{a}}raru \cite{C} for the ordinary Chern character.  
 \subsubsection{The secondary Chern character} 
 In the main text  
 some of the results in this section will be formulated more generally  for compactly generated symmetric monoidal categories, but we will  limit our present exposition to the geometric setting.   %However, we   limit our present exposition to the geometric setting.    

Let $X$ be a derived stack. The \emph{secondary K-theory}  of $X$    is a kind of categorification of  algebraic K-theory introduced independently by To\"en  and  Bondal--Larsen--Lunts  \cite{BLL}.  The group of connected components of $K^{(2)}(X)$ is spanned by equivalence classes of objects in $\mathrm{ShvCat}^{\mathrm{sat}}(X)$ under the relation   
$$  [\cB] = [\cA] + [\cC] \quad \text{if there is a Verdier localization} \quad 
\cA \to \cB \to \cC.
$$
Secondary $K$-theory encodes subtle geometric and arithmetic information: 
 if $X$ is a smooth variety (in characteristic $0$), it is the recipient of highly non-trivial maps from the Grothendieck ring of varieties over $X$ and from the cohomological Brauer group  
\begin{equation}
\label{eqbrauersecondary}
K_0(\Var_X) \rightarrow  K^{(2)}_0(X) \, , \quad  
\H^ 2_{\text{\'et}}(X, \mathbb{G}_m) \to  K ^{(2)}_0(X).
\end{equation} 
Also, by \emph{additivity},   the % To\"en and Vezzosi's 
secondary Chern character factors through  secondary K-theory 
\begin{equation*}
\label{secochernchar2}
\ch^{(2)}\colon K^ {(2)}(X)  \rightarrow \cO(\cL^2 X)^{(S^1 \times S^1)}.
\end{equation*}

Let $f\colon X \to Y$ be a map of derived stacks. Under appropriate assumptions on $f$, Theorem \ref{mainB} implies a GRR theorem for the secondary Chern character.   
\begin{customthm}{E}[Theorem \ref{Secondary motivic GRR} and Example \ref{ex:smoothproper}]
\label{mainE}
Let $f\colon X \to Y$ be a morphism of perfect stacks which is representable, proper, and fiber smooth.
Then there is a 
commutative diagram of spectra  
$$ 
\xymatrix{
K^{(2)}(X) 
\ar[r]^-{\ch^{(2)}} \ar[d]_-{f_*}  
& \cO(\cL^2X)^{(S^1 \times S^1)} \ar[d]^-{\int_{\cL f} }  \\ 
K^{(2)}(Y) 
\ar[r]^-{\ch^{(2)}}   
& \cO(\cL^2Y)^{(S^1 \times S^1)}.  
}
$$
\end{customthm} 
Let now $k$ be a field of characteristic $0.$  Unlike in   (\ref{eqbrauersecondary}),  
%In contrast with (\ref{eqbrauersecondary}), 
 assume that $X$ is  a \emph{singular  variety} over $k$. The Grothendieck ring of varieties over $X$ %does not map to secondary K-theory,   
   %but rather to %the %a %different 
   %group  
   maps to a  %slight 
   variant of   secondary 
   K-theory,  
   %denoted $K^{(2)}_{BM,0}(X),$ 
   %which is  
   generated by saturated $k$-linear categories proper  over $X,$  which is  
 denoted by $K^{(2)}_{\mathrm{BM},0}(X).$ The definition of the secondary Chern character has to be recalibrated accordingly. 
% We adjust accordingly the definition of the secondary Chern character. 
It lifts to a morphism  out of 
 $K^{(2)}_{\mathrm{BM}}(X)$ which  
 takes values in the $G$-theory of $\cL X,$ instead of its $K$-theory (or Hochschild homology).  %the loop space. 
As $G$-theory is insensitive to derived thickenings, we obtain a map%Leveraging the canonical identified with the $G$-theory of $X,$  we obtain a map 
 $$
\ch^{(2)}_{\mathrm{BM}}\colon  K^{(2)}_{\mathrm{BM}}(X) \longrightarrow G(\cL X) \stackrel{\simeq} \longrightarrow G(X). 
$$ 

In  \cite{brasselet2005hirzebruch},   
Brasselet, Sch\"urmann and Yokura  introduced the \emph{motivic Chern class} %transformation 
$$
mC_*\colon K_0(\Var_X) \longrightarrow  G_0(X) \otimes \mathbb{Z}[y] \, 
$$ 
%was introduced in \cite{brasselet2005hirzebruch} 
with the purpose of unifying 
%The motivic Chern class unifies 
several different invariants of interest in singularity theory. %The map $mC_*$ 
The motivic Chern class recovers   MacPherson's total Chern class of singular varieties \cite{macpherson1974chern} and is closely related to the Cappel--Shaneson homology L-class   \cite{cappell1991stratifiable}.  We show that $\ch^{(2)}_{\mathrm{BM}}$ matches the specialization of the motivic Chern class at $y=-1.$ This follows from  an analogue of  Theorem \ref{mainE} for $\ch^{(2)}_{\mathrm{BM}}.$

\begin{customthm}{F}[Theorem \ref{motchern}]
\label{mainF}
There is a commutative diagram  of abelian groups 
\begin{equation}
\begin{gathered}
\xymatrix{
K_0(\Var_X) \ar[r]^-{mC_*} \ar[d]  & G_0(X) \otimes \mathbb{Z}[y] \ar[d]^-{y=-1} \\ 
K^{(2)}_{\mathrm{BM},0}(X) \ar[r]^-{\ch^{(2)}_{\mathrm{BM}}} & G_0(X).
}
\end{gathered}
\end{equation}
\end{customthm}
 
 \subsubsection{The de Rham realization} 
We keep the assumption that  $k$ is a field of characteristic $0$. 
 
The classical Riemann--Roch theorem  %calculates 
states that the Euler characteristic of line bundles on curves can be computed in terms of their degree and  the genus of the curve. Delicate  algebraic  information is revealed to depend only on the underlying  topology.  
All subsequent extensions of the Riemann--Roch theorem %, such as the Atiyah--Singer index theorem,  
can be viewed as finer  articulations of this principle, which persists in the categorified setting. 
%.  This feature persists in the categorified setting. 
It takes the shape of a dictionary relating  %natural identification 
%between 
the categorified Chern character 
 of categorical sheaves of geometric origin (an \emph{algebraic} invariant) and the classical  de Rham realization  (which is \emph{topological} in nature).

%All generalizations of the Riemann--Roch theorem, such as the Atiyah--Singer index theorem,  can be viewed as finer articulations of this principle. 

 Let $X$ be a smooth $k$-scheme, and let $\mathrm{Sm}_{X}$ be the category of smooth $X$-schemes.
 The de Rham realization is a functor 
 %landing in the category of  D-modules  
 \begin{equation}
 \label{derhamderham}
\dR_X\colon \Sm_X^\op \longrightarrow \cD_X\dmod
\end{equation}
which sends  a smooth map 
$f\colon Y \to X$ to  the flat vector bundle over $X$ encoding the fiberwise de Rham cohomology of $f$ equipped with the %, such that
%\begin{itemize}
%\item whose fibers calculate the cohomology of the fibers of $f,$ 
%encoding the cohomology of the fibers of $f $     % equipped with the Gauss--Manin connection. 
%equipped with 
%\item and that is 
 \emph{Gauss--Manin connection}.
%\end{itemize}

The map %smooth and proper map 
  $f\colon Y \to X$  gives rise to a sheaf of $\infty$-categories over $X$: as $X$ is 1-affine, 
  this can be encoded as the $\QCoh(X)$-module structure on $\QCoh(Y)$.
  Letting $f$ range  over $\mathrm{Sm}_{X}$, we obtain a functor 
$$
\QCoh_X\colon  \Sm_X^\op \longrightarrow \mathrm{Mod}_{\QCoh(X)}^{\mathrm{dual}}, \quad  \quad  \QCoh_X(Y \stackrel{f} \to X) = \QCoh(Y) \circlearrowleft %\in \mathrm{Mod}_{
\QCoh(X).
%}^{\mathrm{dual}}
$$

%The comparison between 
Comparing the de Rham realization and the categorified Chern character  %rests on 
requires a finer understanding of the sheaf theory of loop spaces.  
%Quasi-coherent sheaves on $\cL X$ are closely related to $\cD_X$-modules.  
%This is a categorified 
%form of the HKR equivalence, and was first suggested in \cite{BNlsc}. From our perspective  the most relevant result in this direction is the equivalence obtained in \cite{preygel2014ind} 
By a categorified 
form of the HKR equivalence \cite{BNlsc}, quasi-coherent sheaves on $\cL X$ are closely related to $\cD_X$-modules. For our purposes  the most relevant result in this direction is an 
equivalence, obtained in \cite{preygel2014ind}, between the \emph{Tate construction} of   %applied to 
 $\mathrm{IndCoh}(\cL X)$ and the $\mathbb{Z}/2$-folding of the category of D-modules 
\begin{equation}
\label{tate}
 \mathrm{Ind}(\Coh(\cL X)^{S^1}) \otimes_{k[[u]]}k((u)) \xrightarrow{\simeq} \cD_X\dmod_{\mathbb{Z}/2}.
\end{equation}
 Leveraging the equivalence (\ref{tate}) we can reformulate the 
 categorified Chern character as a functor landing  
 %we obtain a definition of the categorified Chern character landing in the 2-periodic category of D-modules 
 %in the definition of the categorified Chern character yields a symmetric monoidal functor landing 
 in the 2-periodic category of D-modules 
\begin{equation*}
\label{chernderhamcharacter}
\mathrm{Ch}^ {\mathrm{dR}}\colon \mathrm{Mod}_{\QCoh(X)}^{\mathrm{dual}}  \longrightarrow \cD_X\dmod_{\mathbb{Z}/2}.
\end{equation*}

Theorem \ref{main22}  is the main ingredient in the proof of Theorem \ref{mainG} below. In the statement of the theorem, % \ref{mainG}, 
$\dR_X$ stands for the $2$-periodization of   de Rham realization functor (\ref{derhamderham}). %statement below. 
% applied to $\mathrm{Ch}_{\mathrm{dR}}$. 
%  to $\mathrm{Ch}_{\mathrm{dR}}$ yields the following comparison result.  
%where we denote by $\rho_{\mathrm{dR}}$ the $\mathbb{Z}/2$-folding of the de Rham realization (\ref{derhamderham}). 
%The categorified GRR theorem applied to $\mathrm{Ch}_{\mathrm{dR}}$ implies the following theorem.
 
\begin{customthm}{G}[Theorem \ref{Ch-dR-comparison}]
\label{mainG}
For $X$ a smooth $k$-scheme, there is a commutative diagram  of $\infty$-categories 
\begin{equation*}
\xymatrix{
 \Sm_X^\op \ar[d]_-{\QCoh_X} \ar[r]^ -{\dR_X}  &  \mathcal{D}_{X}\dmod_{\mathbb{Z}/2}.   \\ 
 \mathrm{Mod}_{\QCoh(X)}^{\mathrm{dual}}   \ar[ur]_-{\Ch^{\mathrm{dR}}} &  
}
\end{equation*}
\end{customthm}

In fact, we will prove a generalization of Theorem~\ref{mainG} for $X$ an arbitrary derived $k$-scheme, replacing the $\infty$-category of $\cD_X$-modules by that of \emph{crystals} over $X$.

If $Y \to X$ is a smooth map,  Theorem \ref{mainG} yields an equivalence natural in $Y$ 
$$ 
\dR_X (Y \to X) \simeq \mathrm{Ch}^{\mathrm{dR}}(\QCoh(Y)). 
$$ 
This implies in particular that, up to $\mathbb{Z}/2$-folding, the Gauss--Manin connection on the cohomology of the fibers of a smooth map $f$ is of \emph{non-commutative origin}. That is, it only depends on $\QCoh(Y)$ and its 
 $\QCoh(X)$-linear structure. 

\begin{remark}
Evaluating $\Ch^{\mathrm{dR}}$ on a dualizable sheaf of categories 
over $X$ equips its relative \emph{periodic cyclic homology}  with a natural Gauss--Manin connection.  In the more restricted setting of sheaves of algebras, such a Gauss--Manin connection  %of this kind 
was introduced by Getzler in \cite{getzler1993cartan}. We believe that $\Ch^{\mathrm{dR}}$ recovers Getzler's prescription, and we plan to return to this question in a future work.
\end{remark}
%\subsection{The structure of the paper}

  \subsection*{Acknowledgments}
We thank David Ben-Zvi, Dennis Gaitsgory, J\"{o}rg Sch\"{u}rmann and Bertrand To\"{e}n for useful discussions. P.S. was supported by the NCCR SwissMAP grant of the Swiss National Science Foundation. N.S. thanks the Max Planck Institute for Mathematics, where much of this work was carried out, for excellent working conditions.

\subsection*{Conventions}

Throughout the paper, a connective $\bE_\infty$ ring spectrum $k$ is fixed. 

We use the following notation:
\begin{itemize}
\item If $\cC$ is an $(\infty,n)$-category and $m< n$, $\iota_m\cC$ denotes the underlying $(\infty,m)$-category of $\cC$, obtained by discarding non-invertible $k$-morphisms for $k>m$.

\item If $\cC$ is an $(\infty, 2)$-category, $h_2\cC$ denotes the homotopy 2-category of $\cC$.

\item If $\cC$ is a symmetric monoidal $(\infty,2)$-category, $\cC^\dual$ denotes the non-full subcategory of $\iota_1\cC$ whose objects are the 1-dualizable objects and whose morphisms are the right-adjointable morphisms. In particular, if $\cC$ is a symmetric monoidal $(\infty,1)$-category, then $\cC^\dual$ is the $\infty$-groupoid of dualizable objects in $\cC$.

\item $\Prst$ is the symmetric monoidal $\infty$-category of stable presentable $\infty$-categories and colimit-preserving functors.

\item $\bPrst$ is the symmetric monoidal $(\infty, 2)$-category of stable presentable $\infty$-categories, so that $\iota_1 \bPrst\cong \Prst$.

\item We denote by double arrows $\Rightarrow$ possibly non-invertible 2-morphisms in diagrams. In the absence of such a symbol, the diagram is assumed to commute up to an invertible 2-morphism.
\end{itemize}

\section{Preliminaries}

\subsection{Ambidexterity}

Throughout the paper by an $(\infty, 2)$-category $\cC$ we will mean a complete 2-fold Segal space. We refer to \cite[Section 6]{JFS} for a definition of a symmetric monoidal $(\infty, 2)$-category, so that $h_2\cC$ becomes a symmetric monoidal 2-category.

\begin{definition}
Suppose $\cE_1$ and $\cE_2$ are $(\infty, 2)$-categories. An \defterm{adjunction}
\[\adj{F\colon \cE_1}{\cE_2\colon G}\]
is an adjunction in the homotopy 2-category of $(\infty, 2)$-categories. It is called \defterm{ambidextrous} if the unit $\eta\colon \id \to GF$ and the counit $\epsilon\colon FG\to\id$ have right adjoints $\eta^\R\colon GF\to \id$ and $\epsilon^\R\colon \id\to FG$.
\end{definition}

\begin{remark}
We refer to \cite{RiehlVerity} for a comparison of the above definition of adjunctions for $(\infty, 1)$-categories and that given by Lurie in \cite{HTT}.
\end{remark}

\begin{remark}
It is easy to see that the transformation $\epsilon^\R\colon \id\to FG$ exhibits $F$ as right adjoint to $G$. In other words, $G$ is both left and right adjoint to $F$, which explains the terminology.
On the other hand, the notion of ambidexterity itself comes in left and right variants, and the choice made in the above definition is motivated by our main example (see Proposition~\ref{prop:rigidambidextrous}).
\end{remark}

\begin{remark}\label{rm:swallowtail}
Recall that if $\eta\colon \id\to GF$ exhibits $G$ as right adjoint to $F$, then we can find a counit $\epsilon\colon FG\to\id$ and invertible modifications
\[
\xymatrix{
F(\cM) \ar^{F(\eta_{\cM})}[r] \ar@/_2pc/@{=}[rr] & FGF(\cM) \ar^{\epsilon_{F(\cM)}}[r] \ar@{}[d]^(.2){}="a"^(.63){}="b" \ar^{\tau_1}@{=>} "a";"b" & F(\cM) \\
&&
}
\qquad
\xymatrix{
G(\cM) \ar^{\eta_{G(\cM)}}[r] \ar@/_2pc/@{=}[rr] & GFG(\cM) \ar^{G(\epsilon_\cM)}[r] \ar@{}[d]^(.2){}="a"^(.63){}="b" \ar^{\tau_2}@{=>} "a";"b" & G(\cM) \\
&&
}
\]
called \defterm{triangulators}, satisfying the \defterm{swallowtail axioms}:
\[
\xymatrix@C=0cm{
& FGFG(\cM) \ar^{\epsilon_{FG(\cM)}}[rr] \ar^{FG(\epsilon_\cM)}[dr] \ar@{}[d]^(.2){}="a"^(.9){}="b" \ar^{\tau_2}@{=>} "a";"b" && FG(\cM) \ar^{\epsilon_\cM}[dr] & \\
FG(\cM) \ar^{F(\eta_{G(\cM)})}[ur] \ar@{=}[rr] && FG(\cM) \ar^{\epsilon_\cM}[rr] && \cM
}
=
\xymatrix@C=0cm{
& FGFG(\cM) \ar^{\epsilon_{FG(\cM)}}[rr] \ar@{}[dr]^(.1){}="a"^(.5){}="b" \ar_{\tau_1}@{=>} "a";"b" && FG(\cM) \ar^{\epsilon_\cM}[dr] & \\
FG(\cM) \ar^{F(\eta_{G(\cM)})}[ur] \ar@{=}[rr] \ar@{=}[urrr] && FG(\cM) \ar^{\epsilon_\cM}[rr] && \cM
}
\]

\[
\xymatrix@C=0cm{
& GF(\cM) \ar^{\eta_{GF(\cM)}}[rr] && GFGF(\cM) \ar^{G(\epsilon_{F(\cM)})}[dr] \ar@{}[d]^(.2){}="a"^(.9){}="b" \ar^{\tau_1}@{=>} "a";"b"  & \\
\cM \ar^{\eta_\cM}[rr] \ar^{\eta_\cM}[ur] && GF(\cM) \ar^{GF(\eta_\cM)}[ur] \ar@{=}[rr] && GF(\cM)
}
=
\xymatrix@C=0cm{
& GF(\cM) \ar^{\eta_{GF(\cM)}}[rr] \ar@{=}[drrr] && GFGF(\cM) \ar^{G(\epsilon_{F(\cM)})}[dr] \ar@{}[dl]^(.1){}="a"^(.5){}="b" \ar^{\tau_2}@{=>} "a";"b" & \\
\cM \ar^{\eta_\cM}[rr] \ar^{\eta_\cM}[ur] && GF(\cM) \ar@{=}[rr] && GF(\cM)
}
\]
in the homotopy 2-categories $h_2\cE_1$ and $h_2\cE_2$. See, for example, \cite[Remark 2.2]{Gurski}.
\end{remark}

\begin{definition}
Suppose $\cE_1,\cE_2$ are symmetric monoidal $(\infty, 2)$-categories. 
A \defterm{symmetric monoidal ambidextrous adjunction}
\[\adj{F\colon \cE_1}{\cE_2\colon G}\]
is an ambidextrous adjunction where $F$ is symmetric monoidal, satisfying the \defterm{projection formula}: for any objects $\cM_1\in\cE_1$ and $\cM_2\in\cE_2$ the composite
\[\cM_1\otimes G\cM_2\xrightarrow{\eta} GF(\cM_1\otimes G\cM_2)\cong G(F\cM_1\otimes FG\cM_2)\xrightarrow{\id\otimes\epsilon} G(F\cM_1\otimes \cM_2)\]
is an equivalence.
\end{definition}

\begin{remark}
The projection formula isomorphism $\cM_1\otimes G\cM_2\xrightarrow{\sim} G(F\cM_1\otimes \cM_2)$ satisfies various compatibilities with the natural transformations $\eta,\epsilon$. These will be implicit in the diagrams we draw.
\end{remark}

Since $G$ is right adjoint to $F$, it has a natural lax monoidal structure and since it is also left adjoint to $F$, it has a natural oplax monoidal structure. Let us now work out compatibilities between the two.

The lax monoidal structure on $G$ is given by the composite
\[G\cM_1\otimes G\cM_2\xrightarrow{\sim}G(FG\cM_1\otimes \cM_2)\xrightarrow{\epsilon\otimes\id} G(\cM_1\otimes\cM_2)\]
that we denote by $\alpha$. Similarly, the oplax monoidal structure on $G$ is given by the composite
\[G(\cM_1\otimes \cM_2) \xrightarrow{\epsilon^\R\otimes\id} G(FG\cM_1\otimes \cM_2)\xrightarrow{\sim} G\cM_1\otimes G\cM_2\]
that we denote by $\alpha^\R$.

The lax compatibility with the units is expressed by the morphisms
\[1_{\cE_1}\xrightarrow{\eta} GF(1_{\cE_1})\cong G(1_{\cE_2}),\qquad G(1_{\cE_2})\cong GF(1_{\cE_1})\xrightarrow{\eta^\R} 1_{\cE_1}.\]

For a triple of objects $\cM_1,\cM_2,\cM_3\in\cE_2$ we have a 2-isomorphism
\[
\xymatrix{
G(\cM_1\otimes \cM_2)\otimes G(\cM_3) \ar^{\alpha}[d] \ar^-{\epsilon^\R\otimes\id\otimes\id}[r] & G(FG(\cM_1)\otimes\cM_2)\otimes G(\cM_3) \ar^{\sim}[r] \ar^{\alpha}[d] & G(\cM_1)\otimes G(\cM_2)\otimes G(\cM_3) \ar^{\id\otimes\alpha}[d]
\\
G(\cM_1\otimes \cM_2\otimes\cM_3) \ar^-{\epsilon^\R\otimes\id\otimes\id}[r] & G(FG(\cM_1)\otimes \cM_2\otimes \cM_3) \ar^{\sim}[r] & G(\cM_1)\otimes G(\cM_2\otimes \cM_3)
}
\]
which gives rise to a modification
\begin{equation}
(\id\otimes\alpha)\circ(\alpha^\R\otimes\id)\cong \alpha^\R\circ\alpha.
\label{eq:modification1}
\end{equation}

For a pair of objects $\cM_1,\cM_2\in\cE_2$ we have a 2-isomorphism
\[
\xymatrix{
FG(\cM_1)\otimes \cM_2 \ar^{\id\otimes\epsilon^\R}[r] \ar^{\epsilon\otimes\id}[d] \ar@/_1pc/_{\epsilon^\R}[rr] & FG\cM_1\otimes FG\cM_2 \ar^{\sim}[r] & FG(FG\cM_1\otimes \cM_2) \ar^{\epsilon\otimes\id}[d] \\
\cM_1\otimes \cM_2 \ar^{\epsilon^\R}[rr] && FG(\cM_1\otimes \cM_2)
}
\]
which gives rise to a modification
\begin{equation}
\alpha\circ (\id\otimes\epsilon^\R)\cong \epsilon^\R\circ (\epsilon\otimes\id).
\label{eq:modification2}
\end{equation}

Similarly, we have a 2-isomorphism
\[
\xymatrix{
FG(\cM_1\otimes \cM_2) \ar^{\epsilon}[r] \ar^{\epsilon^\R\otimes\id}[d] & \cM_1\otimes \cM_2 \ar^{\epsilon^\R\otimes\id}[dd] \\
FG(FG\cM_1\otimes \cM_2) \ar^{\sim}[d] \ar^{\epsilon}[dr] & \\
FG\cM_1\otimes FG\cM_2 \ar^{\id\otimes\epsilon}[r] & FG\cM_1\otimes \cM_2
}
\]
which gives rise to a modification
\begin{equation}
(\id\otimes\epsilon)\circ \alpha^\R\cong (\epsilon^\R\otimes\id)\circ \epsilon.
\label{eq:modification3}
\end{equation}

\begin{lemma}
The modifications \eqref{eq:modification2} and \eqref{eq:modification3} intertwine units and counits, i.e. we have equalities of 2-morphisms
\begin{enumerate}
\item
\[
\xymatrix{
FG\cM_1\otimes FG\cM_2 \ar^{\epsilon\otimes\id}[r] \ar@{=}[d] \ar@/^2pc/^{\epsilon\circ\alpha}[rr] & \cM_1\otimes FG\cM_2 \ar^{\id\otimes\epsilon}[r] \ar^{\epsilon^\R\otimes\id}[d] & \cM_1\otimes\cM_2 \ar^{\epsilon^\R}[d] \\
FG\cM_1\otimes FG\cM_2 \ar@{=}[r] \twocell{ur} & FG\cM_1\otimes FG\cM_2 \ar^{\alpha}[r] & FG(\cM_1\otimes \cM_2)
}
=
\xymatrix{
FG\cM_1\otimes FG\cM_2 \ar^{\epsilon\circ\alpha}[r] \ar@{=}[d] & \cM_1\otimes \cM_2 \ar^{\epsilon^\R}[d] \\
FG\cM_1\otimes FG\cM_2 \ar^{\alpha}[r] \twocell{ur} & FG(\cM_1\otimes \cM_2)
}
\]
\item
\[
\xymatrix{
FG(\cM_1\otimes\cM_2) \ar^{\epsilon}[r] \ar^{\alpha^\R}[d] & \cM_1\otimes\cM_2 \ar^{\epsilon^\R\otimes\id}[d] \ar@{=}[r] & \cM_1\otimes \cM_2 \ar@{=}[d] \\
FG\cM_1\otimes FG\cM_2 \ar^{\id\otimes\epsilon}[r] \ar@/_2pc/_{\epsilon\circ\alpha}[rr] & FG\cM_1\otimes \cM_2 \ar^-{\epsilon\otimes\id}[r] \twocell{ur} & \cM_1\otimes\cM_2
}
=
\xymatrix{
FG(\cM_1\otimes\cM_2) \ar^{\epsilon}[r] \ar^{\alpha^\R}[d] & \cM_1\otimes\cM_2 \ar@{=}[d] \\
FG\cM_1\otimes FG\cM_2 \ar^{\epsilon\circ\alpha}[r] \twocell{ur} & \cM_1\otimes\cM_2
}
\]
\end{enumerate}
in $h_2\cE_2$.
\label{lm:modification3intertwiner}
\end{lemma}

\begin{proposition}
Let $\cE_1,\cE_2$ be symmetric monoidal $(\infty,2)$-categories and $F\colon \cE_1\rightleftarrows \cE_2\colon G$ a symmetric monoidal ambidextrous adjunction. Then $G\colon \cE_2\rightarrow \cE_1$ preserves dualizable objects. Moreover, given a dualizable object $\cM\in\cE_2$ with the dual $\cM^\vee$, the dual of $G(\cM)$ is given by $G(\cM^\vee)$ with the evaluation map given by
\[G(\cM)\otimes G(\cM^\vee)\xrightarrow{\alpha} G(\cM\otimes \cM^\vee)\xrightarrow{\ev} G(1_{\cE_2})\xrightarrow{\eta^\R} 1_{\cE_1}\]
and the coevaluation map given by
\[1_{\cE_1}\xrightarrow{\eta} G(1_{\cE_2})\xrightarrow{\coev} G(\cM\otimes\cM^\vee)\xrightarrow{\alpha^\R} G(\cM)\otimes G(\cM^\vee).\]
\label{prop:ambidextrousduality}
\end{proposition}
\begin{proof}
We construct triangulators using the diagrams
\[
\xymatrix{
G(\cM) \ar^{\eta\otimes\id}[d] \ar@{=}[dr] \\
G(1_{\cE_2})\otimes G(\cM) \ar^{\coev\otimes\id}[d] \ar^{\alpha}[r] & G(\cM) \ar^{\coev\otimes\id}[d] \ar@{=}[dr] \\
G(\cM\otimes \cM^\vee) \otimes G(\cM) \ar^{\alpha^\R\otimes\id}[d] \ar^{\alpha}[r] & G(\cM\otimes\cM^\vee\otimes\cM) \ar^{\alpha^\R}[d] \ar^-{\id\otimes\ev}[r] & G(\cM) \ar^{\alpha^\R}[d] \ar@{=}[dr] \\
G(\cM)\otimes G(\cM^\vee)\otimes G(\cM) \ar^-{\id\otimes\alpha}[r] & G(\cM)\otimes G(\cM^\vee\otimes\cM) \ar^-{\id\otimes\ev}[r] & G(\cM)\otimes G(1_{\cE_2}) \ar_-{\id\otimes\eta^\R}[r] & G(\cM)
}
\]
and
\[
\xymatrix{
G(\cM^\vee) \ar^-{\id\otimes\eta}[r] \ar@{=}[dr] & G(\cM^\vee)\otimes G(1_{\cE_2}) \ar^-{\id\otimes\coev}[r] \ar^{\alpha}[d] & G(\cM^\vee)\otimes G(\cM\otimes\cM^\vee) \ar^-{\id\otimes\alpha^\R}[r] \ar^{\alpha}[d] & G(\cM^\vee)\otimes G(\cM)\otimes G(\cM^\vee) \ar^{\alpha\otimes\id}[d] \\
& G(\cM^\vee) \ar^-{\id\otimes\coev}[r] \ar@{=}[dr] & G(\cM^\vee\otimes \cM\otimes \cM^\vee) \ar^-{\alpha^\R}[r] \ar^{\ev\otimes\id}[d] & G(\cM^\vee\otimes\cM) \otimes G(\cM^\vee) \ar^{\ev\otimes\id}[d] \\
&&G(\cM^\vee) \ar^{\alpha^\R}[r] \ar@{=}[dr] & G(1_{\cE_2})\otimes G(\cM^\vee) \ar^{\eta^\R\otimes\id}[d] \\
&&& G(\cM^\vee)
}
\]
using~\eqref{eq:modification1}. Here the corner 2-isomorphisms are constructed as
\[
\xymatrix{
& G(\cM) \ar_{\eta\otimes\id}[dl] \ar_{\eta\otimes\id}[d] \ar@{=}[dr] & \\
G(1_{\cE_2})\otimes G(\cM) \ar^-{\sim}[r] & G(FG(1_{\cE_2})\otimes \cM) \ar_-{\epsilon\otimes\id}[r] & G(\cM)
}
\qquad
\xymatrix{
&& G(\cM^\vee)\otimes G(1_{\cE_1}) \ar^{\sim}[d] \\
& G(\cM^\vee) \ar^{\id\otimes\eta}[ur] \ar^{\eta}[r] \ar@{=}[dr] & GFG(\cM^\vee) \ar^{\epsilon}[d] \\
&& G(\cM^\vee)
}
\]
\end{proof}

If $\cE$ is a symmetric monoidal $(\infty,2)$-category, we denote by $\cE^{\dual}\subset \iota_1\cE$ the non-full subcategory whose objects are the 1-dualizable objects and whose morphisms are the right-adjointable morphisms.

\begin{proposition}
If $F\colon \cE_1\rightleftarrows \cE_2\colon G$ is a symmetric monoidal ambidextrous adjunction, it restricts to an adjunction
\[F\colon \cE_1^{\dual}\rightleftarrows \cE_2^{\dual}\colon G.\]
\label{prop:ambidextrousdualadjoint}
\end{proposition}

\begin{proof}
	Since $F$ and $G$ are $(\infty,2)$-functors, they preserve right-adjointable morphisms.
	The functor $F$ preserves dualizable objects since it is symmetric monoidal, and the functor $G$ preserves dualizable objects by Proposition \ref{prop:ambidextrousduality}.

It remains to check that $\eta\colon \id\rightarrow G F$ and $\epsilon\colon F G\rightarrow \id$ are right-adjointable on dualizable objects, but this holds by the assumption of ambidexterity.
\end{proof}

In the future we will also need a certain ``coherent'' version of duality.

\begin{definition}
\label{def:coherentduality}
Suppose $\cE$ is a symmetric monoidal $\infty$-category. A \defterm{coherent dual pair} is given by the following data:
\begin{itemize}
\item Objects $\cM,\cM^\vee\in\cE$.

\item 1-morphisms $\coev\colon 1\rightarrow \cM\otimes \cM^\vee$ and $\ev\colon \cM^\vee\otimes\cM\rightarrow 1$.

\item Invertible 2-morphisms
\[
\xymatrix{
\cM \ar@/_2pc/@{=}_{\id}[rr] \ar^-{\coev\otimes\id}[r] & \cM\otimes \cM^\vee\otimes \cM \ar^-{\id\otimes\ev}[r] \ar@{}[d]^(.2){}="a"^(.63){}="b" \ar^{T_1}@{=>} "a";"b" & \cM \\
&&
}
\qquad
\xymatrix{
\cM^\vee \ar@/_2pc/@{=}_{\id}[rr] \ar^-{\id\otimes\coev}[r] & \cM^\vee\otimes \cM\otimes \cM^\vee \ar^-{\ev\otimes\id}[r] \ar@{}[d]^(.2){}="a"^(.63){}="b" \ar^{T_2}@{=>} "a";"b" & \cM^\vee \\
&&
}
\]
\end{itemize}

These are required to satisfy the \defterm{swallowtail axioms}:
\begin{itemize}
\item
\[
\xymatrix@C=0.5cm@R=1cm{
& 1 \ar_{\coev}[dl] \ar^{\coev}[dr] & \\
\cM\otimes\cM^\vee \ar^{\id\otimes\coev}[dr] \ar@/_3pc/@{=}_{\id}[ddr] && \cM\otimes\cM^\vee \ar_{\coev\otimes\id}[dl] \ar@/^3pc/@{=}^{\id}[ddl] \\
& \cM\otimes\cM^\vee\otimes\cM\otimes\cM^\vee \ar^(0.7){\id\otimes\ev\otimes\id}[d] \ar@{}[dl]^(.2){}="a1"^(.63){}="b1" \ar^{T_2}@{=>} "a1";"b1" \ar@{}[dr]^(.2){}="a2"^(.63){}="b2" \ar^{T_1}@{=>} "a2";"b2" & \\
& \cM\otimes\cM^\vee &
}
=
\xymatrix@C=0.5cm@R=1cm{
& 1 \ar_{\coev}[dl] \ar^{\coev}[dr] & \\
\cM\otimes\cM^\vee \ar@{=}_{\id}[rr] && \cM\otimes\cM^\vee
}
\]

\item
\[
\xymatrix@C=0.5cm@R=1cm{
& 1 & \\
\cM^\vee\otimes\cM \ar^{\ev}[ur] \ar@/_3pc/@{=}_{\id}[ddr] && \cM^\vee\otimes\cM \ar_{\ev}[ul] \ar@/^3pc/@{=}^{\id}[ddl] \\
& \cM^\vee\otimes\cM\otimes\cM^\vee\otimes\cM \ar_{\ev\otimes\id}[ul] \ar^{\id\otimes\ev}[ur]  \ar@{}[dl]^(.2){}="a1"^(.63){}="b1" \ar^{T_2}@{=>} "a1";"b1" \ar@{}[dr]^(.2){}="a2"^(.63){}="b2" \ar^{T_1}@{=>} "a2";"b2" & \\
& \cM\otimes\cM^\vee \ar_(0.3){\id\otimes\coev\otimes\id}[u] &
}
=
\xymatrix@C=0cm{
& 1 & \\
\cM^\vee\otimes\cM \ar@{=}_{\id}[rr] \ar^{\ev}[ur] && \cM^\vee\otimes\cM \ar_{\ev}[ul]
}
\]
\end{itemize}
which are understood as equalities in the homotopy 2-category $h_2\cE$.
\end{definition}

By \cite[Theorem 2.14]{Pstragowski}, every dualizable object is part of a coherent dual pair.

\subsection{Rigidity}

In this section we define the notion of a rigid symmetric monoidal functor, generalizing the discussion of \cite[Section D]{Ga1}.

Let $\Mod_k=\Mod_k(\mathrm{Sp})\in\Prst$ be the $\infty$-category of $k$-modules. The $\infty$-category \[\Prst_k=\Mod_{\Mod_k}(\Prst)\] has an induced symmetric monoidal structure. Let us also introduce the symmetric monoidal $(\infty, 2)$-category
\[\bPrst_k = \Mod_{\Mod_k}(\bPrst).\]

\begin{definition}
A \defterm{$k$-linear symmetric monoidal $\infty$-category} is a commutative algebra object in $\Prst_k$.
\end{definition}

Let  $\cC$ be a $k$-linear symmetric monoidal $\infty$-category. 
Then $\cC$ is an object in 
\[
\mathrm{CAlg}(\Prst_k) = \mathrm{CAlg}(\iota_1\bPrst_k).\]  
We denote by $\Mod_\cC = \Mod_\cC(\Prst_k)$ 
be its $\infty$-category of modules. 
As explained in  \cite[Section 4.4]{HSS} there exists a functor  
\[
\mathrm{CAlg}(\iota_1\bPrst_k) \to \mathrm{Cat}^\otimes_{(\infty,2)} \quad  \cC \mapsto \Mod_\cC(\bPrst_k)
\]
sending $\cC$ to the monoidal $(\infty,2)$-category of $\cC$-modules in $\bPrst_k$. Further, as shown in \cite[Section 4.4]{HSS}, this construction has the property that
\[
\iota_1 \Mod_\cC(\bPrst_k) \simeq \Mod_\cC. 
\]
Thus $\Mod_\cC(\bPrst_k)$ gives a natural   $(\infty, 2)$-categorical enhancement of $\Mod_\cC$.

Given a symmetric monoidal functor $f\colon \cD\rightarrow \cC$ we get an induced adjunction
\[\adj{f^*\colon \Mod_\cD(\bPrst_k)}{\Mod_\cC(\bPrst_k)\colon f_*},\]
where the functor $f^*$ sends a $\cD$-module category $\cM$ to $\cC\otimes_\cD \cM$ and $f_*$ is the forgetful functor. The counit of the adjunction
\[\epsilon_\cM\colon f^*f_*\cM\cong \cC\otimes_\cD\cM\longrightarrow \cM\]
is given by the action map and the unit
\[\eta_\cN\colon \cN\longrightarrow f_*f^* \cN\cong \cN\otimes_\cD \cC\]
is given by $n\mapsto n\boxtimes 1_\cC$.

Let
\[\Delta\colon \cC\otimes_\cD\cC\longrightarrow \cC\]
be the tensor product functor. It can naturally be enhanced to a morphism in $\Mod_{\cC\otimes_\cD\cC}$. Since the underlying functor preserves colimits, it has a possibly discontinuous right adjoint. Moreover, the right adjoint a priori is only lax compatible with the action of $\cC\otimes_\cD\cC$.

\begin{definition}
Let $f\colon \cD\rightarrow \cC$  be a symmetric monoidal functor of $k$-linear symmetric monoidal $\infty$-categories. We call $f$ \defterm{rigid} if
\begin{enumerate}
\item the morphism $\Delta\colon \cC\otimes_\cD\cC\rightarrow \cC$ in $\Mod_{\cC\otimes_\cD\cC}(\bPrst_k)$ is right adjointable.

\item $f\colon \cD\rightarrow \cC$ in $\Mod_\cD(\bPrst_k)$ is right adjointable.
\end{enumerate}
\end{definition}

\begin{definition}\label{def: rigid}
We say a $k$-linear symmetric monoidal $\infty$-category $\cC$ is \defterm{rigid} if the unit functor $\Mod_k\rightarrow \cC$ is rigid in the above sense.
\end{definition}

Given a rigid symmetric monoidal functor $f\colon \cD\rightarrow \cC$ we denote the corresponding right adjoints by
\[f^\R\colon \cC\longrightarrow \cD,\qquad \Delta^\R\colon \cC\longrightarrow \cC\otimes_\cD\cC.\]

\begin{proposition}\label{prop:selfduality}
The functors
\[\ev\colon \cC\otimes_{\cD}\cC\stackrel{\Delta}\longrightarrow \cC\stackrel{f^\R}\longrightarrow \cD\]
and
\[\coev\colon \cD\stackrel{f}\longrightarrow \cC\stackrel{\Delta^\R}\longrightarrow \cC\otimes_\cD\cC\]
exhibit a self-duality $\cC\simeq  \cC^\vee$ in $\Mod_{\cD}$.
\end{proposition}
\begin{proof}
We have to check that the composite
\begin{align*}
\cC&\xrightarrow{\id\otimes f}  \cC\otimes_{\cD} \cC \\
&\xrightarrow{\id\otimes\Delta^\R}  \cC\otimes_{\cD}\cC\otimes_{\cD} \cC \\
&\xrightarrow{\Delta\otimes\id}  \cC\otimes_{\cD}\cC \\
&\xrightarrow{f^\R\otimes\id}  \cC
\end{align*}
is naturally isomorphic to the identity. Indeed, by the first axiom of rigidity we know that $\Delta^\R$ lies in $\Mod_{\cC\otimes_\cD \cC}$. Using the canonical algebra map 
$ \cC \to \cC\otimes_\cD \cC$ given by $x\mapsto x\boxtimes 1$, we see that $\Delta^\R$ also lies in $\Mod_\cC$. Hence we have a commutative diagram of $\infty$-categories
\[
\xymatrixcolsep{5pc}
\xymatrix{
\cC\otimes_{\cD} \cC \ar^{\id\otimes\Delta^\R}[r] \ar^{\Delta}[d] & \cC\otimes_{\cD} \cC\otimes_{\cD} \cC \ar^{\Delta\otimes\id}[d] \\
\cC \ar^{\Delta^\R}[r] & \cC\otimes_{\cD}\cC
}
\]
and the claim follows since the composite
\[\cC\xrightarrow{\id\otimes f} \cC\otimes_{\cD}\cC \xrightarrow{\Delta} \cC\]
is naturally isomorphic to the identity.
\end{proof}

Given an object $x\in\cC$, the functor $\cD\rightarrow \cC$ given by $d\mapsto f(d)\otimes x$ preserves colimits, so it admits a right adjoint $\iHom(x, -)\colon \cC\to\cD$, which is a lax $\scr D$-module functor.

\begin{proposition}
Let $\cD\rightarrow \cC$ be a rigid symmetric monoidal functor. An object $x\in\cC$ is dualizable iff $\iHom(x, -)$ preserves colimits and is $\cD$-linear.
\label{prop:compactdualizable}
\end{proposition}
\begin{proof}
Suppose $x\in\cC$ is dualizable. Then we have a sequence of equivalences
\begin{align*}
\Map_{\cC}(f(d)\otimes x, y)&\simeq \Map_{\cC}(f(d), x^\vee\otimes y) \\
&\simeq \Map_{\cD}(d, f^\R(x^\vee\otimes y))
\end{align*}
and hence $\iHom(x, y)\simeq f^\R(x^\vee\otimes y)$. Therefore, $\iHom(x, -)$ preserves colimits since $f^\R$ does, and it is $\cD$-linear since $f^\R$ is.

Conversely, suppose $\iHom(x, -)$ preserves colimits and is $\cD$-linear. Consider the functor
\[\iHom_2(x, -)\colon \cC\otimes_{\cD}\cC\rightarrow \cC\]
obtained from $\iHom(x,-)$ by extending scalars from $\scr D$ to $\scr C$.

We define the duality data as follows. Let
\[x^\vee = \iHom_2(x, \Delta^\R(1)).\]
By construction we have an evaluation morphism $x^\vee\boxtimes x\rightarrow \Delta^\R(1)$.

We define the coevaluation to be the composite
\begin{align*}
1 &\longrightarrow \iHom_2(x, 1\boxtimes x) \\
&\longrightarrow \iHom_2(x, \Delta^\R(x)) \\
&\stackrel{\sim}\longleftarrow \iHom_2(x, x\otimes_1 \Delta^\R(1)) \\
&\simeq x\otimes \iHom_2(x, \Delta^\R(1)) \\
&= x\otimes x^\vee
\end{align*}

Here we use the unit $1\boxtimes x\rightarrow \Delta^\R\circ \Delta(1\boxtimes x)\simeq \Delta^\R(x)$ in the second line, the first axiom of rigidity in the third line, and the $\cC$-linearity of $\iHom_2(x,-)$ in the fourth line.

The evaluation is defined to be the composite
\begin{align*}
x^\vee\otimes x &= \Delta(x^\vee\boxtimes x) \\
&\longrightarrow \Delta\circ \Delta^\R(1) \\
&\longrightarrow 1
\end{align*}

The duality axioms follow from the naturality of the unit and counit morphisms.
\end{proof}

\begin{corollary}\label{cor:rigid}
Let $\cC$ be a rigid symmetric monoidal $\infty$-category. Then compact objects coincide with dualizable objects in $\cC$.
\end{corollary}

Conversely, one has the following statement.

\begin{proposition}
Suppose $\cC$ is a compactly generated $k$-linear symmetric monoidal $\infty$-category. Then it is rigid iff the following conditions are satisfied:
\begin{itemize}
\item The unit object $1_\cC$ is compact.

\item Every compact object admits a dual.
\end{itemize}
\end{proposition}

Let us state several important properties of rigid symmetric monoidal functors which we will need.

\begin{proposition}
Suppose $f\colon \cD\rightarrow \cC$ is rigid. Then the adjunction
\[\adj{f^*\colon \Mod_\cD(\bPrst_k)}{\Mod_\cC(\bPrst_k)\colon f_*}\]
is a symmetric monoidal ambidextrous adjunction.
\label{prop:rigidambidextrous}
\end{proposition}
\begin{proof}
Recall that the first axiom of rigidity states that $\Delta\colon \cC\otimes_\cD\cC\rightarrow \cC$ has a right adjoint $\Delta^\R$ in $\Mod_{\cC\otimes_\cD \cC}(\bPrst_k)$. In particular, it has a right adjoint in $\Mod_\cC(\bPrst_k)$, where $\cC\rightarrow \cC\otimes_\cD\cC$ is given by $c\mapsto 1\otimes c$. We can identify $\epsilon\cong \Delta\otimes_\cC\id_{(-)}$, therefore it has a right adjoint given by $\Delta^\R\otimes_\cC\id_{(-)}$ which is obviously a strict natural transformation.

Similarly, the second axiom of rigidity states that $f\colon \cD\rightarrow \cC$ has a right adjoint $f^\R$ in $\Mod_\cD(\bPrst_k)$. But we can identify $\eta\cong f\otimes_\cD\id_{(-)}$ and hence it has a right adjoint given by $f^\R\otimes_\cD\id_{(-)}$.

Finally, given $\cM_1\in\Mod_\cD$ and $\cM_2\in\Mod_\cC$ the composite
\[\cM_1\otimes_\cD f_*\cM_2\xrightarrow{\eta} f_*f^*(\cM_1\otimes_\cD f_*\cM_2)\cong f_*(f^*\cM_1\otimes_\cC f^*f_*\cM_2)\xrightarrow{\id\otimes\epsilon} f_*(f^*\cM_1\otimes_\cC \cM_2)\]
is an equivalence iff it is so in $\Prst_k$. But its image in $\Prst_k$ is
\[\cM_1\otimes \cM_2\xrightarrow{1\otimes\id\otimes\id} \cC\otimes \cM_1\otimes \cM_2\cong \cM_1\otimes (\cC\otimes \cM_2)\xrightarrow{\id\otimes\epsilon}\cM_1\otimes\cM_2\]
which is equivalent to the identity by the unit axiom.
\end{proof}

Let us now show that rigid functors are stable under compositions and pushouts.

\begin{proposition}
Suppose $\cE\rightarrow \cD$ and $\cD\rightarrow \cC$ are rigid. Then the composition $\cE\rightarrow \cD\rightarrow \cC$ is rigid.
\label{prop:rigidcomposition}
\end{proposition}
\begin{proof}
Since $\cD\rightarrow \cC$ is right-adjointable as a $\cD$-module, it is also right-adjointable as an $\cE$-module. Therefore, the composite $\cE\rightarrow \cD\rightarrow \cC$ is right-adjointable in $\Mod_\cE(\bPrst_k)$.

The tensor product $\cC\otimes_\cE\cC\rightarrow \cC$ can be written as a composite
\[\cC\otimes_\cE\cC\rightarrow \cC\otimes_\cD\cC\rightarrow \cC.\]

The first functor can be identified with
\[\cC\otimes_\cE\cC\rightarrow \cD\otimes_{\cD\otimes_\cE\cD}(\cC\otimes_{\cE} \cC).\]
Since $\cD\otimes_\cE \cD\rightarrow \cD$ is right-adjointable as an $\cD\otimes_\cE\cD$-module, we therefore see that $\cC\otimes_\cE\cC\rightarrow \cC\otimes_\cD\cC$ is right-adjointable as a $\cC\otimes_\cE\cC$-module.

The second functor $\cC\otimes_\cD \cC\rightarrow \cC$ is right-adjointable as a $\cC\otimes_\cD\cC$-module and hence as a $\cC\otimes_\cE\cC$-module.
\end{proof}

\begin{proposition}
Suppose $f\colon \cD\rightarrow \cC$ is rigid and $\cD\rightarrow \cE$ is an arbitrary symmetric monoidal functor. Then
\[\cE\longrightarrow \cE\otimes_\cD \cC\]
is rigid.
\label{prop:rigidpushout}
\end{proposition}
\begin{proof}
The morphism $\cD\rightarrow \cC$ is right-adjointable in $\Mod_\cD$. The functor
\[\cE\otimes_\cD(-)\colon \Mod_\cD\rightarrow \Mod_\cE\]
sends it to $\cE\rightarrow \cE\otimes_\cD\cC$ which is therefore also right-adjointable.

Let $\cP = \cE\otimes_\cD \cC$. We can identify
\[\cP\otimes_\cE\cP\cong \cE\otimes_{\cE\otimes\cE}(\cP\otimes\cP)\cong \cE\otimes_\cD \cD\otimes_{\cD\otimes \cD} (\cC\otimes\cC).\]
Therefore, we can upgrade $\cE\otimes_{\cD}(-)$ to a functor
\[\Mod_{\cC\otimes_\cD\cC}\rightarrow \Mod_{\cP\otimes_\cE \cP}.\]
It sends the tensor functor $\cC\otimes_\cD\cC\rightarrow \cC$ to the tensor functor $\cP\otimes_\cE\cP\rightarrow \cP$ which is therefore right-adjointable.
\end{proof}

\subsection{Loop spaces}
\label{subloopspaces}
Since the $\infty$-category $\CAlg(\Prst_k)$ has all small colimits, it is naturally tensored over spaces.

\begin{definition}
Let $\cC$ be a $k$-linear symmetric monoidal $\infty$-category. Its \defterm{loop space} is defined to be the $k$-linear symmetric monoidal $\infty$-category
\[\cL\cC = S^1\otimes \cC\cong \cC\otimes_{\cC\otimes \cC} \cC.\]
\end{definition}

\begin{remark}
We can identify $\cL\cC$ as a $\cC$-module with $\Delta^*_\cC\Delta_{\cC, *}\cC$.
\end{remark}

The inclusion of the basepoint $\pt\rightarrow S^1$ gives rise to a symmetric monoidal functor
\[p_\cC\colon \cC\longrightarrow \cL\cC.\]

Given a functor of symmetric monoidal $\infty$-categories $f\colon \cD\rightarrow \cC$, we get an induced symmetric monoidal functor
\[\cL f\colon \cL\cD\longrightarrow \cL\cC.\]

Note that we can identify it with the composite
\begin{equation}
\cL\cD\rightarrow \cL\cD\otimes_\cD\cC\cong \cC\otimes_{\cC\otimes\cC}(\cC\otimes_\cD\cC)\rightarrow \cC\otimes_{\cC\otimes\cC}\cC\cong \cL\cC.
\label{eq:loopcomposite}
\end{equation}

We will also denote by
\[p\colon \cL\cD\otimes_\cD\cC\rightarrow \cL\cC\]
the functor induced by $\cL f$ and $p_\cC$.

\begin{proposition}\label{prop:right-adjoint}
Let $f\colon \cD\rightarrow \cC$ be a symmetric monoidal functor.

\begin{enumerate}
\item If $f\colon \cD\rightarrow \cC$ is rigid, $\cL f\colon\cL\cD\rightarrow \cL\cC$ is right-adjointable in $\Mod_{\cL\cD}(\bPrst_k)$

\item If $f\colon \cD\rightarrow \cC$ and $\Delta\colon\cC\otimes_\cD\cC\rightarrow \cC$ are rigid, so is $\cL f\colon\cL\cD\rightarrow \cL\cC$.
\end{enumerate}
\label{prop:Lfrigid}
\end{proposition}
\begin{proof}
Suppose $\cD\rightarrow \cC$ is rigid. Therefore, $\cL\cD\rightarrow \cL\cD\otimes_\cD\cC$ is rigid by Proposition \ref{prop:rigidpushout}. Similarly, since $\cC\otimes_\cD\cC\rightarrow \cC$ is right-adjointable in $\Mod_{\cC\otimes_\cD\cC}(\bPrst_k)$,
\[\cL\cD\otimes_\cD\cC\cong \cC\otimes_{\cC\otimes\cC}(\cC\otimes_\cD\cC)\rightarrow \cC\otimes_{\cC\otimes\cC} \cC\]
is right-adjointable in $\Mod_{\cL\cD\otimes_\cD\cC}(\bPrst_k)$. Therefore, the composite \eqref{eq:loopcomposite} is right-adjointable in $\Mod_{\cL\cD}(\bPrst_k)$.

Now suppose in addition $\cC\otimes_\cD\cC\rightarrow \cC$ is rigid. Then $\cC\otimes_{\cC\otimes\cC}(\cC\otimes_\cD\cC)\rightarrow \cC\otimes_{\cC\otimes\cC} \cC$ is rigid by Proposition \ref{prop:rigidpushout}. Therefore, by Proposition \ref{prop:rigidcomposition} the functor $\cL\cD\rightarrow \cL\cC$ is rigid as well.
\end{proof}

Let $\lambda\colon \id\rightarrow (\cL f)_*(\cL f)^*$ be the unit of the adjunction $\adj{(\cL f)^*\colon \Mod_{\cL\cD}}{\Mod_{\cL\cC}\colon (\cL f)_*}$

\begin{proposition}
Suppose $f\colon \cD\rightarrow \cC$ is a rigid symmetric monoidal functor. Then $\lambda$ admits a right adjoint $\lambda^\R$ which is a strict natural transformation.
\label{prop:lambdarightadjoint}
\end{proposition}
\begin{proof}
Note that $\lambda\cong \cL f\otimes_{\cL\cD} \id_{(-)}$.
Since $f$ is rigid, $\cL f$ is right-adjointable as an $\cL\cD$-module functor, by Proposition \ref{prop:Lfrigid}. Therefore, the transformation $\lambda$ has a strictly natural right adjoint given by $\lambda^\R\cong (\cL f)^\R\otimes_{\cL\cD} \id_{(-)}$.
\end{proof}

\subsection{Smooth and proper modules}

In this section we introduce further finiteness conditions on functors and modules relevant to the uncategorified GRR theorem. The notions of smooth, proper and saturated category go back to works of Bondal, Kapranov and others in the setting of classical triangulated categories, see for instance \cite{bondal1990representable}. We reefer the reader to  \cite[Section 4.6.4]{HA} for a discussion of closely related concepts in the $\infty$-categorical setting. 

\begin{definition}\label{def-proper-smooth}
	Let $f\colon \cD\to\cC$ be a $k$-linear symmetric monoidal functor. We say that $f$ is:
	\begin{enumerate}
		\item \defterm{proper} if $f\colon \cD\rightarrow \cC$ is rigid and $f^\R$ admits a right adjoint in $\Mod_\cD(\bPrst_k)$.
		\item \defterm{smooth} if $\Delta\colon \cC\otimes_\cD\cC\rightarrow \cC$ is rigid and $\Delta^\R$ admits a right adjoint in $\Mod_{ \cC \otimes_{\cD} \cC}(\bPrst_k)$.
	\end{enumerate}
\end{definition}

\begin{definition}\label{def-saturated}
Let $\cC$ be a $k$-linear symmetric monoidal $\infty$-category. We say that a dualizable $\cC$-module $\cM \in \Mod^{\dual}_\cC$ is 
	\begin{enumerate}
		\item \defterm{proper} if the evaluation  map $\ev_\cM$ admits a right adjoint in  $\Mod_\cC(\bPrst_k)$;
		\item \defterm{smooth} if the coevaluation map $\coev_{\cM}$ admits a right adjoint in  $\Mod_\cC(\bPrst_k)$;
		\item \defterm{saturated} if $\cM$ is smooth and proper. 
	\end{enumerate}
\end{definition}

\begin{lemma}
	Let $\cM$ be a $\cC$-module. Then $\cM$ is saturated if and only if it is fully dualizable in the symmetric monoidal $(\infty,2)$-category $\Mod_\cC(\bPrst_k)$.
\end{lemma}
\begin{proof}
Indeed, by \cite[Theorem 3.9]{Pstragowski}, $\cM$ is fully dualizable if and only if it is dualizable with right-adjointable evaluation and coevaluation maps.
\end{proof}

\begin{lemma}\label{lem:smoothproper}
	Let $f\colon \cD\to\cC$ be a rigid morphism of $k$-linear symmetric monoidal $\infty$-categories.
	\begin{enumerate}
		\item If $f$ is proper, then $f_*\colon \Mod_\cC \to \Mod_\cD$ preserves proper $\infty$-categories.
		\item If $f$ is smooth, then $f_*\colon \Mod_\cC \to \Mod_\cD$ preserves smooth $\infty$-categories.
		\item If $f$ is smooth and proper, then $f_*\colon \Mod_\cC \to \Mod_\cD$ preserves saturated $\infty$-categories.
	\end{enumerate}
\end{lemma}
\begin{proof}
	This follows immediately from the definitions and Proposition~\ref{prop:ambidextrousduality}.
\end{proof}

\begin{lemma}
Suppose $f\colon \cD\rightarrow \cC$ is a symmetric monoidal functor which is smooth and proper. Then $\cL f\colon\cL\cD\rightarrow \cL\cC$ is proper.
\label{lem:proper}
\end{lemma}
\begin{proof}
Since $\cD\rightarrow \cC$ and $\cC\otimes_\cD\cC\rightarrow \cC$ are rigid, $\cL f$ is rigid by Proposition \ref{prop:Lfrigid}.

Decompose $\cL f$ using \eqref{eq:loopcomposite} as
\[\cL\cD\rightarrow \cL\cD\otimes_\cD\cC\cong \cC\otimes_{\cC\otimes\cC}(\cC\otimes_\cD\cC)\rightarrow \cC\otimes_{\cC\otimes\cC} \cC.\]
Since $\cD\rightarrow \cC$ is twice right adjointable in $\Mod_{\cD}(\bPrst_k)$ by properness of $f$, $\cL\cD\rightarrow \cL\cD\otimes_\cD\cC$ is twice right adjointable in $\Mod_{\cL\cD}(\bPrst_k)$.

Similarly, since $\cC\otimes_\cD\cC\rightarrow \cC$ is twice right adjointable in $\Mod_{\cC\otimes_\cD\cC}(\bPrst_k)$ by smoothness of $f$,
\[\cC\otimes_{\cC\otimes\cC}(\cC\otimes_\cD\cC)\rightarrow \cC\otimes_{\cC\otimes\cC}\cC\]
is also twice right adjointable in $\Mod_{\cL\cD\otimes_\cD\cC}(\bPrst_k)$.
\end{proof}

\subsection{Geometric setting}

Recall that a derived prestack is a functor from the $\infty$-category of connective $\bE_\infty$ algebras over $k$ to the $\infty$-category of spaces. Our main source of rigid symmetric monoidal functors is given by considering passable morphisms of derived prestacks.

If $f\colon X\rightarrow Y$ is a morphism of prestacks, then the pullback $f^*\colon \QCoh(Y)\rightarrow \QCoh(X)$ is a symmetric monoidal functor.

\begin{definition}
A morphism of prestacks $f\colon X\rightarrow Y$ is \defterm{passable} if the following conditions are satisfied:
\begin{enumerate}
\item The diagonal $X\rightarrow X\times_Y X$ is quasi-affine.

\item The pullback $f^*\colon \QCoh(Y)\rightarrow \QCoh(X)$ admits a right adjoint $f_*\colon \QCoh(X)\rightarrow \QCoh(Y)$ in $\Mod_{\QCoh(Y)}(\bPrst_k)$.

\item The $\infty$-category $\QCoh(X)$ is dualizable as a $\QCoh(Y)$-module.
\end{enumerate}
\label{def:passable}
\end{definition}

\begin{proposition}
\label{passableimpliesrigid}
Suppose $f\colon X\rightarrow Y$ is a passable morphism of prestacks and $Y$ has a quasi-affine diagonal. Then the pullback functor $f^*\colon \QCoh(Y)\rightarrow \QCoh(X)$ is a rigid symmetric monoidal functor.
\label{prop:passablerigid}
\end{proposition}
\begin{proof}
The proof is similar to the proof of \cite[Proposition 5.1.7]{Ga1}.

First of all, the second axiom of passability for $f\colon X\rightarrow Y$ is exactly the second axiom of rigidity for $f^*\colon \QCoh(Y)\rightarrow \QCoh(X)$.

Since $\QCoh(X)$ is dualizable as a $\QCoh(Y)$-module category, the functor $\QCoh(X)\otimes_{\QCoh(Y)}(-)$ preserves limits. Therefore, the functor
\begin{align*}
\QCoh(X)\otimes_{\QCoh(Y)}\QCoh(X) &= \left(\lim_{S\rightarrow X} \QCoh(S)\right)\otimes_{\QCoh(Y)} \QCoh(X) \\
&\rightarrow \lim_{S\rightarrow X} \left(\QCoh(S)\otimes_{\QCoh(Y)} \QCoh(X)\right)
\end{align*}
is an equivalence where the limit is over affine derived schemes $S$ with a morphism to $X$. Since the diagonal of $Y$ is quasi-affine, the morphism $S\rightarrow X\rightarrow Y$ is quasi-affine as well. But then by \cite[Proposition B.1.3]{Ga1}, the functor $\QCoh(S)\otimes_{\QCoh(Y)} \QCoh(X)\rightarrow \QCoh(S\times_Y X)$ is an equivalence. Since
\[\QCoh(X\times_Y X)\simeq \lim_{S\rightarrow X} \QCoh(S\times_Y X),\]
this proves that the natural functor $\QCoh(X)\otimes_{\QCoh(Y)} \QCoh(X)\rightarrow \QCoh(X\times_Y X)$ is an equivalence.

Since the diagonal $X\rightarrow X\times_Y X$ is quasi-compact and representable, the functor
\[\Delta_*\colon \QCoh(X)\rightarrow \QCoh(X)\otimes_{\QCoh(Y)} \QCoh(X)\simeq \QCoh(X\times_Y X)\]
is continuous and satisfies the projection formula. This immediately implies the first axiom of rigidity of $f^*\colon \QCoh(Y)\rightarrow \QCoh(X)$.
\end{proof}

\begin{comment}
\begin{proposition}
Suppose $f\colon X\rightarrow Y$ is a morphism of perfect stacks which is representable and proper.

\begin{enumerate}
\item If $f$ is of finite Tor-amplitude and locally almost of finite presentation, then $f^*$ is proper.

\item If $f$ is smooth, then $f^*$ is smooth and proper.
\end{enumerate}
\label{prop:smoothproper}
\end{proposition}
\begin{proof}
Since $f$ is proper, it is quasi-compact. 
\end{proof}
\end{comment}

\begin{example} \label{ex:smoothproper}
	Let $f\colon X\to Y$ be a morphism between weakly perfect stacks, in the sense of \cite[Definition 9.4.3.3]{SAG}.
	\begin{enumerate}
		\item If $f$ is representable by quasi-compact quasi-separated spectral algebraic spaces, then $f$ is passable. Condition (2) follows from \cite[Corollary 6.3.4.3]{SAG}, and condition (3) follows from \cite[Corollary 9.4.3.6]{SAG}.
		In particular, $f^*$ is rigid.
		\item If $f$ is representable, proper, of finite Tor-amplitude, and locally almost of finite presentation, then $f^*$ is proper.
		This follows from (1) and \cite[Proposition 6.4.2.1 and Corollary 6.4.2.7]{SAG}.
	\item If $f$ is representable, proper, and fiber smooth, then $f^*$ is smooth and proper. This follows from (2) as both $f$ and its diagonal have finite Tor-amplitude \cite[Lemma 11.3.5.2]{SAG}.
	\end{enumerate}
\end{example}

\begin{example}
Let us mention an example of a stack that is \emph{not} passable. Let $k$ be a field of positive characteristic, and let  $\mathbb{G}_a$ be the additive group over $\mathrm{Spec}(k)$. Then the classifying stack $X = B\mathbb{G}_a$ is not passable. Indeed one of the necessary conditions for being  passable is that the structure sheaf $\mathcal{O}_X$ is compact. However by \cite[Proposition 3.1]{hall2019one} the only compact object in the derived category of quasi-coherent sheaves on $X$ is $0$. \end{example}

\section{The Chern character}
\subsection{Traces} 
\label{traces} 
In this section we recall the definition of the trace functor given in Section 2.2 of \cite{HSS} and state some of its properties.

 Let $\cC$ be a symmetric monoidal 
$(\infty,n)$-category. In \cite[Section 2.3]{HSS} we define a 
symmetric monoidal $(\infty, n-1)$-category  $\Aut(\cC),$ which carries  a canonical  
$S^1$-action.    
 The objects and 1-morphisms of $\Aut(\cC)$ can be described as follows: 
\begin{enumerate}
\item An object of  $\Aut(\cC)$ is a pair 
$(A, a),$ where $A$ is a dualizable object of $\cC,$   and $a$ is an  automorphism of $A.$  
\item 
A $1$-morphism $(A, a) \to (B,b)$ in $\Aut(\cC)$ is a commutative diagram %pair $(\phi, \alpha)$
\begin{tikzmath}
		\diagram{A & B \\ A & B  \\   };
		\arrows (11-) edge node[above]{$\phi$} (-12) (11) edge node[left]{$a$} (21) (12) edge node[right]{$b$} (22) (21-) edge node[below]{$\phi$} (-22)
		(11) edge[font=\normalsize,draw=none] node[rotate=45]{$\Longrightarrow$} (22)
		(11) edge[draw=none] node[pos=.3]{$\alpha$} (22);
	\end{tikzmath}
	where $\phi: A \to B$ is right-dualizable, and $\alpha$ is an invertible 2-cell. 
%	\item 
%A $1$-morphism $(A, a) \to (B,b)$ in $\End(\cC)$ is a pair $(\phi, \alpha)$
%\begin{tikzmath}
%		\diagram{A & B \\ A & B  \\   };
%		\arrows (11-) edge node[above]{$\phi$} (-12) (11) edge node[left]{$a$} (21) (12) edge node[right]{$b$} (22) (21-) edge node[below]{$\phi$} (-22)
%		(11) edge[font=\normalsize,draw=none] node[rotate=45]{$\Longrightarrow$} (22)
%		(11) edge[draw=none] node[pos=.3]{$\alpha$} (22);
%	\end{tikzmath}
%	where $\phi: A \to B$ is right-dualizable, and 
%	$\alpha$ is a 2-cell. 
 \end{enumerate}
 
\begin{example}
Let $\cC$ be a $k$-linear symmetric monoidal $\infty$-category. Then
\begin{enumerate}
\item $\Aut(\cC)\cong \Fun(S^1, \cC^\dual)$.

\item $\Aut(\Mod_\cC)\cong \Fun(S^1, \Mod_\cC^{\dual})$.
\end{enumerate}
The $S^1$-action on $\Aut(\cC)$ and $\Aut(\Mod_\cC)$  is induced by the action of $S^1$ on itself. 
\end{example}
We also define a \emph{trace functor} 
$$
\Tr\colon \Aut(\cC) \longrightarrow \Omega\cC
$$
which is symmetric monoidal and natural in $\cC$ \cite[Definitions 2.9 and 2.11]{HSS}.
  
\begin{proposition}[\cite{HSS} Lemma 2.4]
\label{prophsshsss}
Let $\cC$ be a symmetric monoidal $(\infty,n)$-category. Then: 
\begin{enumerate}
\item The functor $\Tr$ sends an object $(A,a)$ in $\Aut(\cC)$ to the composite
$$
1_\cC \xrightarrow{\coev_A} A \otimes A^{\vee}  \xrightarrow{a \otimes \id } 
A \otimes A^{\vee}  \xrightarrow{\ev_A} 1_\cC  
$$

\item The functor $\Tr$ sends a 
 $1$-morphism in $\Aut(\cC)$ 
 $$
(\phi, \alpha) :  \,  (A, a)  \longrightarrow (B,b) 
$$  to the 2-cell in $\Omega\cC$ given by the composite 
%the image of a $1$-morphism $(\phi,\alpha)\colon (X,f)\to (Y,g)$ is depicted by the diagram
\begin{tikzequation}\label{eqn:2trace}
	\def\rowsep{1em}
	\diagram{
	& A\tens A^\vee & & & A\tens A^\vee & \\
	1_\cC & & & & & 1_\cC %\rlap, 
	\\
	& B\tens B^\vee & & & B\tens B^\vee & \\
	};
	\arrows (21) edge (12) edge (32) (12-) edge node[above]{$a\tens \id$} (-15) (32-) edge node[below]{$b\tens \id$} (-35) (15) edge (26) (35) edge (26) (12) edge node[right]{$\phi\tens\phi^{\R\vee}$} (32) (15) edge node[left]{$\phi\tens\phi^{\R\vee}$} (35)
	(21) edge[font=\normalsize,draw=none] node[rotate=45]{$\Longleftarrow$} node[pos=.13,rotate=45]{$\Longleftarrow$} node[pos=.87,rotate=45]{$\Longleftarrow$} (26)
	(21) edge[draw=none] node[above left]{$\alpha\tens \id$} (26)
	;
\end{tikzequation}
where the triangular $2$-cells on the left and on the right %of the diagram 
are given by  
$$
(\phi\tens\phi^{\R\vee})\coev_A = (\phi\phi^\R\tens\id)\coev_B \stackrel\epsilon\to \coev_B, \quad \text{and} \quad 	\ev_A \stackrel\eta\to \ev_A(\phi^\R\phi\tens\id) = \ev_B(\phi\tens\phi^{\R\vee}).
$$
%\begin{gather*}
%	(\phi\tens\phi^{r\vee})\coev_A = (\phi\phi^r\tens\id)\coev_B \stackrel\epsilon\to \coev_B,\\
%	\ev_A \stackrel\eta\to \ev_A(\phi^r\phi\tens\id) = \ev_B(\phi\tens\phi^{r\vee}).
%\end{gather*}
\end{enumerate}
\end{proposition}
 
 % and we will restrict $\Tr$ to it. 
\begin{proposition}[\cite{HSS} Theorem 2.14]
\label{prop:S1trace}
The symmetric monoidal trace functor
$$
\Tr\colon \Aut(\cC) \longrightarrow \Omega\cC
$$
 is  $S^1$-invariant   
with respect to the canonical $S^1$-action on 
$\Aut(\cC)$  and the trivial $S^1$-action on $\Omega\cC.$  
\end{proposition}

\begin{proposition}
Let $\cC$ be a symmetric monoidal $(\infty, n)$-category. The composite
\[\Aut(\Omega\cC)\cong \Omega\Aut(\cC)\xrightarrow{\Omega\Tr}\Omega^2\cC\]
is equivalent to the trace functor $\Tr\colon \Aut(\Omega\cC)\rightarrow \Omega(\Omega\cC)$ as an $S^1$-equivariant symmetric monoidal functor.
\label{prop:tracedecategorified}
\end{proposition}

\begin{proof}
Consider the diagram
\[
\xymatrix{
\iota_{n-2}\Fun_\tens^\oplax(\Fr^\rig(S^1), \Omega\cC) \ar^{\sim}[r] \ar^{\Omega}[d] & \Omega\iota_{n-1}\Fun_\tens^\oplax(\Fr^\rig(S^1), \cC) \ar^{\Omega}[d] \\
\Fun_\tens^\oplax(\Omega\Fr^\rig(S^1), \Omega^2\cC) \ar^-{\sim}[r] \ar[d] & \Omega\iota_{n-1}\Fun_\tens^\oplax(\Omega\Fr^\rig(S^1), \Omega\cC) \ar[d] \\
\Omega^2\cC \ar@{=}[r] & \Omega^2\cC,
}
\]
where the lower vertical maps are evaluation at the trace $\Tr(u)\in\Omega\Fr^\rig(S^1)$ of the universal automorphism $u$. By construction, the vertical composites are the respective trace functors, whose $S^1$-equivariance is induced by an $S^1$-invariant refinement of $\Tr(u)$ (see the proof of \cite[Theorem 2.14]{HSS}). This is therefore a commutative diagram of $S^1$-equivariant symmetric monoidal functors, which proves the claim.
\end{proof}

\subsection{Chern character}
\label{sect:uncatchern}

Let $\cC$ be a $k$-linear symmetric monoidal 
$\infty$-category. The identity functor
\[
S^1 \otimes \cC \simeq \cL \cC   \to  \cL \cC
\]
induces by adjunction a symmetric monoidal functor $\cC\to\Fun(S^1, \cL\cC)$ and a functor $S^1\to\Fun^\otimes(\cC, \cL\cC)$. 
Choosing once and for all a basepoint $p\colon \pt\to S^1$, the latter is equivalent to the following data:
\begin{itemize}
\item a symmetric monoidal functor $p_\cC\colon \cC \to \cL \cC$, which is induced by the inclusion of the basepoint;
\item a natural equivalence of symmetric monoidal functors
$ 
\mon\colon p_\cC \xrightarrow{\simeq}  p_\cC.
$
\end{itemize}
\begin{definition}
We call  $\mon$  the 
\defterm{monodromy automorphism}.
\end{definition}

\begin{definition}
The \defterm{Chern character} of $\cC$ is the composite
$$
\ch\colon \cC^\dual \longrightarrow 
\Fun(S^1, (\cL \cC)^\dual) \simeq 
\Aut(\cL \cC) \stackrel{\Tr} \longrightarrow 
\Omega \cL \cC.
$$
 \end{definition}
 
 By Proposition~\ref{prop:S1trace}, $\ch$ is $S^1$-equivariant and hence factors through the fixed points of the $S^1$-action on $\Omega \cL \cC$:
 \[
 \ch\colon \cC^\dual \longrightarrow (\Omega\cL\cC)^{S^1}.
 \]
 
We will now relate the monodromy automorphism to certain coCartesian diagrams. Let $p\colon \pt\rightarrow S^1$ be the basepoint. We have $\Aut(p) = \Z$. Consider morphisms of spaces
\begin{equation}
\xymatrix{
\pt\coprod \pt \ar[r] \ar[d] & \pt \ar[d] \\
\pt \ar[r] & S^1.
}
\label{eq:2ptcircle}
\end{equation}
The two composites $\pt\coprod\pt\rightarrow \pt\rightarrow S^1$ are equal, so the space of two-cells completing this diagram to a square is given by $\Aut(p)\times \Aut(p)\cong \Z\times\Z$. For every pair of integers $(n, m)$, we thus obtain a commutative square of the form above. %The $\infty$-category $\CAlg(\Prst_k)$ of $k$-linear symmetric monoidal $\infty$-categories admits small colimits, so it is naturally tensored over spaces. In particular, 
Given $\cC\in\CAlg(\Prst_k)$, the above square of spaces gives rise to a square
\[
\xymatrix{
\cC\otimes\cC \ar^{\Delta}[r] \ar^{\Delta}[d] & \cC \ar^{p}[d] \\
\cC \ar^{p}[r] & \cL\cC
}
\]
in $\CAlg(\Prst_k)$.

\begin{proposition}
Let $\cC\in\CAlg(\Prst_k)$ be a $k$-linear symmetric monoidal $\infty$-category. The $(1, 0)$ square
\[
\xymatrix{
\cC\otimes\cC \ar^{\Delta}[r] \ar^{\Delta}[d] & \cC \ar^{p}[d] \\
\cC \ar^{p}[r] & \cL\cC
}
\]
is coCartesian.
\label{prop:10cocartesian}
\end{proposition}
\begin{proof}
It is enough to prove the claim in the $\infty$-category $\cS$ of spaces.

We have a homotopy pushout diagram of groupoids
\begin{center}
\begin{tikzpicture}
  \filldraw (0, 0) circle (2pt);

  \filldraw (-0.4, 2) circle (2pt);
  \filldraw (0.4, 2) circle (2pt);
  
  \filldraw (2.6, 2) circle (2pt);

  \filldraw (2.2, 0) circle (2pt);
  \filldraw (3.0, 0) circle (2pt);
  \draw (2.6, 0) circle (0.4);

  \draw[->] (0.4, 0) -- (1.8, 0);
  \draw[->] (0, 1.8) -- (0, 0.6);
  \draw[->] (0.6, 2) -- (2, 2);
  \draw[->] (2.6, 1.8) -- (2.6, 0.6);
\end{tikzpicture}
\end{center}
which is a coCartesian square in $\cS$.

We can complete it to a diagram
\begin{center}
\begin{tikzpicture}
  \filldraw (0, 0) circle (2pt);

  \filldraw (-0.4, 2) circle (2pt) node[above] {$a$};
  \filldraw (0.4, 2) circle (2pt) node[above] {$b$};

  \filldraw (2.6, 2) circle (2pt);

  \filldraw (2.2, 0) circle (2pt);
  \filldraw (3.0, 0) circle (2pt);
  \draw (2.6, 0) circle (0.4);
  \draw (2.8, 0.5) node {$a$};
  \draw (2.8, -0.5) node {$b$};

  \filldraw (2.6, -2.4) circle (2pt);
  \draw (2.6, -2) circle (0.4);
  
  \draw[->] (0.4, 0) -- (1.4, 0);
  \draw[->] (0.2, -0.2) -- (2, -2);
  \draw[->] (2.6, -0.6) -- (2.6, -1.4) node[midway, right] {$\sim$};
  \draw[->] (0, 1.8) -- (0, 0.6);
  \draw[->] (0.6, 2) -- (2, 2);
  \draw[->] (2.6, 1.4) -- (2.6, 0.6);
\end{tikzpicture}
\end{center}
where the isomorphism at the bottom sends the morphism $a$ to the nontrivial automorphism of the point and the morphism $b$ to the identity. The total diagram is exactly a diagram \eqref{eq:2ptcircle} of type $(1, 0)$ which proves the claim.
\end{proof}

\begin{remark}
\label{monwhisk}
Note that whiskering the $(1, 0)$ square
\[
\xymatrix{
\cC\otimes\cC \ar^{\Delta}[r] \ar^{\Delta}[d] & \cC \ar^{p}[d] \\
\cC \ar^{p}[r] & \cL\cC
}
\]
along $\cC\rightarrow \cC\otimes \cC$ given by $x\mapsto x\boxtimes 1$ we obtain the monodromy automorphism of $p_\cC\colon \cC\rightarrow \cL\cC$. Whiskering the same $(1, 0)$ square along $x\mapsto 1\boxtimes x$ we obtain the identity. 
\end{remark}
\begin{remark}
One may similarly show that the $(\pm1, 0)$ and $(0, \pm1)$ squares are coCartesian. However, the $(0, 0)$ square is not coCartesian.
\end{remark}

\subsection{Categorified Chern character}
%\subsubsection{Categorified Monodromy}
\label{ss: Cat-Mon}
\label{monchch}
Let $\cC$ be a $k$-linear symmetric monoidal $\infty$-category, 
and let 
\begin{equation}
\label{moneqmon}
\xymatrix{
& \twocelllabel{dd}{\mathrm{mon}} & \\
\cC \ar@/^1pc/^{p_\cC}[rr] \ar@/_1pc/_{p_\cC}[rr] && \cL\cC \\
&&
}
\end{equation}
be the monodromy automorphism. Applying the functor 
$\Mod(-)$ to (\ref{moneqmon}) gives a natural equivalence 
\begin{equation}
\xymatrix{
& \twocelllabel{dd}{\mathrm{Mon}} & \\
\Mod_\cC \ar@/^1pc/^{p_\cC^*}[rr] \ar@/_1pc/_{p_\cC^*}[rr] && \Mod_{\cL\cC} \\
&&
}
\end{equation}

\begin{definition}
We call  $\mathrm{Mon}$  the 
\defterm{categorified monodromy automorphism}. \end{definition}
\begin{remark}
It follows from Remark \ref{monwhisk} that the categorified monodromy automorphism  can  also be obtained by whiskering the $(1, 0)$ square
\begin{equation}
\begin{gathered}
\label{sqwhiskcat}
\xymatrix{
\Mod_{\cC\otimes\cC} \ar^{\Delta^*}[r] \ar^{\Delta^*}[d] &\Mod_\cC \ar^{p^*}[d] \\
\Mod_ \cC \ar^{p^*}[r] & \Mod_{\cL\cC}
}
\end{gathered}
\end{equation}
along $\pi_1^* : \Mod_\cC\rightarrow \Mod_ {\cC\otimes \cC}.$ In particular, evaluating (\ref{sqwhiskcat}) on $\cM \boxtimes \cN\in \Mod_{\cC \otimes \cC}$ induces the automorphism 
$$ \Mon_{\cM} \otimes \id: p^*\cM  \otimes_{\cL \cC} p^* \cN  \to p^* \cM  \otimes_{\cL \cC} p^* \cN .$$
%In particular 
%given  $\cM \boxtimes \cN\in \Mod_{\cC \otimes \cC}$ %with $\cM , \cN \in \Mod \cC$, 
% the $(1,0)$ pushout induces by the previous section the automorphism 
%$$ \Mon_{\cM} \otimes \id: p^*(\cM) \otimes_{\cL \cC} p^*(\cN) \to p^*(\cM) \otimes_{\cL \cC} p^*(\cN).$$
\end{remark}

Restricting to dualizable objects we get an  equivalence
\[
\xymatrix{
& \twocelllabel{dd}{\mathrm{Mon}} & \\
\Mod_\cC^\dual \ar@/^1pc/^{p_\cC^*}[rr] \ar@/_1pc/_{p_\cC^*}[rr] && \Mod_{\cL\cC}^\dual \\
&&
}
\]
and this determines a map
$$
B\mathbb{Z} \simeq S^1 \to \Fun^\otimes(\Mod_{\cC}^\mathrm{dual},\Mod_{\cL\cC}^\mathrm{dual}). 
$$
By adjunction we obtain a symmetric monoidal functor
$$
\Mod^{\mathrm{dual}}_{\cC}   
\to \Fun(S^1, \Mod^{\mathrm{dual}}_{\cL \cC}). 
$$ 
\begin{definition}
The \defterm{categorified    Chern character} is   the composite  
%\[\ch\colon \iota_0\Perf(X)\rightarrow \Aut\Perf(\scr L X)\rightarrow \cO(\scr L X).\]  
$$
\Ch: \Mod^{\mathrm{dual}}_{\cC}  \longrightarrow \Fun(S^1, \Mod^{\mathrm{dual}}_{\cL \cC}) \cong
\Aut(\Mod^ \dual _{\cL \cC}) \xrightarrow{\Tr} \cL \cC.
$$  
\end{definition}

 By Proposition~\ref{prop:S1trace}, $\Ch$ is $S^1$-equivariant and hence factors through the fixed points of the $S^1$-action on $\cL \cC$, i.e., the $\infty$-category of $S^1$-equivariant objects of $\cL\cC$:
 \[
\Ch: \Mod^{\mathrm{dual}}_{\cC} \longrightarrow (\cL\cC)^{S^1}.
 \]

\subsection{Decategorifying the Chern character}

Note that the categorified Chern character being symmetric monoidal induces a map of spaces
\[\Ch\colon \Omega\Mod_\cC^{\dual}\longrightarrow \Omega(\cL\cC)^{S^1}\simeq (\Omega\cL\cC)^{S^1}.\]
We will show that this map coincides with the uncategorified Chern character.

\begin{lemma}
The composite
\[\CAlg(\Prst_k)\xrightarrow{\Mod}\CAlg(\Cat_{(\infty, 2)})\xrightarrow{\Omega} \CAlg(\Cat_{(\infty, 1)})\]
is equivalent to the forgetful functor.
\label{lm:moddecategorified}
\end{lemma}

\begin{proof}
We have $\Omega\Mod_\cC\simeq \Fun_{\Mod_{\cC}}(\cC,\cC)\simeq \cC$.
\end{proof}

\begin{theorem}
Let $\cC$ be a $k$-linear symmetric monoidal $\infty$-category. The composite
\[\cC^\dual \cong \Omega\Mod_\cC^{\dual}\xrightarrow{\Ch} \Omega\cL\cC\]
is equivalent to the Chern character
\[\ch\colon \cC^\dual \longrightarrow \Omega\cL\cC\]
as $S^1$-equivariant $\bE_\infty$ maps.
\label{thm:cherndecategorified}
\end{theorem}
\begin{proof}
The composite
\[S^1\xrightarrow{\mon}\Fun^{\otimes}(\cC, \cL\cC)\xrightarrow{\Mod}\Fun^{\otimes}(\Mod_\cC, \Mod_{\cL\cC})\xrightarrow{\Omega} \Fun^{\otimes}(\cC, \cL\cC)\]
is equivalent by Lemma \ref{lm:moddecategorified} to $\mon\colon S^1\rightarrow \Fun^{\otimes}(\cC, \cL\cC)$. Therefore, by adjunction we get a commutative diagram
\[
\xymatrix{
\cC \ar^{\sim}[r] \ar^{\mon}[d] & \Omega\Mod_\cC \ar^{\Mon}[d] \\
\Fun(S^1, \cL\cC) \ar^{\sim}[r] & \Omega\Fun(S^1, \Mod_{\cL\cC})
}
\]
of $S^1$-equivariant symmetric monoidal $\infty$-categories.

Consider the diagram
\[
\xymatrix{
\cC^\dual \ar^{\sim}[r] \ar^{\mon}[d] \ar^{\sim}[r] & \Omega\Mod_\cC^{\dual} \ar[d] \ar^{\Mon}[d] \\
\Fun(S^1, (\cL\cC)^\dual) \ar^{\Tr}[d] \ar^{\sim}[r] & \Omega\Fun(S^1, \Mod_{\cL\cC}^{\dual}) \ar^{\Tr}[d] \\
\Omega\cL\cC \ar@{=}[r] & \Omega\cL\cC
}
\]
of $S^1$-equivariant $\bE_\infty$ spaces.
The bottom square commutes by Proposition \ref{prop:tracedecategorified}. The vertical composite on the left coincides with the Chern character $\ch\colon \cC^\dual\rightarrow \Omega\cL\cC$ and the claim follows from the commutativity of the diagram.
\end{proof}

\section{The categorified Grothendieck--Riemann--Roch theorem}

\subsection{Statement}

We have an obvious functoriality of the Chern character with respect to symmetric monoidal functors.

\begin{proposition}
\label{compatiblewithpull}
Let $f\colon \cD\rightarrow \cC$ be a symmetric monoidal functor. Then there is a commutative diagram of $\infty$-categories
\begin{equation}
\xymatrix{
\Mod^\dual_\cC \ar^-{\Ch}[r] & (\cL\cC)^{S^1} \\
\Mod^\dual_\cD \ar^{f^*}[u] \ar^-{\Ch}[r] & (\cL\cD)^{S^1} \ar^{\cL f}[u]
}
\label{eq:chernpullback}
\end{equation}
\end{proposition}

If $f$ is moreover rigid, by Propositions \ref{prop:ambidextrousdualadjoint} and \ref{prop:Lfrigid} we can pass to right adjoints of the vertical functors in \eqref{eq:chernpullback}; the resulting diagram a priori only commutes up to a natural transformation.

\begin{theorem}\label{thm:GRR}
Let $f\colon \cD\rightarrow \cC$ be a rigid symmetric monoidal functor. Then passing to right adjoints of the vertical functors in \eqref{eq:chernpullback} we obtain a diagram
\begin{equation}
\xymatrix{
\Mod^\dual_\cC \ar^{f_*}[d] \ar^-{\Ch}[r] & (\cL\cC)^{S^1} \ar^{(\cL f)^\R}[d] \\
\Mod^\dual_\cD \ar^-{\Ch}[r] \twocell{ur} & (\cL\cD)^{S^1}
}
\label{eq:GRRdiagram}
\end{equation}
which commutes up to an invertible 2-morphism.
\end{theorem}

The rest of the section is devoted to the proof of this theorem. Let $\cM$ be a dualizable $\cC$-module category. Without loss of generality (see \cite[Theorem 2.14]{Pstragowski}) we may assume that the duality data for $\cM$ is coherent in the sense of Definition \ref{def:coherentduality}.

The natural transformation in \eqref{eq:GRRdiagram} is obtained as the composite
\[
\Ch(f_*\cM)\longrightarrow (\cL f)^\R (\cL f) \Ch(f_*\cM)\cong (\cL f)^\R \Ch(f^* f_*\cM) \longrightarrow (\cL f)^\R \Ch(\cM).
\]

In turn, this is obtained as the composite 2-morphism in the diagram
\begin{equation}
\xymatrix{
\cL\cD \ar^{\lambda}[r] \ar^{\coev_{f_*\cM}}[d] & (\cL f)_*\cL\cC \ar@{=}[r] \ar^{\coev_{f^*f_*\cM}}[d] & (\cL f)_*\cL\cC \ar^{\coev}[d]
\\
p^*_\cD (f_*\cM\otimes f_*\cM^\vee) \ar^{\Mon_{f_*\cM}\otimes \id}[d] \ar^-{\lambda}[r] & (\cL f)_*p^*_\cC(f^*f_*\cM\otimes f^*f_*\cM^\vee) \ar^{\Mon_{f^*f_*\cM}\otimes \id}[d] \ar^-{\epsilon\otimes (\epsilon^\R)^\vee}[r] \twocell{ur} & (\cL f)_* p^*_\cC(\cM\otimes \cM^\vee) \ar^{\Mon_\cM\otimes \id}[d]
\\
p^*_\cD (f_*\cM\otimes f_*\cM^\vee) \ar^{\ev_{f_*\cM}}[d] \ar^-{\lambda}[r] & (\cL f)_*p^*_\cC(f^*f_*\cM\otimes f^*f_*\cM^\vee) \ar^{\ev_{f^*f_*\cM}}[d] \ar^-{\epsilon\otimes (\epsilon^\R)^\vee}[r] & (\cL f)_*p^*_\cC (\cM\otimes \cM^\vee) \ar^{\ev}[d]
\\
\cL\cD \ar@{=}[d] \ar^{\lambda}[r] & (\cL f)_*\cL\cC \ar^{\lambda^\R}[d] \ar@{=}[r] \twocell{ur} & (\cL f)_*\cL\cC \ar^{\lambda^\R}[d]
\\
\cL\cD \ar@{=}[r] \twocell{ur} & \cL\cD \ar@{=}[r] & \cL\cD
}
\label{eq:GRRdiagram1}
\end{equation}
in $\Mod_{\cL\cD}(\bPrst_k)$, where the columns are given by individual Chern characters. We are going to prove that the composite 2-morphism in \eqref{eq:GRRdiagram1} is a 2-isomorphism. All equalities of 2-morphisms in this section are to be understood in the homotopy 2-category.

As a first step, we are going to analyze the subdiagrams in \eqref{eq:GRRdiagram1} containing $(\epsilon^\R)^\vee$. We have a 2-isomorphism $(\epsilon^\R_\cM)^\vee\cong \epsilon_{\cM^\vee}$ constructed via the following diagram:
\begin{equation}
\xymatrix{
f^*f_*\cM^\vee \ar^{\id\otimes\coev}[d] \ar^{\epsilon}[r] & \cM^\vee \ar_{\id\otimes\coev}[d] \ar@{=}[dr] \\
f^*f_*\cM^\vee\otimes\cM\otimes\cM^\vee \ar^{\id\otimes\epsilon^\R}[d] \ar^{\epsilon\otimes\id\otimes\id}[r] & \cM^\vee\otimes \cM\otimes \cM^\vee \ar^{\epsilon^\R\otimes\id}[d] \ar_{\ev\otimes\id}[r] & \cM^\vee \ar_{\epsilon^\R\otimes\id}[d] \ar@{=}[dr] \\
f^*f_*\cM^\vee\otimes f^*f_*\cM\otimes \cM^\vee \ar^{\alpha\otimes\id}[r] & f^*f_*(\cM^\vee\otimes\cM)\otimes \cM^\vee \ar^-{\ev\otimes\id}[r] & f^*f_*\cC\otimes\cM^\vee \ar_-{\eta^\R\otimes\id}[r] & \cM^\vee
}
\label{eq:epsilonRdual}
\end{equation}
where the bottom-left square has the 2-isomorphism given by \eqref{eq:modification2}.

\subsection{Analyzing the (co)evaluation}

We will first apply the isomorphism $(\epsilon^\R_\cM)^\vee\cong \epsilon_{\cM^\vee}$ to the top part of diagram \eqref{eq:GRRdiagram1}.

\begin{lemma}
Under the identification $(\epsilon^\R_\cM)^\vee\cong \epsilon_{\cM^\vee}$ given by \eqref{eq:epsilonRdual} the diagram
\[
\xymatrix@C=2cm{
\cC \ar@{=}[r] \ar^{\coev_{f^*f_*\cM}}[dd] & \cC \ar^{\coev_\cM}[d] \\
& \cM\otimes_\cC \cM^\vee \ar^{\epsilon^\R\otimes\id}[d] \\
f^*f_*\cM\otimes_\cC f^*f_*\cM^\vee \ar^-{\id\otimes(\epsilon^\R)^\vee}[r] & f^*f_*\cM\otimes_\cC \cM^\vee
}
\]
becomes equal to
\[
\xymatrix@C=2cm{
\cC \ar@{=}[r] \ar^{\eta}[d] & \cC \ar@{=}[d] \\
f^*f_*\cC \ar^{\epsilon}[r] \ar^{\coev_\cM}[d] & \cC \ar^{\coev_\cM}[d] \\
f^*f_*(\cM\otimes_\cC\cM^\vee) \ar^{\alpha^\R}[d] \ar^{\epsilon}[r] & \cM\otimes_\cC \cM^\vee \ar^{\epsilon^\R\otimes\id}[d] \\
f^*f_*\cM\otimes_\cC f^*f_*\cM^\vee \ar^{\id\otimes\epsilon}[r] & f^*f_*\cM\otimes_\cC \cM^\vee
}
\]
in $h_2\Mod_\cC(\bPrst_k)$, where the bottom rectangle is given by \eqref{eq:modification3}.
\label{lm:reduction1}
\end{lemma}
\begin{proof}
The original diagram can be expanded to
\begin{center}
\resizebox{!}{0.016\textwidth}{
\xymatrix{
\cC \ar^{\coev}[r] \ar^{\eta}[d] & \cM\otimes \cM^\vee \ar^{\epsilon^\R\otimes\id}[r] \ar^{\eta\otimes\id\otimes\id}[d] & f^*f_*\cM\otimes \cM^\vee \ar@{=}[dr] \ar^{\eta\otimes\id\otimes\id}[d] &&&
\\
f^*f_*\cC \ar^{\coev}[d] \ar^-{\id\otimes\coev}[r] & f^*f_*\cC\otimes \cM\otimes \cM^\vee \ar^{\coev\otimes\id\otimes\id}[d] \ar^{\id\otimes\epsilon^\R\otimes\id}[r] & f^*f_*\cC\otimes f^*f_*\cM\otimes \cM^\vee \ar^{\coev\otimes\id\otimes\id}[d] \ar^{\alpha\otimes\id}[r] & f^*f_*\cM\otimes \cM^\vee \ar@{=}@/^1pc/[dr] \ar^{\coev\otimes\id\otimes\id}[d] &
\\
f^*f_*(\cM\otimes \cM^\vee) \ar^{\alpha^\R}[d] \ar^-{\id\otimes\coev}[r] & f^*f_*(\cM\otimes \cM^\vee)\otimes \cM\otimes \cM^\vee \ar^{\alpha^\R\otimes\id\otimes\id}[d] \ar^-{\id\otimes\id\otimes\epsilon^\R\otimes\id}[r] & f^*f_*(\cM\otimes \cM^\vee)\otimes f^*f_*\cM\otimes \cM^\vee \ar^{\alpha^\R\otimes\id\otimes\id}[d] \ar^{\alpha\otimes\id}[r] & f^*f_*(\cM\otimes \cM^\vee\otimes \cM)\otimes \cM^\vee \ar^-{\id\otimes\ev\otimes\id}[r] \ar^{\alpha^\R\otimes\id}[d] & f^*f_*\cM\otimes \cM^\vee \ar@{=}@/^1pc/[dr] \ar^{\alpha^\R\otimes\id}[d]
\\
f^*f_*\cM\otimes f^*f_*\cM^\vee \ar^-{\coev}[r] \ar_{\id\otimes \epsilon}[dr] & f^*f_*\cM\otimes f^*f_*\cM^\vee\otimes \cM\otimes \cM^\vee \ar^-{\id\otimes\id\otimes\epsilon^\R\otimes\id}[r] \ar^{\id\otimes\epsilon\otimes\id\otimes\id}[dr] & f^*f_*\cM\otimes f^*f_*\cM^\vee\otimes f^*f_*\cM\otimes \cM^\vee \ar^{\id\otimes\alpha\otimes\id}[r] & f^*f_*\cM\otimes f^*f_*(\cM^\vee\otimes \cM)\otimes \cM^\vee \ar^-{\id\otimes\ev\otimes\id}[r] & f^*f_*\cM\otimes f^*f_*\cC\otimes \cM^\vee \ar^-{\id\otimes\eta^\R\otimes\id}[r] & f^*f_*\cM\otimes \cM^\vee
\\
& f^*f_*\cM\otimes \cM^\vee \ar^-{\id\otimes\id\otimes\coev}[r] \ar@/_2pc/@{=}[rr] & f^*f_*\cM\otimes \cM^\vee\otimes \cM\otimes \cM^\vee \ar^{\id\otimes\ev\otimes\id}[r] \ar^{\id\otimes\epsilon^\R\otimes\id}[ur] & f^*f_*\cM\otimes \cM^\vee \ar^{\id\otimes\epsilon^\R\otimes\id}[ur] \ar@{=}[urr] &&
}
}
\end{center}

Using the fact that the 2-isomorphisms \eqref{eq:modification2} and \eqref{eq:modification3} are modifications, we get
\begin{center}
\resizebox{!}{0.016\textwidth}{
\xymatrix{
\cC \ar^{\coev}[r] \ar^{\eta}[d] & \cM\otimes \cM^\vee \ar^{\epsilon^\R\otimes\id}[r] \ar^{\eta\otimes\id\otimes\id}[d] & f^*f_*\cM\otimes \cM^\vee \ar@{=}[dr] \ar^{\eta\otimes\id\otimes\id}[d] &&&
\\
f^*f_*\cC \ar^{\coev}[d] \ar^-{\id\otimes\coev}[r] & f^*f_*\cC\otimes \cM\otimes \cM^\vee \ar^{\coev\otimes\id\otimes\id}[d] \ar^{\id\otimes\epsilon^\R\otimes\id}[r] \ar^{\epsilon\otimes\id\otimes\id}[dr] & f^*f_*\cC\otimes f^*f_*\cM\otimes \cM^\vee \ar^{\alpha\otimes\id}[r] & f^*f_*\cM\otimes \cM^\vee \ar@{=}@/^1pc/[dr] \ar^{\coev\otimes\id\otimes\id}[d] &
\\
f^*f_*(\cM\otimes \cM^\vee) \ar^{\alpha^\R}[d] \ar^-{\id\otimes\coev}[r] & f^*f_*(\cM\otimes \cM^\vee)\otimes\cM\otimes \cM^\vee \ar^{\alpha^\R\otimes\id\otimes\id}[d] \ar^{\epsilon\otimes\id\otimes\id}[dr] & \cM\otimes \cM^\vee \ar^{\epsilon^\R\otimes\id}[ur] \ar^{\coev\otimes\id\otimes\id}[d] & f^*f_*(\cM\otimes \cM^\vee\otimes \cM)\otimes \cM^\vee \ar^-{\id\otimes\ev\otimes\id}[r] \ar^{\alpha^\R\otimes\id}[d] & f^*f_*\cM\otimes \cM^\vee \ar@{=}@/^1pc/[dr] \ar^{\alpha^\R\otimes\id}[d]
\\
f^*f_*\cM\otimes f^*f_*\cM^\vee \ar^-{\coev}[r] \ar_{\id\otimes \epsilon}[dr] & f^*f_*\cM\otimes f^*f_*\cM^\vee\otimes \cM\otimes \cM^\vee \ar^{\id\otimes\epsilon\otimes\id\otimes\id}[dr] & \cM\otimes \cM^\vee\otimes \cM\otimes \cM^\vee \ar^{\epsilon^\R\otimes\id\otimes\id\otimes\id}[d] \ar^{\epsilon^\R\otimes\id}[ur] & f^*f_*\cM\otimes f^*f_*(\cM^\vee\otimes \cM)\otimes \cM^\vee \ar^-{\id\otimes\ev\otimes\id}[r] & f^*f_*\cM\otimes f^*f_*\cC\otimes \cM^\vee \ar^-{\id\otimes\eta^\R\otimes\id}[r] & f^*f_*\cM\otimes \cM^\vee
\\
& f^*f_*\cM\otimes \cM^\vee \ar^-{\id\otimes\id\otimes\coev}[r] \ar@/_2pc/@{=}[rr] & f^*f_*\cM\otimes \cM^\vee\otimes \cM\otimes \cM^\vee \ar^{\id\otimes\ev\otimes\id}[r] \ar^{\id\otimes\epsilon^\R\otimes\id}[ur] & f^*f_*\cM\otimes \cM^\vee \ar^{\id\otimes\epsilon^\R\otimes\id}[ur] \ar@{=}[urr] &&
}
}
\end{center}

Using the fact that the triangulator $\tau_1$ is a modification of natural transformations $\epsilon\eta\cong\id$, we obtain
\begin{center}
\resizebox{!}{0.012\textwidth}{
\xymatrix{
\cC \ar^{\coev}[r] \ar^{\eta}[d] & \cM\otimes \cM^\vee \ar^{\eta\otimes\id\otimes\id}[d] \ar@{=}@/^2pc/[ddr] &&&&
\\
f^*f_*\cC \ar^{\coev}[d] \ar^-{\id\otimes\coev}[r] & f^*f_*\cC\otimes \cM\otimes \cM^\vee \ar^{\coev\otimes\id\otimes\id}[d] \ar^{\epsilon\otimes\id\otimes\id}[dr] && f^*f_*\cM\otimes \cM^\vee \ar@{=}@/^1pc/[dr] \ar^{\coev\otimes\id\otimes\id}[d] &
\\
f^*f_*(\cM\otimes \cM^\vee) \ar^{\alpha^\R}[d] \ar^-{\id\otimes\coev}[r] & f^*f_*(\cM\otimes \cM^\vee)\otimes \cM\otimes \cM^\vee \ar^{\alpha^\R\otimes\id\otimes\id}[d] \ar^{\epsilon\otimes\id\otimes\id}[dr] & \cM\otimes \cM^\vee \ar^{\epsilon^\R\otimes\id}[ur] \ar^{\coev\otimes\id\otimes\id}[d] & f^*f_*(\cM\otimes \cM^\vee\otimes \cM)\otimes \cM^\vee \ar^-{\id\otimes\ev\otimes\id}[r] \ar^{\alpha^\R\otimes\id}[d] & f^*f_*\cM\otimes \cM^\vee \ar@{=}@/^1pc/[dr] \ar^{\alpha^\R\otimes\id}[d]
\\
f^*f_*\cM\otimes f^*f_*\cM^\vee \ar^-{\coev}[r] \ar_{\id\otimes \epsilon}[dr] & f^*f_*\cM\otimes f^*f_*\cM^\vee\otimes \cM\otimes \cM^\vee \ar^{\id\otimes\epsilon\otimes\id\otimes\id}[dr] & \cM\otimes \cM^\vee\otimes \cM\otimes \cM^\vee \ar^{\epsilon^\R\otimes\id\otimes\id\otimes\id}[d] \ar^{\epsilon^\R\otimes\id}[ur] & f^*f_*\cM\otimes f^*f_*(\cM^\vee\otimes \cM)\otimes \cM^\vee \ar^-{\id\otimes\ev\otimes\id}[r] & f^*f_*\cM\otimes f^*f_*\cC\otimes \cM^\vee \ar^-{\id\otimes\eta^\R\otimes\id}[r] & f^*f_*\cM\otimes \cM^\vee
\\
& f^*f_*\cM\otimes \cM^\vee \ar^-{\id\otimes\id\otimes\coev}[r] \ar@/_2pc/@{=}[rr] & f^*f_*\cM\otimes \cM^\vee\otimes \cM\otimes \cM^\vee \ar^{\id\otimes\ev\otimes\id}[r] \ar^{\id\otimes\epsilon^\R\otimes\id}[ur] & f^*f_*\cM\otimes \cM^\vee \ar^{\id\otimes\epsilon^\R\otimes\id}[ur] \ar@{=}[urr] &&
}
}
\end{center}

Cancelling out the two triangulators $\tau_1$ appearing in the modifications
\[(\id\otimes\eta^\R)\circ\alpha^\R\cong\id,\qquad (\id\otimes\eta^\R)\circ (\id\otimes\epsilon^\R)\cong\id\]
we obtain
\begin{center}
\resizebox{!}{0.015\textwidth}{
\xymatrix{
\cC \ar^{\coev}[r] \ar^{\eta}[d] & \cM\otimes_\cC \cM^\vee \ar^{\eta\otimes\id\otimes\id}[d] \ar@{=}@/^2pc/[ddr] &&&&
\\
f^*f_*\cC \ar^{\coev}[d] \ar^-{\id\otimes\coev}[r] & f^*f_*\cC\otimes \cM\otimes \cM^\vee \ar^{\coev\otimes\id\otimes\id}[d] \ar^{\epsilon\otimes\id\otimes\id}[dr] && f^*f_*\cM\otimes \cM^\vee \ar@{=}@/^1pc/[dr] \ar^{\coev\otimes\id\otimes\id}[d] &
\\
f^*f_*(\cM\otimes \cM^\vee) \ar^{\alpha^\R}[d] \ar^-{\id\otimes\coev}[r] & f^*f_*(\cM\otimes \cM^\vee)\otimes\cM\otimes \cM^\vee \ar^{\alpha^\R\otimes\id\otimes\id}[d] \ar^{\epsilon\otimes\id\otimes\id}[dr] & \cM\otimes \cM^\vee \ar^{\epsilon^\R\otimes\id}[ur] \ar^{\coev\otimes\id\otimes\id}[d] & f^*f_*(\cM\otimes \cM^\vee\otimes \cM)\otimes \cM^\vee \ar^-{\id\otimes\ev\otimes\id}[r] & f^*f_*\cM\otimes \cM^\vee
\\
f^*f_*\cM\otimes f^*f_*\cM^\vee \ar^-{\coev}[r] \ar_{\id\otimes \epsilon}[dr] & f^*f_*\cM\otimes f^*f_*\cM^\vee\otimes \cM\otimes \cM^\vee \ar^{\id\otimes\epsilon\otimes\id\otimes\id}[dr] & \cM\otimes \cM^\vee\otimes \cM\otimes_\cC \cM^\vee \ar^{\epsilon^\R\otimes\id\otimes\id\otimes\id}[d] \ar^{\id\otimes\ev\otimes\id}[r] \ar^{\epsilon^\R\otimes\id}[ur] & \cM\otimes\cM^\vee \ar^{\epsilon^\R\otimes\id}[ur]
\\
& f^*f_*\cM\otimes \cM^\vee \ar^-{\id\otimes\id\otimes\coev}[r] \ar@/_2pc/@{=}[rr] & f^*f_*\cM\otimes \cM^\vee\otimes \cM\otimes \cM^\vee \ar^{\id\otimes\ev\otimes\id}[r] & f^*f_*\cM\otimes \cM^\vee \ar@{=}[uur] &&
}
}
\end{center}

Using naturality of $\epsilon^\R$ we get
\begin{center}
\resizebox{!}{0.020\textwidth}{
\xymatrix{
\cC \ar^{\coev}[r] \ar^{\eta}[d] & \cM\otimes \cM^\vee \ar^{\eta\otimes\id\otimes\id}[d] \ar@{=}@/^2pc/[ddr] &&&&
\\
f^*f_*\cC \ar^{\coev}[d] \ar^-{\id\otimes\coev}[r] & f^*f_*\cC\otimes \cM\otimes \cM^\vee \ar^{\coev\otimes\id\otimes\id}[d] \ar^{\epsilon\otimes\id\otimes\id}[dr] &&&
\\
f^*f_*(\cM\otimes \cM^\vee) \ar^{\alpha^\R}[d] \ar^-{\id\otimes\coev}[r] & f^*f_*(\cM\otimes \cM^\vee)\otimes\cM\otimes \cM^\vee \ar^{\alpha^\R\otimes\id\otimes\id}[d] \ar^{\epsilon\otimes\id\otimes\id}[dr] & \cM\otimes \cM^\vee \ar^{\coev\otimes\id\otimes\id}[d] \ar@{=}@/^1pc/[dr]
\\
f^*f_*\cM\otimes f^*f_*\cM^\vee \ar^-{\coev}[r] \ar_{\id\otimes \epsilon}[dr] & f^*f_*\cM\otimes f^*f_*\cM^\vee\otimes \cM\otimes \cM^\vee \ar^{\id\otimes\epsilon\otimes\id\otimes\id}[dr] & \cM\otimes \cM^\vee\otimes \cM\otimes \cM^\vee \ar^{\epsilon^\R\otimes\id\otimes\id\otimes\id}[d] \ar^-{\id\otimes\ev\otimes\id}[r]& \cM\otimes\cM^\vee \ar^{\epsilon^\R\otimes\id}[d]
\\
& f^*f_*\cM\otimes \cM^\vee \ar^-{\id\otimes\id\otimes\coev}[r] \ar@/_2pc/@{=}[rr] & f^*f_*\cM\otimes \cM^\vee\otimes \cM\otimes \cM^\vee \ar^-{\id\otimes\ev\otimes\id}[r] & f^*f_*\cM\otimes \cM^\vee &&
}
}
\end{center}

Using the fact that \eqref{eq:modification3} is a modification we get
\begin{center}
\resizebox{!}{0.020\textwidth}{
\xymatrix{
\cC \ar^{\coev}[r] \ar^{\eta}[d] & \cM\otimes \cM^\vee \ar^{\eta\otimes\id\otimes\id}[d] \ar@{=}@/^2pc/[ddr] &&&&
\\
f^*f_*\cC \ar^{\coev}[d] \ar^-{\id\otimes\coev}[r] & f^*f_*\cC\otimes \cM\otimes \cM^\vee \ar^{\coev\otimes\id\otimes\id}[d] \ar^{\epsilon\otimes\id\otimes\id}[dr] &&&
\\
f^*f_*(\cM\otimes \cM^\vee) \ar^{\alpha^\R}[d] \ar^{\epsilon}[dr] \ar^-{\id\otimes\coev}[r] & f^*f_*(\cM\otimes \cM^\vee)\otimes\cM\otimes \cM^\vee \ar^{\epsilon\otimes\id\otimes\id}[dr] & \cM\otimes \cM^\vee \ar^{\coev\otimes\id\otimes\id}[d] \ar@{=}@/^1pc/[dr]
\\
f^*f_*\cM\otimes f^*f_*\cM^\vee \ar_{\id\otimes \epsilon}[dr] & \cM\otimes\cM^\vee \ar^{\epsilon^\R\otimes\id}[d] \ar^{\id\otimes\id\otimes\coev}[r] & \cM\otimes \cM^\vee\otimes \cM\otimes \cM^\vee \ar^{\epsilon^\R\otimes\id\otimes\id\otimes\id}[d] \ar^-{\id\otimes\ev\otimes\id}[r]& \cM\otimes\cM^\vee \ar^{\epsilon^\R\otimes\id}[d]
\\
& f^*f_*\cM\otimes \cM^\vee \ar^-{\id\otimes\id\otimes\coev}[r] \ar@/_2pc/@{=}[rr] & f^*f_*\cM\otimes \cM^\vee\otimes \cM\otimes \cM^\vee \ar^-{\id\otimes\ev\otimes\id}[r] & f^*f_*\cM\otimes \cM^\vee &&
}
}
\end{center}

Applying naturality of $\epsilon$ we get
\[
\xymatrix{
\cC \ar^{\coev}[r] \ar^{\eta}[d] & \cM\otimes \cM^\vee \ar^{\eta\otimes\id\otimes\id}[d] \ar@{=}@/^2pc/[ddr] &&&&
\\
f^*f_*\cC \ar^{\coev}[d] \ar^-{\id\otimes\coev}[r] \ar^{\epsilon}[dr] & f^*f_*\cC\otimes \cM\otimes_\cC \cM^\vee \ar^{\epsilon\otimes\id\otimes\id}[dr] &&&
\\
f^*f_*(\cM\otimes \cM^\vee) \ar^{\alpha^\R}[d] \ar^{\epsilon}[dr] & \cC \ar^{\coev}[d] \ar^{\coev}[r] & \cM\otimes \cM^\vee \ar^{\coev\otimes\id\otimes\id}[d] \ar@{=}@/^1pc/[dr]
\\
f^*f_*\cM\otimes f^*f_*\cM^\vee \ar_{\id\otimes \epsilon}[dr] & \cM\otimes\cM^\vee \ar^{\epsilon^\R\otimes\id}[d] \ar^-{\id\otimes\id\otimes\coev}[r] \ar@/_2pc/@{=}[rr] & \cM\otimes \cM^\vee\otimes \cM\otimes \cM^\vee \ar^-{\id\otimes\ev\otimes\id}[r] & \cM\otimes\cM^\vee \ar^{\epsilon^\R\otimes\id}[d]
\\
& f^*f_*\cM\otimes \cM^\vee \ar@{=}[rr] && f^*f_*\cM\otimes \cM^\vee &&
}
\]

Using the swallowtail axiom for the coherent duality data for $\cM$ we get the result.
\end{proof}

\ 

From Lemma \ref{lm:reduction1} we obtain that the diagram
\[
\xymatrix{
(\cL f)_*\cL\cC \ar@{=}[r] \ar^{\coev_{f^*f_*\cM}}[d] & (\cL f)_*\cL\cC \ar^{\coev}[d] \\
(\cL f)_*p^*_\cC(f^*f_*\cM\otimes f^*f_*\cM^\vee) \ar^-{\epsilon\otimes (\epsilon^\R)^\vee}[r] \ar@/_2pc/_{\epsilon\otimes\epsilon}[r] \twocell{ur} & (\cL f)_* p^*_\cC (\cM\otimes \cM^\vee)
}
\]
is equal to the one obtained by applying $(\cL f)_*p^*_\cC$ to
\[
\xymatrix{
\cC \ar@{=}[r] \ar^{\eta}[d] & \cC \ar@{=}[d] \\
f^*f_*\cC \ar^{\epsilon}[r] \ar^{\coev_\cM}[d] & \cC \ar^{\coev_\cM}[d] \\
f^*f_*(\cM\otimes\cM^\vee) \ar^{\alpha^\R}[d] \ar^{\epsilon}[r] & \cM\otimes \cM^\vee \ar^{\epsilon^\R\otimes\id}[d] \ar@{=}@/^2pc/[dr] & \\
f^*f_*\cM\otimes f^*f_*\cM^\vee \ar^{\id\otimes\epsilon}[r] & f^*f_*\cM\otimes \cM^\vee \ar^{\epsilon\otimes\id}[r] \twocell{ur} & \cM\otimes \cM^\vee
}
\]

Using the obvious equivalence $\epsilon\otimes\epsilon\cong \epsilon\circ \alpha$ and Lemma \ref{lm:modification3intertwiner} (2) it is equal to
\[
\xymatrix{
& \cC \ar@{=}[r] \ar^{\eta}[d] & \cC \ar@{=}[d] \\
& f^*f_*\cC \ar^{\epsilon}[r] \ar^{\coev_\cM}[d] & \cC \ar^{\coev_\cM}[d] \\
\twocell{dr} & f^*f_*(\cM\otimes\cM^\vee) \ar_{\alpha^\R}@/_2pc/[dl] \ar^{\epsilon}[r] \ar@{=}[d] & \cM\otimes \cM^\vee \ar@{=}[d] \\
f^*f_*\cM\otimes f^*f_*\cM^\vee \ar^{\alpha}[r] & f^*f_*(\cM\otimes \cM^\vee) \ar^{\epsilon}[r] & \cM\otimes \cM^\vee
}
\]

Now we are going to analyze the bottom part of diagram \eqref{eq:GRRdiagram1} in a similar way.

\begin{lemma}
Under the identification $(\epsilon^\R_\cM)^\vee\cong \epsilon_{\cM^\vee}$ given by \eqref{eq:epsilonRdual} the diagram
\[
\xymatrix{
  \cM\otimes_\cC f^*f_* \cM^\vee  \ar^{\epsilon^\R\otimes\id}[d] \ar^-{\id \otimes (\epsilon^\R)^\vee}[r] &  \cM\otimes_\cC\cM^\vee \ar^{\ev_\cM}[dd]  \\
 f^* f_*\cM\otimes_\cC f^*f_* \cM^\vee     \ar^{\ev_{f^*f_*\cM}}[d] & \\
\cC \ar@{=}[r] & \cC \\
}
\]
becomes equal to
\[
\xymatrix{
\cM\otimes_\cC f^*f_* \cM^\vee  \ar^{\epsilon^\R \otimes\id}[d]  \ar^{\id\otimes\epsilon}[r] &  \cM\otimes_\cC \cM^\vee \ar@{=}[dd] \\
 f^*(f_*\cM\otimes_{\cD} f_*\cM^\vee) \ar^{\alpha}[d] &   \\
 f^*f_*(\cM\otimes_\cC\cM^\vee)  \ar^{\ev_\cM}[d] &   \cM\otimes_\cC  \cM^\vee   \ar^{\ev_\cM}[d]   \ar^{\epsilon^\R}[l]\\ 
 f^* f_*\cC  \ar^{\eta^\R}[d] &   \cC   \ar@{=}[d]     \ar^{\epsilon^\R}[l]\\
 \cC \ar@{=}[r] &  \cC 
 }
 \]
where the top rectangle is given by \eqref{eq:modification2}.
\label{lm:reduction2}
\end{lemma}

From Lemma \ref{lm:reduction2} we obtain that the diagram
\[
\xymatrix{
(\cL f)_*p^*_\cC(f^*f_*\cM\otimes f^*f_*\cM^\vee) \ar^-{\epsilon\otimes (\epsilon^\R)^\vee}[r] \ar^{\ev_{f^*f_*\cM}}[d] \ar@/^2pc/^{\epsilon\otimes\epsilon}[r] & (\cL f)_* p^*_\cC(\cM\otimes \cM^\vee) \ar^{\ev_\cM}[d] \\
(\cL f)_*\cL\cC \ar@{=}[r] \twocell{ur} & (\cL f)_*\cL\cC
}
\]
is equal to the one obtained by applying $(\cL f)_*p^*_\cC$ to
\[
\xymatrix{
f^*f_*\cM\otimes f^*f_*\cM \ar^{\epsilon\otimes\id}[r] \ar@{=}@/_2pc/[dr] & \cM\otimes f^*f_* \cM^\vee  \ar^{\epsilon^\R \otimes\id}[d]  \ar^{\id\otimes\epsilon}[r] & \cM\otimes \cM^\vee \ar@{=}[dd] \\
\twocell{ur} & f^*(f_*\cM\otimes f_*\cM^\vee) \ar^{\alpha}[d] & \\
& f^*f_*(\cM\otimes \cM^\vee)  \ar^{\ev_\cM}[d] & \cM\otimes \cM^\vee \ar^{\ev_\cM}[d]   \ar^-{\epsilon^\R}[l] \\ 
& f^* f_*\cC \ar^{\eta^\R}[d] & \cC \ar@{=}[d] \ar^-{\epsilon^\R}[l] \\
& \cC \ar@{=}[r] & \cC
}
\]

Using the equivalence $\epsilon\otimes\epsilon\cong \epsilon\circ \alpha$ and Lemma \ref{lm:modification3intertwiner} (1) it is equivalent to
\[
\xymatrix{
f^*f_*\cM\otimes f^*f_* \cM^\vee \ar^{\alpha}[d] \ar^{\alpha}[r] & f^*f_*(\cM\otimes \cM^\vee) \ar@/^2pc/^{\epsilon}[dr] \ar@{=}[d] & \\
f^*f_*(\cM\otimes \cM^\vee) \ar^{\ev_\cM}[d] \ar@{=}[r] & f^*f_*(\cM\otimes\cM^\vee) \twocell{ur} & \cM\otimes\cM^\vee \ar_-{\epsilon^\R}[l] \ar^{\ev_\cM}[d] \\ 
f^* f_*\cC \ar^{\eta^\R}[d] && \cC \ar@{=}[d] \ar^{\epsilon^\R}[ll] \\
\cC \ar@{=}[rr] && \cC
}
\]

\subsection{Reduction to $\cM=\cC$}

Observe that the diagram
\[
\xymatrix{
p^*_\cD(f_*\cM\otimes f_*\cM^\vee) \ar^{\alpha}[d] \ar^-{\lambda}[r] & (\cL f)_* p^*_\cC f^* (f_*\cM\otimes f_*\cM^\vee) \ar^{\alpha}[r] \ar^{\alpha}[d] \twocell{dr} & (\cL f)_* p^*_\cC f^*f_*(\cM\otimes \cM^\vee) \ar^{\epsilon}[d] \\
p^*_\cD f_*(\cM\otimes \cM^\vee) \ar^-{\lambda}[r] \ar^{\ev_\cM}[d] & (\cL f)_*p^*_\cC f^* f_*(\cM\otimes \cM^\vee) \ar^{\ev_\cM}[d] & (\cL f)_*p^*_\cC(\cM\otimes \cM^\vee) \ar^{\ev_\cM}[d] \ar^{\epsilon^\R}[l] \\ 
p^*_\cD f_*\cC \ar^{\lambda}[r] \ar^{\eta^\R}[d] & (\cL f)_*p^*_\cC f_*\cC \ar^{\eta^\R}[d] & (\cL f)_*\cL\cC \ar@{=}[d] \ar^{\epsilon^\R}[l] \\
\cL\cD \ar^{\lambda}[r] \ar@{=}@/_2pc/[dr] & (\cL f)_*\cL\cC \ar@{=}[r] \ar^{\lambda^\R}[d] & (\cL f)_*\cL\cC \ar^{\lambda^\R}[d] \\
\twocell{ur} & \cL\cD \ar@{=}[r] & \cL\cD
}
\]
is equivalent to
\[
\xymatrix{
p^*_\cD(f_*\cM\otimes f_*\cM^\vee) \ar^{\alpha}[d] \ar^-{\lambda}[r] & (\cL f)_* p^*_\cC f^* (f_*\cM\otimes f_*\cM^\vee) \ar^{\alpha}[r] \ar^{\alpha}[d] \twocell{dr} & (\cL f)_* p^*_\cC f^* f_*(\cM\otimes \cM^\vee) \ar^{\epsilon}[d] \\
p^*_\cD f_*(\cM\otimes \cM^\vee) \ar^-{\lambda}[r] \ar@{=}@/_2pc/[dr] & (\cL f)_*p^*_\cC f^*f_*(\cM\otimes \cM^\vee) \ar^{\lambda^\R}[d] & (\cL f)_*p^*_\cC(\cM\otimes \cM^\vee) \ar@{=}[d] \ar^{\epsilon^\R}[l] \\ 
\twocell{ur} & p^*_\cD f_*(\cM\otimes \cM^\vee) \ar^{\ev_\cM}[d] & (\cL f)_*p^*_\cC(\cM\otimes \cM^\vee) \ar^{\ev_\cM}[d] \\
& p^*_\cD f_*\cC \ar^{\eta^\R}[d] & (\cL f)_* \cL\cC \ar^{\lambda^\R}[d] \ar_{\lambda^\R\epsilon^\R}[l] \\
& \cL\cD \ar@{=}[r] & \cL\cD
}
\]

Therefore, applying previous simplifications and removing invertible 2-morphisms from \eqref{eq:GRRdiagram1} we get
\begin{equation*}\label{eq:GRRdiagram2}
\resizebox{!}{0.023\textwidth}{
\xymatrix{
p^*_\cD f_* (\cM\otimes \cM^\vee) \ar^-{\lambda}[r] \ar^{\alpha^\R}[d] & (\cL f)_*p^*_\cC f^*f_* (\cM\otimes \cM^\vee) \ar^{\alpha^\R}[d] \ar@{=}[r] & (\cL f)_*p^*_\cC f^*f_* (\cM\otimes \cM^\vee) \ar^{\epsilon}[r] \ar@{=}[d] & (\cL f)_*p^*_\cC (\cM\otimes \cM^\vee) \ar@{=}[d]
\\
p^*_\cD (f_*\cM\otimes f_*\cM^\vee) \ar^{\Mon_{f_*\cM}\otimes \id}[d] \ar^-{\lambda}[r] & (\cL f)_*p^*_\cC f^*(f_*\cM\otimes f_*\cM^\vee) \ar^{\Mon_{f^*f_*\cM}\otimes \id}[d] \ar^-{\alpha}[r] \twocell{ur} & (\cL f)_*p^*_\cC f^*f_* (\cM\otimes \cM^\vee) \ar^{\epsilon}[r] & (\cL f)_*p^*_\cC (\cM\otimes \cM^\vee) \ar^{\Mon_\cM\otimes\id}[d]
\\
p^*_\cD (f_*\cM\otimes f_*\cM^\vee) \ar^{\alpha}[d] \ar^-{\lambda}[r] & (\cL f)_*p^*_\cC f^*(f_*\cM\otimes f_*\cM^\vee) \ar^{\alpha}[d] \ar^-{\alpha}[r] & (\cL f)_*p^*_\cC f^*f_* (\cM\otimes \cM^\vee) \ar^{\epsilon}[r] \ar@{=}[d] & (\cL f)_*p^*_\cC (\cM\otimes \cM^\vee) \ar^{\epsilon^\R}[d]
\\
p^*_\cD f_*(\cM\otimes \cM^\vee) \ar@{=}@/_2pc/[dr] \ar^-{\lambda}[r] & (\cL f)_*p^*_\cC f^*f_*(\cM\otimes \cM^\vee) \ar^{\lambda^\R}[d] \ar@{=}[r] & (\cL f)_*p^*_\cC f^*f_*(\cM\otimes \cM^\vee) \ar@{=}[r] \twocell{ur} & (\cL f)_*p^*_\cC f^*f_*(\cM\otimes \cM^\vee)
\\
\twocell{ur} & p^*_\cD f_*(\cM\otimes \cM^\vee) &&
}
}
\end{equation*}

We have a diagram
\[
\xymatrix{
p^*_\cD f_* (\cM\otimes \cM^\vee) \ar[r] \ar[d] & (\cL f)_*p^*_\cC (\cM\otimes \cM^\vee) \ar[d] \\
p^*_\cD f_*(\cM\otimes \cM^\vee) \ar@{=}[r] \twocell{ur} & p^*_\cD f_*(\cM\otimes \cM^\vee)
}
\]
which we have to prove commutes up to an invertible 2-morphism. It will be enough to prove that
\[
\xymatrix{
p^*_\cD f_* (\cM\otimes \cN) \ar[r] \ar[d] & (\cL f)_*p^*_\cC (\cM\otimes \cN) \ar[d] \\
p^*_\cD f_*(\cM\otimes \cN) \ar@{=}[r] \twocell{ur} & p^*_\cD f_*(\cM\otimes \cN)
}
\]
commutes up to an invertible 2-morphism for any pair of $\cC$-modules $\cM,\cN$.

Note that all functors are $\Prst_k$-linear and commute with geometric realizations, so by \cite[Theorem 4.8.4.1]{HA} it is enough to prove the assertion for $\cM=\cN=\cC$.

Substituting $\cM=\cN=\cC$ and interchanging the first two columns we obtain a diagram
\[
\xymatrix{
\cL\cD\otimes_\cD \cC \ar@{=}[r] \ar^{\Delta^\R}[d] & \cL\cD\otimes_\cD\cC \ar^{p}[r] \ar@{=}[d] & \cL\cC \ar@{=}[d]
\\
\cL\cD\otimes_\cD(\cC\otimes_\cD\cC) \ar^-{\Delta}[r] \ar^{\Mon_{f_*\cC}\otimes\id}[d] \twocell{ur} & \cL\cD\otimes_\cD \cC \ar^{p}[r] & \cL\cC \ar@{=}[d]
\\
\cL\cD\otimes_\cD(\cC\otimes_\cD\cC) \ar^-{\Delta}[r] \ar^{\Delta}[d] & \cL\cD\otimes_\cD \cC \ar^{p}[r] \ar@{=}[d] & \cL\cC \ar^{p^\R}[d]
\\
\cL\cD\otimes_\cD\cC \ar@{=}[r] & \cL\cD\otimes_\cD\cC \ar@{=}[r] \twocell{ur} & \cL\cD\otimes_\cD\cC
}
\]
in $\bPrst_k$.
By construction this is the mate of the 2-isomorphism in the rectangle
\[
\xymatrix{
\cL\cD\otimes_\cD(\cC\otimes_\cD\cC) \ar^-{\Delta}[r] \ar^{\Mon_{f_*\cC}\otimes\id}_{\sim}[d] & \cL\cD\otimes_\cD \cC \ar^{p}[dd] \\
\cL\cD\otimes_\cD(\cC\otimes_\cD\cC) \ar^{\Delta}[d] & \\
\cL\cD\otimes_\cD \cC \ar^{p}[r] & \cL\cC
}
\]

\subsection{Reduction to rigidity}

We will now simplify the above rectangle to show that it is right adjointable.

Let $\cE$ be an $\infty$-category with finite colimits and consider a diagram
\begin{equation}
\begin{tikzcd}
& B_1 \arrow[rr] \arrow[from=dd] && B_{12} \\
B_0 \arrow[ur] \arrow[rr, crossing over] && B_2 \arrow[ur] & \\
& A_1 \arrow[rr] && A_{12} \arrow[uu] \\
A_0 \arrow[rr] \arrow[uu] \arrow[ur] && A_2 \arrow[uu, crossing over] \arrow[ur] &
\end{tikzcd}
\label{eq:cube}
\end{equation}

The morphism $A_{12}\rightarrow B_{12}$ gives rise to a natural transformation $A_{12}\amalg_{A_2}(-)\rightarrow B_{12}\amalg_{A_2}(-)$ of functors $\cE_{A_2/}\rightarrow \cE$. The universal property of the pushout gives a map $A_2\amalg_{A_0} B_0\rightarrow B_2$. Therefore, we get morphisms
\[A_{12}\amalg_{A_2}(A_2\amalg_{A_0} B_0)\rightarrow B_{12}\amalg_{A_2}(A_2\amalg_{A_0} B_0)\rightarrow B_{12}\amalg_{A_2}B_2\rightarrow B_{12}.\]
Replacing $A_2$ by $A_1$ and $B_2$ by $B_1$ we similarly get morphisms
\[A_{12}\amalg_{A_1}(A_1\amalg_{A_0} B_0)\rightarrow B_{12}\amalg_{A_1}(A_1\amalg_{A_0} B_0)\rightarrow B_{12}\amalg_{A_1}B_1\rightarrow B_{12}.\]
Using the commutativity data of diagram \eqref{eq:cube} we obtain
\begin{equation}
\xymatrix{
A_{12}\amalg_{A_2}(A_2\amalg_{A_0} B_0) \ar[r] \ar^{\sim}[dd] & B_{12}\amalg_{A_2}(A_2\amalg_{A_0}B_0) \ar[r] \ar^{\sim}[dd] & B_{12}\amalg_{A_2} B_2 \ar[dr] & \\
&&& B_{12} \\
A_{12}\amalg_{A_1}(A_1\amalg_{A_0} B_0) \ar[r] & B_{12}\amalg_{A_1} (A_1\amalg_{A_0}B_0) \ar[r] & B_{12}\amalg_{A_1} B_1 \ar[ur] &
}
\label{eq:cubediagram1}
\end{equation}
We can exchange the first two columns to obtain another diagram:
\begin{equation}
\xymatrix{
A_{12}\amalg_{A_2}(A_2\amalg_{A_0} B_0) \ar[r] \ar^{\sim}[dd] & A_{12}\amalg_{A_2} B_2 \ar[r] & B_{12}\amalg_{A_2} B_2 \ar[dr] & \\
&&& B_{12} \\
A_{12}\amalg_{A_1}(A_1\amalg_{A_0} B_0) \ar[r] & A_{12}\amalg_{A_1} B_1 \ar[r] & B_{12}\amalg_{A_1} B_1 \ar[ur] &
}
\label{eq:cubediagram2}
\end{equation}
which we will draw as
\begin{equation}
\xymatrix{
A_{12}\amalg_{A_2}(A_2\amalg_{A_0} B_0) \ar[r] \ar^{\sim}[d] & A_{12}\amalg_{A_2} B_2 \ar[dd] \\
A_{12}\amalg_{A_1}(A_1\amalg_{A_0} B_0) \ar[d] & \\
A_{12}\amalg_{A_1} B_1 \ar[r] & B_{12}
}
\label{eq:pentagondiagram}
\end{equation}
which gives a square in $\cE$, i.e. a functor $\Delta^1\times\Delta^1\rightarrow \cE$.

\begin{lemma}
Suppose that the top and bottom squares in \eqref{eq:cube} are coCartesian. Then the square \eqref{eq:pentagondiagram} is equal to the square
\[
\xymatrix{
A_1\amalg_{A_0}(B_0\amalg_{A_0} A_2) \ar[r] \ar[d] & A_1\amalg_{A_0} B_2 \ar[d] \\
B_1\amalg_{B_0}(B_0\amalg_{A_0} A_2) \ar[r] & B_1\amalg_{B_0} B_2
}
\]
obtained using naturality of the transformation $A_1\amalg_{A_0}(-)\rightarrow B_1\amalg_{B_0}(-)$ of functors $\cE_{B_0/}\rightarrow \cE$ with respect to the morphism $B_0\amalg_{A_0} A_2\rightarrow B_2$.
\label{lm:pentagonnaturality}
\end{lemma}

Now consider the case $\cE=\CAlg(\Prst_k)$. Using the functor $f\colon \cD\rightarrow \cC$ of $k$-linear symmetric monoidal $\infty$-categories we obtain a cube
\[
\begin{tikzcd}
& \cC \arrow[rr] \arrow[from=dd] && \cL\cC \\
\cC\otimes\cC \arrow[ur] \arrow[rr, crossing over] && \cC \arrow[ur] & \\
& \cD \arrow[rr] && \cL\cD \arrow[uu] \\
\cD\otimes\cD \arrow[rr] \arrow[uu] \arrow[ur] && \cD \arrow[uu, crossing over] \arrow[ur] &
\end{tikzcd}
\]
where the top and bottom squares are of type $(1, 0)$ (see Section \ref{sect:uncatchern} for what this means). In this case the isomorphisms
\[A_{12}\amalg_{A_2}(A_2\amalg_{A_0} B_0)\xrightarrow{\sim} A_{12}\amalg_{A_1}(A_1\amalg_{A_0} B_0)\]
and
\[A_{12}\amalg_{A_2}(A_2\amalg_{A_0} B_0)\xrightarrow{\sim} A_{12}\amalg_{A_1}(A_1\amalg_{A_0} B_0)\]
become by Section \ref{ss: Cat-Mon} the categorified monodromy maps
\[\Mon_{f_*\cC}\otimes\id\colon \cL\cD\otimes_\cD(\cC\otimes_\cD\cC)\xrightarrow{\sim} \cL\cD\otimes_\cD(\cC\otimes_\cD\cC)\]
and
\[\Mon_{f^*f_*\cC}\otimes\id\colon \cL\cC\otimes_\cD(\cC\otimes_\cD\cC)\xrightarrow{\sim} \cL\cC\otimes_\cD(\cC\otimes_\cD\cC).\]

The diagram \eqref{eq:cubediagram1} in this case becomes
\[
\xymatrix{
\cL\cD\otimes_\cD(\cC\otimes_\cD\cC) \ar[r] \ar^{\Mon_{f_*\cC}\otimes\id}_{\sim}[dd] & \cL\cC\otimes_\cD(\cC\otimes_\cD\cC) \ar[r] \ar^{\Mon_{f^*f_*\cC}\otimes\id}_{\sim}[dd] & \cL\cC\otimes_\cD\cC \ar[dr] & \\
&&& \cL\cC \\
\cL\cD\otimes_\cD(\cC\otimes_\cD\cC) \ar[r] & \cL\cC\otimes_\cD(\cC\otimes_\cD\cC) \ar[r] & \cL\cC\otimes_\cD\cC \ar[ur]
}
\]

Similarly, the diagram \eqref{eq:cubediagram2} in this case becomes
\[
\xymatrix{
\cL\cD\otimes_\cD(\cC\otimes_\cD\cC) \ar[r] \ar^{\Mon_{f_*\cC}\otimes\id}_{\sim}[dd] & \cL\cD\otimes_\cD\cC \ar[r] & \cL\cC\otimes_\cD\cC \ar[dr] & \\
&&& \cL\cC \\
\cL\cD\otimes_\cD(\cC\otimes_\cD\cC) \ar[r] & \cL\cD\otimes_\cD\cC \ar[r] & \cL\cC\otimes_\cD\cC \ar[ur]
}
\]
which we may draw as a rectangle \eqref{eq:pentagondiagram}:
\[
\xymatrix{
\cL\cD\otimes_\cD(\cC\otimes_\cD\cC) \ar^{\Mon_{f_*\cC}\otimes\id}[d] \ar[r] & \cL\cD\otimes_\cD\cC \ar[dd] \\
\cL\cD\otimes_\cD(\cC\otimes_\cD\cC) \ar[d] & \\
\cL\cD\otimes_\cD\cC \ar[r] & \cL\cC
}
\]
which we have to show is right adjointable.

By Proposition \ref{prop:10cocartesian} the bottom and top squares in the cube are coCartesian, so Lemma \ref{lm:pentagonnaturality} applies and the above rectangle is equivalent to the square
\[
\xymatrix{
\cD\otimes_{\cD\otimes\cD}(\cC\otimes_\cD\cC) \ar^-{\Delta}[r] \ar[d] & \cD\otimes_{\cD\otimes\cD} \cC \ar[d] \\
\cC\otimes_{\cC\otimes\cC}(\cC\otimes_\cD\cC) \ar^-{\Delta}[r] & \cC\otimes_{\cC\otimes\cC} \cC
}
\]
expressing naturality of the transformation $\cD\otimes_{\cD\otimes\cD}(-)\rightarrow \cC\otimes_{\cC\otimes\cC}(-)$ with respect to the morphism of $\cC\otimes\cC$-modules $\Delta\colon\cC\otimes_\cD\cC\rightarrow \cC$. Since $f$ is rigid, $\Delta$ admits a right adjoint in $\Mod_{\cC\otimes_\cD\cC}(\bPrst_k)$ and hence in $\Mod_{\cC\otimes \cC}(\bPrst_k)$, so the square is right adjointable.

\section{Applications to the uncategorified Chern character} 

\subsection{The Ben-Zvi–Nadler Chern character}

Ben-Zvi and Nadler give a construction of a Chern character based on the functoriality properties of traces in symmetric monoidal $(\infty,2)$-categories. Let us recall their definition.

Suppose $\cC$ is a $k$-linear rigid symmetric monoidal $\infty$-category. By Proposition \ref{prop:selfduality}, $\cC$ is dualizable in $\Mod_{\Mod_k}$ and we have an equivalence
\[\cC^\dual\cong \Hom_{\Mod_{\Mod_k}^{\dual}}(\Mod_k, \cC)\]
given by sending a dualizable object $x\in\cC$ to the functor $k\mapsto x$.

Let $\dim\colon \Mod_{\Mod_k}^\dual \to \Mod_k$ be the composite
\[
\Mod_{\Mod_k}^\dual \to \Aut(\Mod_{\Mod_k}) \xrightarrow{\Tr} \Omega\Mod_{\Mod_k} \cong \Mod_k,
\]
where the first functor sends $\cM$ to $\id_\cM$.
Then we have $\dim(\Mod_k)\cong k$ and $\dim(\cC)\cong \Omega\cL\cC$. Therefore, we also get a map
\[\dim\colon \Hom_{\Mod_{\Mod_k}^{\dual}}(\Mod_k, \cC)\rightarrow \Omega\cL\cC.\]

The Chern character defined in \cite{BN3} is given by the composite
\[\cC^\dual\rightarrow \Hom_{\Mod_{\Mod_k}^{\dual}}(\Mod_k, \cC)\xrightarrow{\dim}\Omega\cL\cC.\]

\begin{theorem}
\label{TVchern=BZNchern}
Suppose $\cC$ is a $k$-linear rigid symmetric monoidal $\infty$-category. Then the uncategorified Chern character $\ch\colon \cC^\dual\rightarrow \Omega\cL\cC$ is equivalent to the composite
\[\cC^\dual\rightarrow \Hom_{\Mod_{\Mod_k}^{\dual}}(\Mod_k, \cC)\xrightarrow{\dim}\Omega\cL\cC.\]
\end{theorem}
\begin{proof}
The categorified GRR Theorem \ref{thm:GRR} applied to $\Mod_k\rightarrow \cC$ gives a commutative square
\[
\xymatrix{
\Mod_\cC^{\dual} \ar^{\Ch}[r] \ar[d] & \cL\cC \ar[d] \\
\Mod_{\Mod_k}^{\dual} \ar^{\dim}[r] & \Mod_k.
}
\]

Evaluating it on the endomorphisms of $\cC\in\Mod_\cC^{\dual}$, we get the top square in the diagram
\[
\xymatrix{
\cC^\dual \ar^{\sim}[dr] && \\
& \Omega\Mod_\cC^\dual \ar^{\Ch}[r] \ar[d] & \Omega\cL\cC \ar[d] & \\
& \Hom_{\Mod_{\Mod_k}^{\dual}}(\cC, \cC) \ar^{\dim}[r] \ar[d] & \Hom(\Omega\cL\cC, \Omega\cL\cC) \ar[d] & \\
& \Hom_{\Mod_{\Mod_k}^{\dual}}(\Mod_k, \cC) \ar^{\dim}[r] & \Omega\cL\cC
}
\]

Here the morphism $\Omega\cL\cC\rightarrow \Hom(\Omega\cL\cC, \Omega\cL\cC)$ is adjoint to the multiplication map on $\Omega\cL\cC$,
\[\Hom(\Omega\cL\cC, \Omega\cL\cC)\rightarrow \Omega\cL\cC\]
is given by the evaluation on the identity element and
\[\Hom_{\Mod_{\Mod_k}^{\dual}}(\cC, \cC)\rightarrow \Hom_{\Mod_{\Mod_k}^{\dual}}(\Mod_k, \cC)\] is given by precomposition with the unit $\Mod_k\rightarrow \cC$. The bottom square commutes by the functoriality of dimensions.

The composite
\[\cC^\dual\rightarrow \Omega\Mod_\cC^\dual\rightarrow \Omega\cL\cC\]
is the uncategorified Chern character by Theorem \ref{thm:cherndecategorified}, so the claim follows from the commutativity of the diagram.

\end{proof}

\subsection{From the categorified GRR to the classical GRR}
The classical Grothendieck--Riemann--Roch theorem states the functoriality of the uncategorified Chern character with respect to the pushforward functor $f_*\colon \Perf(X)\rightarrow \Perf(Y)$ for a suitable morphism of schemes $f\colon X \to Y$. In this section we will prove its generalization with values in a sheaf of categories.

Let $f\colon \cD\rightarrow \cC$ be a rigid symmetric monoidal functor. Let $\cT$ be a dualizable $\cC$-module category, $\cT'$ a dualizable $\cD$-module category and $g\colon f_*\cT\rightarrow \cT'$ a right adjointable morphism in $\Mod_\cD$, i.e., a morphism in $\Mod_\cD^{\dual}$.

Then we have Chern characters
\[\Ch\colon \Hom_{\Mod_\cC^{\dual}}(\cC, \cT)\rightarrow \Hom_{(\cL\cC)^{S^1}}(\bu_{\cL\cC}, \Ch(\cT))\]
and
\[\Ch\colon \Hom_{\Mod_\cD^{\dual}}(\cD, \cT')\rightarrow \Hom_{(\cL\cD)^{S^1}}(\bu_{\cL\cD}, \Ch(\cT')).\]

Moreover, we can define pushforward maps as follows. The map
\[\Hom_{\Mod_\cC^{\dual}}(\cC, \cT)\rightarrow \Hom_{\Mod_\cD^{\dual}}(\cD, \cT')\]
sends a morphism $x\colon \cC\rightarrow \cT$ to the composite
\[\cD\rightarrow f_*\cC\xrightarrow{x}f_*\cT\xrightarrow{g} \cT'.\]
Similarly, the map
\[\Hom_{(\cL\cC)^{S^1}}(\bu_{\cL\cC}, \Ch(\cT))\rightarrow \Hom_{(\cL\cD)^{S^1}}(\bu_{\cL\cD}, \Ch(\cT'))\]
is given by sending a morphism $\bu_{\cL\cC}\rightarrow \Ch(\cT)$ to the composite
\[\bu_{\cL\cD}\rightarrow (\cL f)^\R \bu_{\cL\cC}\rightarrow (\cL f)^\R \Ch(\cT)\rightarrow \Ch(\cT'),\]
where the last morphism is $\Ch(f_*\cT\rightarrow \cT')\colon (\cL f)^\R\Ch(\cT)\rightarrow \Ch(\cT')$.

\begin{theorem}
We have a commutative diagram of spaces
\[
\xymatrix{
\Hom_{\Mod_\cC^{\dual}}(\cC, \cT) \ar^-{\Ch}[r] \ar[d] & \Hom_{(\cL\cC)^{S^1}}(\bu_{\cL\cC}, \Ch(\cT)) \ar[d] \\
\Hom_{\Mod_\cD^{\dual}}(\cD, \cT') \ar^-{\Ch}[r] & \Hom_{(\cL\cD)^{S^1}}(\bu_{\cL\cD}, \Ch(\cT')).
}
\]
\label{thm:classicalGRR}
\end{theorem}
\begin{proof}
We have a commutative diagram of spaces
\[
\xymatrix{
\Hom_{\Mod_\cC^{\dual}}(\cC, \cT) \ar^-{\Ch}[r] \ar[d] & \Hom_{(\cL\cC)^{S^1}}(\bu_{\cL\cC}, \Ch(\cT)) \ar^{(\cL f)^\R}[d] \\
\Hom_{\Mod_\cD^{\dual}}(f_*\cC, f_*\cT) \ar^-{\Ch}[r] \ar[d] & \Hom_{(\cL\cD)^{S^1}}((\cL f)^\R \bu_{\cL\cC}, \Ch(f_*\cT)) \ar[d] \\
\Hom_{\Mod_\cD^{\dual}}(\cD, f_*\cT) \ar^-{\Ch}[r] \ar[d] & \Hom_{(\cL\cD)^{S^1}}(\bu_{\cL\cD}, \Ch(f_*\cT)) \ar[d] \\
\Hom_{\Mod_\cD^{\dual}}(\cD, \cT') \ar^-{\Ch}[r] & \Hom_{(\cL\cD)^{S^1}}(\bu_{\cL\cD}, \Ch(\cT'))
}
\]
where the top square commutes by the categorified GRR Theorem \ref{thm:GRR} applied to $f\colon \cD\rightarrow \cC$ and the rest of the squares commute by functoriality of the Chern character $\Ch$.
\end{proof}

\begin{corollary} \label{cor:classicalGRR}
Suppose $f\colon \cD\rightarrow \cC$ is a rigid symmetric monoidal functor which is moreover proper in the sense of Definition \ref{def-proper-smooth}. Then we have a commutative diagram of spaces
\[
\xymatrix{
\cC^\dual\ar_{f^\R}[d] \ar^-{\ch}[r] & (\Omega\cL\cC)^{S^1} \ar[d]^{\int_f} \\
\cD^\dual \ar^-{\ch}[r] & (\Omega\cL\cD)^{S^1}.
}
\]

\end{corollary}
\begin{proof}
The claim is obtained from Theorem \ref{thm:classicalGRR} by setting $\cT=\cC$, $\cT'=\cD$, and $g=f^\R\colon f_*\cC\rightarrow \cD$. The fact that the horizontal maps are the Chern characters follows from Theorem~\ref{thm:cherndecategorified}.
\end{proof}

\begin{remark}
Let $f\colon X\rightarrow Y$ be a morphism of perfect stacks which is representable, proper, of finite Tor-amplitude, and locally almost of finite presentation. Then by Example \ref{ex:smoothproper} the pullback functor $f^*\colon\QCoh(Y)\rightarrow \QCoh(X)$ is rigid and proper. Therefore, the corollary in this case produces a commutative diagram
\[
\xymatrix{
\iota_0 \Perf(X) \ar_{f^\R}[d] \ar^{\ch}[r] & \cO(\cL X) \ar^{\int_f}[d] \\
\iota_0 \Perf(Y) \ar^{\ch}[r] & \cO(\cL Y).
}
\]

If we moreover assume that $X\rightarrow Y$ is a smooth morphism of smooth schemes over a field $k$ of characteristic zero, then by the results of Markarian \cite{M} the HKR isomorphisms $\cO(\cL X)\cong \Omega^{-\bullet}(X)$ intertwine the integration map $\int_f\colon \cO(\cL X)\rightarrow \cO(\cL Y)$ and the integration map $\int_f\colon \Omega^{-\bullet}(X)\rightarrow \Omega^{-\bullet}(Y)$ on differential forms twisted by the relative Todd class $\Td_{X/Y}$. Therefore, we obtain a commutative diagram
\[
\xymatrix{
\iota_0 \Perf(X) \ar_{f^\R}[d] \ar^{\ch}[r] & \cO(\cL X) \ar^{\int_f}[d] \ar^{\sim}[r] & \Omega^{-\bullet}(X) \ar^{\int_f (-)\wedge \Td_{X/Y}}[d] \\
\iota_0 \Perf(Y) \ar^{\ch}[r] & \cO(\cL Y) \ar^{\sim}[r] & \Omega^{-\bullet}(Y).
}
\]
\end{remark}

\subsection{The Grothendieck--Riemann--Roch Theorem for the secondary Chern character}
\label{the grothendieck riemann}
\def\cg{\mathrm{cg}}

In this section we prove a Grothendieck–Riemann–Roch theorem for the secondary Chern character. If $\cC$ is a $k$-linear symmetric monoidal $\infty$-category, we denote by $\Mod^\sat_\cC$ the full subcategory of $\Mod^\dual_\cC$ spanned by the saturated $\cC$-modules (Definition~\ref{def-saturated}).

\begin{definition}
Let $\cC$ be a $k$-linear symmetric monoidal $\infty$-category. We define the \defterm{secondary Chern character} to be the composite
\[\ch^{(2)}\colon \iota_0 \Mod^\sat_{\cC} \to ((\cL \cC)^{\dual})^{S^1} \to \Omega(\cL^2 \cC) ^{(S^1\times S^1)}\]
where the first map is the categorified Chern character for $\cC$ and the second map is the classical Chern character for $\cL\cC$.
\end{definition}

\begin{theorem}[Secondary GRR]
\label{Secondary GRR}
	Let $f \colon\cD\to\cC$ be a smooth and proper functor of symmetric monoidal $\infty$-categories.
	Then the square
	\[
	\xymatrix{ \iota_0\Mod^{\sat}_\cC \ar^-{\ch^{(2)}}[r] \ar[d]_{f_*} & \Omega_\Sp(\cL^2\cC)^{(S^1\times S^1)} \ar[d]^{\int_{\cL f}} \\
		\iota_0 \Mod^{\sat}_\cD \ar^-{\ch^{(2)}}[r]  & \Omega_\Sp(\cL^2\cD)^{(S^1\times S^1)}}
	\]
	commutes.
\end{theorem}

\begin{proof}
	Since $f$ is smooth and proper, the pushforward $f_*\colon \Mod_\cC \to\Mod_\cD$ preserves saturated $\infty$-categories (Lemma~\ref{lem:smoothproper}).	Restricting Theorem~\ref{thm:GRR} to saturated $\infty$-categories, we therefore obtain a commutative square
	\[
	\xymatrix{\Mod_\cC^\sat \ar^-{\Ch}[r] \ar[d]_{f_*} & (\cL\cC )^{S^1} \ar[d]^{\cL f^\R} \\
		\Mod_\cD^\sat \ar^-{\Ch}[r]  & (\cL\cD)^{S^1},}
	\]
 By Lemma  \ref{lem:proper}, $\cL f \colon \cL\cD \to \cL\cC$ is proper.  Hence, composing the above commutative diagram with the classical GRR as in Corollary \ref{cor:classicalGRR}, yields the statement. 
 \end{proof}

 Let us assume  that $\cC$ is compactly generated and rigid in the sense of Definition \ref{def: rigid}. We denote the subcategory of compact objects by $\cC^\omega$. By Corollary \ref{cor:rigid}, we have that $\cC^\dual=\iota_0\cC^\omega$.  
 We will write $\mathrm{mod}_{\cC^\omega}$ for the $\infty$-category of small stable idempotent complete $\cC^\omega$-linear $\infty$-categories. The $\mathrm{Ind}$-completion functor identifies $\mathrm{mod}_{\cC^\omega}$ with a full subcategory of $\Mod_{\cC}^\dual$. In \cite{HSS}, we considered the $\infty$-category $\mathrm{mod}^{\sat}_{\cC^{\omega}}$ of small saturated $\cC^\omega$-linear $\infty$-categories, which is the intersection $\mathrm{mod}_{\cC^\omega}\cap \Mod_\cC^\sat$. As proved in \cite[Theorem 6.20]{HSS}, $\ch^{(2)}$ descends to a morphism of $\bE_\infty$ ring spectra
\[
\ch^{(2)}\colon \K^{(2)}(\cC^\omega) \to \Omega_\Sp(\cL^2\cC)^{(S^1\times S^1)},
\]
where $\Omega_\Sp$ is the spectrum of endomorphisms of the unit object.
This is the diagonal composition in the diagram
\[
\xymatrix{ \iota_0  \mathrm{mod}_{\cC^\omega}^\sat  \ar[r]\ar[d] & \iota_0  (\cL\cC^\omega)^{S^1}  \ar@{.>}[r]  \ar@{.>}[d] & \Omega_\Sp(\cL^2\cC)^{(S^1\times S^1)} \\
	\iota_0 \Mot^\sat(\cC^\omega) \ar[ur] \ar@{.>}[d]  & \K^{S^1}(\cL\cC^\omega)  \ar[ru] & \\
	\K^{(2)}(\cC^\omega)  \ar[ru] & & }\]
where a dotted arrow means a map to the infinite loop space of the target, see \cite[Diagram (6.19)]{HSS}.

 \begin{theorem}[Motivic GRR]
\label{Motivic GRR}
Let us assume that $\cC$ and $\cD$ are compactly generated and rigid, and let $f \colon\cD\to\cC$ be a rigid symmetric monoidal functor.
	Then the square
		\[
		\xymatrix{ \MOT(\cC^\omega) \ar^-{\Ch}[r] \ar[d]_{f_*} & (\cL\cC)^{S^1} \ar[d]^{\cL f^\R} \\
			\MOT(\cD^\omega) \ar^-{\Ch}[r]  & (\cL\cD)^{S^1}}
\]
commutes. If $f$ is smooth and proper, the square
\[	\xymatrix{ \Mot^\sat(\cC^\omega) \ar^-{\Ch}[r] \ar[d]_{f_*} & (\cL\cC^\omega)^{S^1} \ar[d]^{\cL f^\R} \\
		\Mot^\sat(\cD^\omega) \ar^-{\Ch}[r]  & (\cL\cD^\omega)^{S^1}}
	\]
	commutes.
\end{theorem}
\begin{proof}
First note that the functor $f_*\colon \Mod^\dual_{\cC} \to \Mod^\dual_{\cD}$ preserves compactly generated $\infty$-categories. 
The functor $\mathrm{mod}_{\cC^\omega} \to \MOT(\cC^\omega)$ is by definition the universal functor to a presentable stable $\infty$-category that preserves zero objects, exact sequences, and filtered colimits. 
The functors
\[
\mathrm{mod}_{\cC^{\omega}} \xrightarrow{f_*} \mathrm{mod}_{\cD^{\omega}}\xrightarrow{\mathrm{Ind}} \Mod^\dual_{\cD} \xrightarrow{\Ch} (\cL\cD)^{S^1}
\]
and
\[
\mathrm{mod}_{\cC^{\omega}} \xrightarrow{\mathrm{Ind}}  \Mod^\dual_{\cC} \xrightarrow{\Ch} (\cL\cC)^{S^1}
\]
satisfy these conditions and therefore the categorified GRR factors through the commutative square as in the statement. The commutativity of the second square is proved in the same way, using that $\mathrm{mod}^\sat_{\cC^\omega}\to \Mot^\sat(\cC^\omega)$ is the universal functor to a stable idempotent complete $\infty$-category that preserves zero objects and exact sequences.
\end{proof}

\begin{corollary}\label{cor:classicalGRR-for-K-theory}
	Assume that $\cC$ and $\cD$ are compactly generated and rigid, and let $f \colon\cD\to\cC$ be a proper symmetric monoidal functor.
		Then we have a commutative square of spectra
			\[
			\xymatrix{ \K(\cC^\omega) \ar^-{\ch}[r] \ar[d]_{f_*} & \Omega_\Sp(\cL\cC)^{S^1} \ar[d]^{\int_f} \\
				\K(\cD^\omega) \ar^-{\ch}[r]  & \Omega_\Sp(\cL\cD)^{S^1}.}
	\]
\end{corollary}

\begin{proof}
	Repeat the proof of Corollary~\ref{cor:classicalGRR}, using the first square of Theorem~\ref{Motivic GRR} instead of Theorem~\ref{thm:GRR}.
\end{proof}

 \begin{theorem}[Secondary motivic GRR]
\label{Secondary motivic GRR}
Let $\cC$ and $\cD$ be compactly generated rigid categories and let $f \colon\cD\to\cC$ be a smooth and proper symmetric monoidal functor.
	Then the square
	\[
	\xymatrix{ \K^{(2)}(\cC^\omega) \ar^-{\ch^{(2)}}[r] \ar[d]_{f_*} &  \Omega_\Sp(\cL^2\cC)^{(S^1\times S^1)}\ar[d]^{\int_{\cL f}} \\
		\K^{(2)}(\cD^\omega) \ar^-{\ch^{(2)}}[r]  & \Omega_\Sp(\cL^2\cD)^{(S^1\times S^1)}}
	\]
	commutes.
\end{theorem}
\begin{proof}
	Applying nonconnective $K$-theory to the second square in Theorem \ref{Motivic GRR}, we get a commutative square of spectra
	\[
	\xymatrix{ \K^{(2)}(\cC^\omega) \ar^-{\Ch}[r] \ar[d]_{f_*} & \K^{S^1}(\cL\cC^\omega) \ar[d]^{\cL f^\R} \\
		\K^{(2)}(\cD^\omega) \ar^-{\Ch}[r]  & \K^{S^1}(\cL\cD^\omega).}
	\]
	  Now $\cL f\colon \cL \cD \to \cL \cC$  is also a rigid symmetric functor which is proper (Lemma~\ref{lem:proper}). Hence, Corollary~\ref{cor:classicalGRR-for-K-theory} applied to $\cL f\colon \cL \cD \to \cL \cC$ 
	  yields the commutative square 
	\[
	\xymatrix{ \K (\cL \cC^\omega) \ar^-{\ch}[r] \ar[d]_{\cL f^\R} &  \Omega_\Sp(\cL^{2} \cC)^{S^1}\ar[d]^{\int_{\cL f}} \\
		\K(\cL \cD^\omega) \ar^-{\ch}[r]  & \Omega_\Sp(\cL^2\cD)^{S^1}.}
	\]
Combining the two squares yields the statement.
\end{proof}

\subsection{Secondary Chern character and the  motivic Chern class}
 In this section we 
 establish a comparison between %prove that 
 the  secondary Chern character and Brasselet, Sch\"urmann and Youkura's  \emph{motivic Chern class} \cite{brasselet2005hirzebruch}. 
 The motivic Chern class is an enhancement of MacPherson's total Chern class of singular varieties \cite{macpherson1974chern} and, as explained  in \cite{schuermann2009specialization}, it specializes to  other  well-known invariants of singular varieties. %, and we limit ourselves to clarify 

Throughout the section  $k$ is a field of characteristic $0.$ A \emph{variety} is an integral separated scheme of finite type over $\mathrm{Spec}(k)$.
For $X$ a variety, we write $\Mot(X)$ for the presentable stable $\infty$-category $\MOT(\Perf(X))$ of localizing $\Perf(X)$-motives \cite[Definition 5.14]{HSS}.

\begin{definition} 
We denote by $\Mot_{\mathrm{BM}}(X)$ the smallest stable idempotent complete full subcategory of $\Mot(X)$ such that for every proper map
$f\colon Y \to X$ from a smooth variety  the pushforward factors as    %and $Y$ is a smooth variety, then 
$$
\xymatrix{
& \Mot_{\mathrm{BM}}(X) \ar[d] \\
\Mot^\sat(Y) \ar[ru] %\ar@{-->}[ru] 
\ar[r]^-{f_*} & \Mot(X).}  
$$
We 
call $\Mot_{\mathrm{BM}}(X)$ the \defterm{$\infty$-category of Borel--Moore noncommutative motives} over $X.$
\end{definition}
 
%\begin{remark}
%If $X$ is smooth algebraic variety there is an inclusion 
%$\Mot^\sat(X)  \stackrel{\subset } \longrightarrow \Mot_{\mathrm{BM}}(X).$  
%However this, in general, is not an equivalence. 
%The category $\Mot_{\mathrm{BM}}(X)$ is functorial with respect to  proper morphisms,  and 
\begin{remark}
If $f\colon X\to Y$ is a proper morphism of algebraic varieties, there is  a well-defined pushforward functor
$ \,  
f_*\colon \Mot_{\mathrm{BM}}(Y) \to   \Mot_{\mathrm{BM}}(X) \,.    
$ 
The qualifier \emph{Borel--Moore}  alludes to this feature. 
%is meant to allude  to this feature. 
\end{remark}%of $\Mot_{\mathrm{BM}}(X).$ 

 \begin{lemma}
 \label{lemmafacbmch}
The restriction 
of $\Ch$ %is a factorization % restriction of the categorified Chern character 
to $\Mot_{\mathrm{BM}}(X)$ factors as %through the stable category of coherent sheaves on $\cL X$
$$
\xymatrix{
& \mathrm{Coh}(\cL X) \ar[d] \\ 
\Mot_{\mathrm{BM}}(X) \ar[r]^-{\Ch} \ar[ru]%@{-->}[ru]   
& \QCoh(\cL X).}
$$
\end{lemma}
\begin{proof}
%Let $f:Y \to X$ be a proper map with source a smooth variety. 
%We show first that $\cL f_*$ factors as
%\begin{equation}
%\begin{gathered}
%\label{fafactoror}
%\xymatrix{
%& \mathrm{Coh}(\cL X) \ar[d] \\ 
%\Perf(\cL Y) \ar[r]^-{\cL f_*} \ar@{-->}[ru]  & \QCoh(\cL X)}
%\end{gathered}
%\end{equation}
%Let us show that the factorization (\ref{fafactoror}) 
Let $f\colon Y \to X$ be a proper map from a smooth variety $Y$. As $f\colon Y \to X$ is proper, $\cL f$ is proper and there is a well-defined  pushforward 
$ \, 
\cL f_*\colon \mathrm{Coh}(\cL Y) \to  \mathrm{Coh}(\cL X)
\, .$
Further, since $Y$ is smooth, $\cL Y$ is eventually coconnective (see Lemma~\ref{lem:Gorenstein}), and therefore there is an inclusion 
$ \, 
\Perf(\cL Y) \subset \mathrm{Coh}(\cL Y). 
\, $
The motivic GRR theorem (Theorem~\ref{Motivic GRR}) yields a commutative square
\begin{equation}
\label{dddddia}
\begin{gathered}
\xymatrix{\Mot^\sat(Y) \ar[d]_-{f_*} \ar[r]^-\Ch & \Perf(\cL Y) \ar[d]^ -{\cL f_*} \\
\Mot_{\mathrm{BM}}(X)  \ar[r]^-\Ch & \QCoh(\cL X). }
\end{gathered}
\end{equation}
By the previous discussion, the upper composition $\cL f_* \circ \Ch$ %in square (\ref{dddddia})  
lands in  $\mathrm{Coh}(\cL X)$
 $$
\Mot^{\sat}(Y) \xrightarrow{ \cL f_* \circ \Ch } \mathrm{Coh}(\cL X) \subset \QCoh(\cL X).
$$
Then, since   (\ref{dddddia}) is commutative, the lower composition $\Ch \circ f_*$ also corestricts to $\mathrm{Coh}(\cL X).$% \subset \QCoh(\cL X).$ 

Now, by definition $\Mot_{\mathrm{BM}}(X)$ is generated under fibers, cofibers, and retracts by the images of the pushforward functors $$f_*\colon \Mot^{\sat}(Y) \to \Mot(X)$$ as $Y\to X$ ranges over all proper maps with smooth domain. 
%$$
%\Mot_{\mathrm{BM}}(X) = \langle f_*\Mot^{\sat}(Y), \, \, \text{$Y$ is smooth and $Y \to X$ is proper} \rangle
%$$
We conclude that $\Ch$ restricted to $\Mot_{\mathrm{BM}}(X)$ lands in  $\mathrm{Coh}(\cL X),$  which is what we wanted to prove.  
\end{proof}
\begin{definition}
We denote by $K^{(2)}_{\mathrm{BM}}(X)$ the algebraic K-theory of 
$\Mot_{\mathrm{BM}}(X),$ 
$$ 
K^{(2)}_{\mathrm{BM}}(X) := K(\Mot_{\mathrm{BM}}(X)).$$
\end{definition}

%The BM secondary K-theory is the source of a character map which is closely related to the secondary Chern character, and which is   
%introduced  in Definition \ref{bmchern} below. %which 
%In Definition \ref{} below we introduce the Borel--Moore secondary Chern character. 
Let $i_X\colon X \to \cL X$ be the embedding of the trivial loops. By \cite[Proposition 9.2]{barwick2015exact}  the pushforward   in $G$-theory,  
$ \, 
i_{X*}\colon G(X) \longrightarrow G(\cL X),$ 
is an equivalence. %between the G-theory of $X$ and $\cL X.$ %is equivalent to the G-theory of $X$: further, the equivalence is given by   
\begin{definition}
\label{bmchern}
We denote by $\ch^{(2)}_{\mathrm{BM}}$ the map 
$$
\ch^{(2)}_{\mathrm{BM}}\colon K^{(2)}_{\mathrm{BM}}(X)  = K(\Mot_{\mathrm{BM}}(X)) \xrightarrow{K(\Ch)} K(\mathrm{Coh}(\cL X)) \simeq G(X). 
$$
We call $\ch^{(2)}_{\mathrm{BM}}$ the \defterm{BM  secondary Chern character}. 
\end{definition}
%With small abuse of notations we will denote $\ch^{(2)}_{\mathrm{BM}}$ also the map induced on the groups of connected components
%$$
%\ch^{(2)}_{\mathrm{BM}}: K_{\mathrm{BM}, 0}^{(2)}(X) \longrightarrow G_0(X)
%$$
The categorified GRR theorem implies a  GRR statement for the BM secondary Chern character. 
\begin{proposition}
\label{grrbmchernch}
Let $f\colon Y \to X$ be a proper map of algebraic varieties. Then there is a commutative square
$$
\xymatrix{
K_{\mathrm{BM}}^{(2)}(Y) \ar[r]^-{\ch^{(2)}_{\mathrm{BM}}} \ar[d]_-{f_*} & G(Y) \ar[d]^-{f_*} \\ 
K_{\mathrm{BM}}^{(2)}(X) \ar[r]^-{\ch^{(2)}_{\mathrm{BM}}} & G(X).
}
$$
\end{proposition}
%\begin{proof}
%By the categorified GRR theorem and Lemma \ref{lemmafacbmch} we have a commutative square 
%$$
%\xymatrix{
%\Mot_{\mathrm{BM}}(Y) \ar[r]^-{\ch^{(2)}_{\mathrm{BM}}}  \ar[d]_-{f_*} & \mathrm{Coh}(\cL Y) \ar[d]^-{f_*} \\
%\Mot_{\mathrm{BM}}(X) \ar[r]^-{\ch^{(2)}_{\mathrm{BM}}} \ar[r]  & \mathrm{Coh}(\cL X). 
%}
%$$
%The statement follows by taking $K$-theory. 
%\end{proof}

%The BM secondary K-theory of $X$ is the  recipient of a natural map from the Grothendieck group of varieties over $X,$ $K_0(\Var_X).$ This generalizes the map defined in \cite{BLL} when $X$ is point. We start by recalling the presentation of $K_0(\Var_X)$ obtained by Bittner in \cite{bittner2004universal}. %: we remark that the work of Bittner relies on  resolution of singularities and therefore requires the assumption that the ground field $k$  has characteristic $0.$

%Let $X$ be a variety. %, and let $K_0(\Var_X)$ be the Grothendieck group of varieties over $X$. 
In \cite{bittner2004universal} Bittner obtains a presentation of the Grothendieck group of varieties over $X,$ $K_0(\Var_X)$, which we recall next. She proves that $K_0(\Var_X)$ is isomorphic to the free abelian group on isomorphism classes of proper maps
$ \, \, 
[Y   \longrightarrow X],$ 
such that $Y$ is smooth and equidimensional, subject to the following two relations 
\begin{enumerate}
\item $[\varnothing \longrightarrow X] = 0$
\item For every diagram 
$$
\xymatrix{E \ar[r]^-{i}  \ar[d]_-p  & \Bl_{Z}(Y) \ar[d]^-q  & \\
Z \ar[r]^-{j} & Y \ar[r]^-f  & X}
$$
where $j$ is a closed embedding of smooth equidimensional algebraic varieties, $\Bl_{Z}(Y)$ is the blow-up of $Y$ along $Z,$ and $E$ is the exceptional divisor, 
$$
[\Bl_ZY   \longrightarrow X] - [E \longrightarrow X] = [Y \longrightarrow X] - [Z \longrightarrow X] 
\quad \text{in} \quad K_0(\Var_X).
$$
\end{enumerate}
 
If $\cC$ is an $\infty$-category over $X$ having the property that its motive lies in $\Mot_{\mathrm{BM}}(X),$ we denote by $[\cC]$ its class in $K_{\mathrm{BM}}^{(2)}(X).$ 
%
%
%
%
%
%Let $f: Y \rightarrow X$ be a proper map from a smooth algebraic variety $Y$. Let 
%$f_*\cU(\Perf(Y))$ be the corresponding motive in $\Mot_{\mathrm{BM}}(X),$ and  denote $[f_*\cU(\Perf(Y))]$ its class in $K^{(2)}_{\mathrm{BM},0}(X).$ 
\begin{proposition}
There is a homomorphism of groups 
$ 
\mu\colon K_0(\Var_X) \to K^{(2)}_{\mathrm{BM}, 0}(X)
$ 
given by the assignment: 
$$  
\quad [Y \stackrel{f} \longrightarrow X] \in K_0(\Var_X) \mapsto f_*[\Perf(Y)] \in K^{(2)}_{\mathrm{BM}, 0}(X).
$$ 
\end{proposition}
\begin{proof}
The proof is the same as the one given in  \cite{BLL} for the case $X= \mathrm{Spec}(k).$ The key ingredient is Orlov's formula for the category of perfect complexes of blow-ups, see \cite[Proposition 7.5]{BLL}. The only thing to prove is that the assignment 
$$  
\quad [Y \stackrel{f} \longrightarrow X] \in K_0(\Var_X) \mapsto f_*[\Perf(Y)] \in K^{(2)}_{\mathrm{BM}, 0}(X).
$$ 
is compatible with the relations $(1)$ and $(2)$ coming from Bittner's presentation of $K_0(\Var_X)$.

 Relation $(1)$ reduces to the fact that $\Perf(\varnothing)$ is the $0$ category. Now let $Z \stackrel{j} \to Y$ be as in relation $(2)$, and let $s$ be the codimension of $Z$ in $Y$.  Orlov's formula,  proved in \cite{orlov1993projective}, gives a $\Perf(Y)$-linear semi-orthogonal decomposition of  $\Perf(\Bl_{Z}(Y))$ with $s$ factors: one copy of $\Perf(Y)$ and $s$-$1$ copies of $\Perf(Z)$. The exceptional divisor $E \subset \Bl_{Z}(Y)$ is a  projective bundle over $E$ of rank $s$-$1$. Thus by  \cite{orlov1993projective} $\Perf(E)$ has a $\Perf(Z)$-linear semi-orthogonal decomposition with $s$ factors  all equivalent to $\Perf(Z)$. 

We obtain the following identities in $K^{(2)}_{\mathrm{BM}, 0}(X)$:
$$ 
(f   q)_*[\Perf(\Bl_{Z}(Y))] = f_*[\Perf(Y)] + (s-1)(f  j)_*[\Perf(Z)], \quad (f   q   i)_*[\Perf(E)] = (i f)_*[\Perf(Z)]
$$
This immediately implies that 
$$
f_*[\Perf(\Bl_{Z}(Y))] -  (f q i)_*[\Perf(E)]  =  f_*[\Perf(Y)] -  (fj)_*[\Perf(Z)]  
$$
which can be rewritten as the identity 
$$
\mu(\Bl_{Z}(Y)) - \mu(E) = \mu(Y) - \mu(Z)
$$
This shows that relation $(2)$ is satisfied, and concludes the proof.
\end{proof}

\subsubsection{The motivic Chern class}
Let $X$ be a variety. 
The motivic Chern class was defined in \cite{brasselet2005hirzebruch}. It is the morphism
$$
mC_*\colon K_0(\Var_X) \longrightarrow  G_0(X) \otimes \Z[y] 
$$
 %, and it
%As we will explain below, %following \cite{brasselet2005hirzebruch} 
%$\ch_\mathrm{mot}$ can be viewed as 
%an enhancement of  MacPherson's construction of the total Chern class of singular varieties \cite{macpherson1974chern}. 
which is uniquely determined by the following two properties: 
\begin{enumerate}
\item If $X$ is smooth, $\, \, mC_*([X \xrightarrow{1_X}  X]) = \sum [\Omega^ i_X] \cdot  y^ i \in G_0(X) \otimes \Z[y]$ 
\item If $Y \to X$ is a proper map and $Y$ is a smooth algebraic variety there is a commutative diagram \begin{equation}
\label{grrmotchernbra}
\begin{gathered}
\xymatrix{
K_0(\Var_Y) \ar[r]^-{\ch_\mathrm{mot}} \ar[d]_-{f_*} & G_0(Y) \otimes \Z[y] \ar[d]^-{f_*} \\
K_0(\Var_X) \ar[r]^-{\ch_\mathrm{mot}} & G_0(X) \otimes \Z[y].
}
\end{gathered}
\end{equation}
\end{enumerate}
\begin{theorem}
\label{motchern}
Let $X$ be a variety. Then there is a commutative diagram
\begin{equation}
\begin{gathered}
\label{commutcomparmot}
\xymatrix{
K_0(\Var_X) \ar[r]^-{mC_*} \ar[d]_-\mu & G_0(X) \otimes \Z[y] \ar[d]^-s \\ 
K^{(2)}_{\mathrm{BM}, 0}(X) \ar[r]^-{\ch^{(2)}_\mathrm{BM}} & G_0(X)
}
\end{gathered}
\end{equation}
where the vertical map on the right is the quotient map
$$
G_0(X) \otimes \Z[y] \rightarrow G_0(X) \otimes \Z[y]/(y+1)=G_0(X). 
$$
\end{theorem}
\begin{proof}
By Proposition \ref{grrbmchernch} the  BM secondary Chern character satisfies a GRR theorem for pushforwards along proper maps. %and diagram \ref{grrmotchernbra}. 
Then, in view of the defining properties (1) and (2) of the motivic Chern class, to prove the claim  it is sufficient to verify the following two compatibilities. The first is that, if $X$ is smooth, diagram (\ref{commutcomparmot}) commutes when evaluated on $[X \xrightarrow{1_X} X].$ This holds, since 
$$\ch_{\mathrm{BM}}^{(2)} \circ \mu([X \xrightarrow{1_X} X]) =  \ch_{\mathrm{BM}}^{(2)}([\Perf(X)]) = [i_X^*\cO_{\cL X}] = \sum (-1)^ i \Omega^ i_X = s \circ mC_*([X \xrightarrow{1_X} X]).$$
 Finally, we need to check that if $f\colon Y \rightarrow X$ is a proper map, the square
 $$
 \xymatrix{
K_0(\Var_Y) \ar[r]^-{f_*} \ar[d]_-\mu & K_0(\Var_X)  \ar[d]^-{\mu} \\ 
K^{(2)}_{\mathrm{BM}, 0}(Y) \ar[r]^-{f_*} & K^{(2)}_{\mathrm{BM}, 0}(X)
}
 $$
 commutes. This is clear, and this concludes the proof. 
\end{proof}

%\begin{theorem}
%The square
% $$
% \xymatrix{
%\Var_Y \ar[r]^-{\ch_{mot}} \ar[d]_-\mu &  G_0(Y)  \\ 
%\iota_0 Cat(Y) \ar[r]^-{\ch} & \iota_0 \Coh(\cL Y) \ar[u]_{u}
%}
% $$
% commutes. Here the map $u \circ \ch $ is given by the canonical pushforward map $\iota_0 \Coh(\cL Y) \stackrel{j_*} \to \iota_0 \Coh(Y) \to G_0(Y) $.
%  
% \end{theorem}
% \begin{proof}
% Clearly the map $u \circ \ch $ satisfies the norm condition. Let $f: Y \rightarrow X$ be a proper map, then the commutativity of the following diagram 
% 
% 
%$$
% \xymatrix{
%\Coh(\cL X)   \ar[r]_{j_*}&  \Coh(X )    \ar[r] &  G_0(X)  \\ 
% \Coh(\cL Y)   \ar[u]_{\cL f_*}  \ar[r]_{j_*}  & \Coh(Y)  \ar[r] \ar[u]_{f_*} & G_0(Y)   \ar[u]_{f_*}
%}
% $$ together with the categorified GRR show that $ u \circ \ch$ satisfies the GRR condition. 
%This concludes the proof. 
%\end{proof} 
%
%
%Note also that there is a unique section $ G_0(X) \to G_0(X) \otimes \Z[y]$ to the map $ G_0(X) \otimes \Z[y] \to G_0(X) \simeq  G_0(X) \otimes \Z[y]/( y-1)$, which satisfies the norm and GRR condition.  
%
%

\section{The categorified Chern character and the de Rham realization}
\label{labelcatcherchar}
In this section we prove that the categorified Chern 
character recovers the de Rham realization. The main technical input will be 
the categorified GRR theorem. We will leverage work of Preygel on the comparison 
between coherent sheaves on the loop stack and crystals \cite{preygel2014ind}.

Throughout this section we will work over a fixed ground field 
$k$ of characteristic zero, and ``derived scheme'' will mean ``derived scheme almost of finite type over $k$''. We write $\Sch$ for the $\infty$-category of derived schemes.
If $X$ is a derived scheme, we denote by $\Sch_X$ the overcategory $\Sch_{/X}$.
Recall that a morphism of derived schemes $Y\to X$ is \emph{smooth} if for every classical scheme $Z$ and every morphism $Z\to X$, the projection $Y\times_XZ\to Z$ is a smooth morphism of classical schemes. We denote by $\Sm_X\subset \Sch_X$ the full subcategory of smooth $X$-schemes.

We will use heavily the theory of ind-coherent sheaves developed in \cite{gaitsgory2013ind,GR}. Recall that for $X$ a derived prestack (locally almost of finite type), there is defined a symmetric monoidal presentable stable $\infty$-category $\IndCoh(X)$, and for any morphism $f\colon Y\to X$ there is a symmetric monoidal pullback functor
\[
f^!\colon \IndCoh(X) \to \IndCoh(Y).
\]
If $f$ is schematic and quasi-compact (more generally, ind-inf-schematic), there is also a pushforward functor
\[
f_*\colon \IndCoh(Y) \to \IndCoh(X)
\]
with the following properties: if $f$ is proper (more generally, ind-proper), then $f_*$ is left adjoint to $f^!$, and if $f$ is an open immersion, then $f_*$ is right adjoint to $f^!$.
Furthermore, there is a canonical action of $\QCoh(X)$ on $\IndCoh(X)$, denoted by $\otimes$. The functor
\[
\Upsilon\colon \QCoh(X) \to \IndCoh(X), \quad \Upsilon(\cF) = \cF\otimes \omega_X,
\]
where $\omega_X\in\IndCoh(X)$ is the unit object,
is symmetric monoidal and intertwines the $*$-pullback of quasi-coherent sheaves and the $!$-pullback of ind-coherent sheaves.

If $X$ is a derived scheme, we have $\IndCoh(X)=\Ind(\Coh(X))$, where $\Coh(X)\subset\QCoh(X)$ is the subcategory of bounded pseudo-coherent complexes or, using the terminology of 
 \cite{SAG}, bounded  \emph{almost perfect} objects. 

\subsection{Ind-coherent sheaves on loop spaces and crystals}
In this section we review definitions and results from Preygel's article \cite{preygel2014ind}. %  We will use the terminology   adopted in \cite{preygel2014ind}, but explain how it compares to the usage in \cite{SAG}. 

 Let $\cC$ be a $k$-linear $\infty$-category equipped with an $S^1$-action. The invariant category 
$\cC^{S^1}$ is linear over %the cochains on $BS^1,$ which are  
$C^*(BS^1, k) \simeq k[[u]]$, with $u$ in (homological) degree $-2$, and we set \[\cC^{\mathrm{Tate}}:=\cC^{S^1} \otimes_{k[[u]]} k((u)).\] If $\cC$ is \emph{large} the Tate construction is often not quite the right concept. Under the assumption that $\cC$ is a stable $\infty$-category with a coherent
  t-structure \cite[Definition 4.2.7]{preygel2014ind},
   Preygel introduces the  \emph{tTate construction} $\cC^{\mathrm{tTate}}$ as a better behaved alternative. It is defined by
\[
\cC^\tTate := \Ind(\Coh(\cC)^{S^1})\otimes_{k[[u]]}k((u)),
\]
where $\Coh(\cC)\subset \cC$ is the full subcategory of bounded almost perfect objects\footnote{In \cite{preygel2014ind} Preygel calls almost perfect objects \emph{almost compact}. Note that in the terminology of \cite[Appendix C.5.5]{SAG}, $\Ind(\Coh(\cC))$ is the stabilization of the anticompletion of $\cC_{\geq 0}$.} and $(-)\otimes_{k[[u]]}k((u))$ is extension of scalars for presentable linear $\infty$-categories.
   %We refer the reader to \cite[Sections 1.3 and 4.1]{preygel2014ind} for a definition of coherent t-structures and of the tTate construction.
The notation $\Coh(\cC)$ is  motivated by geometric applications. Indeed, let $X$ be a derived scheme with a $S^1$-action.   If $\cC =  \Ind(\Coh(X))$ then $\Coh(\cC)$ coincides with the stable category of coherent complexes on $X$, $\Coh(X)$.  Further by  \cite[Remark 4.5.6]{preygel2014ind} we have that 
	\[
	\IndCoh(X)^\tTate \simeq \Ind(\Coh(X)^{S^1}) \otimes_{k[[u]]} k((u)) \simeq \Ind(\Coh(X)^\Tate).
	\]
  
Let $X$ be a derived scheme and $X_\dR$ the associated de Rham prestack.
Recall that a \emph{crystal} on $X$ is by definition a quasi-coherent sheaf on $X_\dR$, and that the functor
\[
\Upsilon \colon \QCoh(X_\dR) \to \IndCoh(X_\dR)
\]
is an equivalence \cite[Proposition 2.4.4]{GR2}.
The inclusion of the constant loops $X\to \cL X$ is a nil-isomorphism, and hence it induces an equivalence of de Rham prestacks. We therefore have an $S^1$-equivariant map
$$
\pi_X\colon \cL X \rightarrow (\cL X)_{\mathrm{dR}} \simeq X_{\mathrm{dR}},
$$
where the $S^1$-action is given by loop rotation on $\cL X$ and is trivial on $X_{\mathrm{dR}}$.  
The morphism $\pi_X$ is an inf-schematic nil-isomorphism and hence induces an adjunction
$$
\pi_{X,*}: \mathrm{IndCoh}(\cL X) \rightleftarrows \mathrm{IndCoh}(X_{\dR}): \pi_X^!,
$$
where the right adjoint is symmetric monoidal.

\begin{theorem}[\cite{preygel2014ind}, Theorem 1.3.5]
\label{Toly}
For every derived scheme $X$, the morphism $\pi_X$ induces inverse equivalences
of symmetric monoidal $\infty$-categories
$$
(\pi_{X,*})^{\mathrm{tTate}} : \mathrm{IndCoh}(\cL X)^{\mathrm{tTate}} \stackrel{\simeq}  \longleftrightarrow \mathrm{IndCoh}(X_{\dR})^{\mathrm{tTate}}  :  (\pi^!_X)^{\mathrm{tTate}}.
$$
\end{theorem}

\begin{definition}
If $\cC$ is a presentable stable $k$-linear $\infty$-category,  we denote by $\cC_{\mathbb{Z}/2}$ its 
 \defterm{$\mathbb{Z}/2$-folding}, 
 $$
 \cC_{\mathbb{Z}/2} := \cC \otimes_k k((u))
 $$
 where $u$ is in degree $-2$. 
 \end{definition}
 
 If $S^1$ acts trivially on a stable $\infty$-category $\cC$ with coherent t-structure, we have
 \[
 \cC^\tTate \simeq  \Ind(\Coh(\cC))_{\Z/2}
 \]
 \cite[Lemma 4.5.4]{preygel2014ind}.
 In particular, Theorem~\ref{Toly} gives equivalences of symmetric monoidal $\infty$-categories
 \begin{equation}\label{eqn:tTate-to-dR}
 	\IndCoh(\cL X)^\tTate \simeq \IndCoh(X_\dR)_{\Z/2} \simeq \QCoh(X_\dR)_{\Z/2}.
 \end{equation}
 %The canonical functor $\cC^{BS^1}\to \cC^\tTate\simeq \cC_{\Z/2}$ sends an $S^1$-equivariant object $M$ to $M^{S^1}\otimes$
 We will not distinguish notationally between an object of $\cC$ and its image in the $\Z/2$-folding $\cC_{\Z/2}$, as it will always be clear from the context which is meant.
 
 We now discuss the functoriality of the construction $\cC\mapsto \cC^\tTate$, following \cite[\S4.6]{preygel2014ind}. An exact functor $F\colon \cC\to\cD$ between stable $\infty$-categories with coherent t-structures is called \emph{coherent} if it is left t-exact up to a shift and $F|\cC_{<0}$ preserves filtered colimits.
  For such a functor there is an induced commutative diagram
  \[
  \xymatrix{
  \cC_{<\infty} \ar[r] \ar[d] & \Ind(\Coh(\cC)) \ar[r] \ar[d] & \cC \ar[d]^F \\
   \cD_{<\infty} \ar[r] & \Ind(\Coh(\cD)) \ar[r]& \cD,
  }
  \]
  where $\cC_{<\infty}=\bigcup_n \cC_{\leq n}\subset \cC$ is the subcategory of homologically bounded above objects.
  Note that the functors $f^!\colon \IndCoh(X)\to\< \IndCoh(Y)$ and $f_*\colon \IndCoh(Y)\to\<\IndCoh(X)$ are coherent for any morphism of derived schemes $f\colon Y\to X$.
  Similarly, if $\cC$ has a symmetric monoidal structure whose unit is bounded above and such that $x\otimes (-)$ is coherent for every $x\in\cC_{<0}$, there is an induced symmetric monoidal structure on $\Ind(\Coh(\cC))$ that restricts to the original one on $\cC_{<\infty}$.

  % If $F\colon \cC\to\cD$ is an $S^1$-equivariant coherent functor, we obtain in particular an induced functor in $\Prst$
%   \[
%   F^\tTate\colon \cC^\tTate \to \cD^\tTate.
%   \]
  If $\cC$ has an $S^1$-action, the diagram of functors
 \[
 \xymatrix{
 & \cC^{S^1}_{<\infty} \ar[dl] \ar[d] & \\
 \cC^{S^1} & \Ind(\Coh(\cC)^{S^1}) \ar[l] \ar[r] & \cC^\tTate
 }
 \]
 is thus natural in $\cC$ with respect to coherent functors. Replacing the functor $\Ind(\Coh(\cC)^{S^1}) \to\<  \cC^{S^1}$ by its right adjoint, we obtain a canonical functor
 \[
 \Theta \colon \cC^{S^1} \to \cC^\tTate,
 \]
 which is \emph{left-lax} natural in $\cC$, and whose restriction to $\cC^{S^1}_{<\infty}$ is strictly natural by Lemma 4.6.1 and Remark 4.6.3 of \cite{preygel2014ind}. 
 If $\cC$ is symmetric monoidal as before, the above diagram is one of symmetric monoidal functors. It follows that $\Theta$ is right-lax symmetric monoidal, and that its restriction to $\cC^{S^1}_{<\infty}$ is symmetric monoidal; in particular, $\Theta$ is unital. 
 %If the action of $S^1$ on $\cC$ is trivial, then $\cC^{S^1}$ is the $\infty$-category of $S^1$-equivariant objects of $\cC$ and $\cC^\tTate\simeq \cC_{\Z/2}$. In this case, the functor $\Theta$ sends an $S^1$-equivariant object $E$ to its Tate construction $E^{tS^1}=E^{hS^1}\otimes_{k[[u]]}k((u))$.
 
\begin{lemma}\label{lem:tS^1}
	Suppose that $\cC=\Ind(\Coh(\cC))$ and that $S^1$ acts trivially on $\cC$. Then the functor
	\[\Theta\colon \cC^{S^1} \to \cC^\tTate\simeq \cC_{\Z/2}\] sends an $S^1$-equivariant object $E$ to its Tate construction \[E^{tS^1}=E^{S^1}\otimes_{k[[u]]}k((u)).\]
\end{lemma}

\begin{proof}
	There is a commutative square of left adjoint functors
	 \[
	 \xymatrix{
	 \cC \ar[r] \ar@{=}[d] & \cC^{S^1} \\
	\Ind(\Coh(\cC)) \ar[r] & \Ind(\Coh(\cC)^{S^1}), \ar[u]
	 }
	 \]
	 where the horizontal functors equip an object with trivial $S^1$-action. Under the equivalence \[\Ind(\Coh(\cC)^{S^1})\simeq \cC\otimes_k k[[u]]\] of \cite[Lemma 4.5.4]{preygel2014ind}, the lower horizontal functor is given by extension of scalars. Passing to right adjoints, we deduce that the functor \[\cC^{S^1}\to \Ind(\Coh(\cC)^{S^1})\simeq \cC\otimes_k k[[u]]\] sends an object $E$ to its $S^1$-fixed points $E^{S^1}$ with their $k[[u]]$-module structure.
\end{proof}

  \subsection{The categorified de Rham   Chern character}
  
 For $X$ smooth over $k$, the categorified de Rham Chern character will be a functor
 \[
 \Ch^\dR\colon \Mod^\dual_{\QCoh(X)} \longrightarrow \cD_X\dmod_{\Z/2}
 \]
 associating to every dualizable sheaf of $\infty$-categories on $X$ a $2$-periodic $\cD_X$-module.
 More generally, for $X$ not necessarily smooth, we will define $\Ch^\dR$ as a functor valued in the $\infty$-category $\QCoh(X_\dR)_{\Z/2}$ of $2$-periodic crystals. The relationship between crystals and D-modules will be reviewed in \S\ref{deRhamComparison}.
  
% Let $X$ be an eventually coconnective scheme. Then there is a well-defined   functor
%   $$
%   \Xi: \mathrm{QCoh}(X) \to \mathrm{IndCoh}(X) %, \quad \cF   \mapsto \cF \otimes \omega_S
%   $$
% obtained by Ind-completing  the fully-faithful embedding (see \cite[Proposition 1.5.3]{gaitsgory2013ind})
% $$
% \Perf(X) \to \mathrm{Coh}(X).
% $$
% Note that if  $X$ is smooth, then its loop space $\cL X$ is also almost of finite type over $k$, quasi-compact, quasi-separated and eventually coconnective.

% Let $X$ be a derived scheme. Recall that $\IndCoh(X)$ is a module over $\QCoh(X)$ and that it admits a symmetric monoidal structure such that the functor
% \[
% \Upsilon\colon \QCoh(X) \to \IndCoh(X), \quad \Upsilon(\cF) = \cF\otimes \omega_X,
% \]
% is symmetric monoidal. For every $f\colon Y\to X$, the $!$-pullback functor
% \[
% f^!\colon \IndCoh(X)\to \IndCoh(Y)
% \]
%  is symmetric monoidal and $\QCoh(X)$-linear. In particular, if $p\colon X\to \mathrm{Spec}(k)$ is the structure map, then $\omega_X=p^!(k)$.

  \begin{definition}
\label{defdRCh}
Let $X$ be a derived scheme. The categorified \defterm{de Rham Chern character} is the functor
$$
\mathrm{Mod}^{\mathrm{dual}}_{\mathrm{QCoh}(X)} \xrightarrow{\mathrm{Ch}} \mathrm{QCoh}(\cL X)^{S^1} \xrightarrow{\Upsilon}\mathrm{IndCoh}(\cL X)^{S^1} \xrightarrow{\Theta} \mathrm{IndCoh}(\cL X)^{\tTate}   \simeq \QCoh(X_{\mathrm{dR}})_{\Z/2}, 
$$
where the last equivalence is~\eqref{eqn:tTate-to-dR}.
We denote the de Rham Chern character by $\mathrm{Ch}^{\mathrm{dR}}$.   
\end{definition}

Note that $\Ch^\dR$ is a unital right-lax symmetric monoidal functor, being a composition of such functors. Moreover, the restriction of $\Ch^\dR$ to fully dualizable $\QCoh(X)$-modules is strictly symmetric monoidal, since $\Upsilon\circ \Ch$ takes such modules to $\Coh(\cL X)^{S^1}$.

When $X$ is a smooth scheme, $\Ch^\dR$ is an enhancement of periodic cyclic homology:
 
 \begin{lemma}\label{lem:Ch=HP}
 	Let $X$ be a smooth scheme and let $\pi\colon X\to X_\dR$ be the canonical map. Then the composite functor
 	\[
 	\mathrm{Mod}^{\mathrm{dual}}_{\mathrm{QCoh}(X)} \xrightarrow{\Ch^\dR} \QCoh(X_\dR)_{\Z/2} \xrightarrow{\pi^*} \QCoh(X)_{\Z/2}
 	\]
 	sends $\cM$ to its relative periodic cyclic homology $\mathrm{HP}(\cM/X)=\HH(\cM/X)^{tS^1}$. 
	%Moreover, $\Ch^\dR$ commutes with restriction to any open subscheme $U\subset X$.
 \end{lemma}

 \begin{proof}
	Let $e\colon X\to \cL X$ be the inclusion of the constant loops. Since $e$ is proper, the functor $e^!$ admits a coherent left adjoint $e_*$, so that it commutes with $\Theta$.
 	Since $X$ is smooth, the functor $\Upsilon\colon \QCoh(X)\to \IndCoh(X)$ is an equivalence.
 	Using these facts, one can identify $\pi^*\circ \Ch^\dR$ with the composition
 	\[
 	\Mod^\dual_{\QCoh(X)} \xrightarrow{\Ch} \QCoh(\cL X)^{S^1} \xrightarrow{e^*} \QCoh(X)^{S^1} \xrightarrow{\Theta} \QCoh(X)_{\Z/2}.
 	\]
 	By definition, 
	$\Ch(\cM)$ is the trace of the monodromy automorphism of $p^*\cM$, where $p\colon \cL X\to\< X$.
	Since $p\circ e=\id_X$, $e^*(\Ch(\cM))$ is the trace of the identity on $\cM$, that is, the Hochschild homology $\HH(\cM/X)$ with its canonical $S^1$-action.
	 % $\HH(\cM/k)$ with its canonical $S^1$-equivariant $\HH(R/k)$-module structure. Hence,
% 	$e^*(\Ch(\cM))$ is the relative Hochschild homology $\HH(\cM/R)$ with its $S^1$-action.
 	By Lemma~\ref{lem:tS^1}, the last functor sends an $S^1$-equivariant object $E$ to its Tate fixed points $E^{tS^1}$, which completes the proof. 
 \end{proof}
 
\begin{remark}\label{rem:ChdR-Mot}
	The functor $\Ch^\dR$ sends localization sequences of dualizable $\QCoh(X)$-modules to cofiber sequences, and its restriction to compactly generated $\QCoh(X)$-modules extends to a functor
	\[\Ch^\dR\colon \Mot(X)=\MOT(\Perf(X)) \to \QCoh(X_\dR)_{\Z/2}\] (since $\Ch$ has both properties). However, Lemma~\ref{lem:Ch=HP} shows that $\Ch^\dR$ does not preserve filtered colimits, so that it is not a localizing invariant in the sense of \cite[Definition 5.16]{HSS}.
\end{remark}

We now investigate the naturality properties of the categorified de Rham Chern character.
If $f\colon Y\to X$ is a morphism of derived schemes, we have a commutative diagram
\begin{equation}\label{ChdR-pullbacks}
\xymatrix@C=15pt{
\mathrm{Mod}^{\mathrm{dual}}_{\QCoh(X)} \ar[r]^-{\Ch} %\ar @{} [dr] |{A} 
\ar[d]_{f^*}
& \QCoh(\cL X)^{S^1} \ar[r]^-{\Upsilon} \ar[d]_-{\cL f^*} & \mathrm{IndCoh}(\cL X)^{S^1} \ar[r]^-{\Theta} \ar[d]_{\cL f^!} & \mathrm{IndCoh}(\cL X)^{\mathrm{tTate}} \twocell{dl} \ar[r]^-{\simeq} \ar[d]_ -{\cL f^!} & 
\mathrm{QCoh}(X_\dR)_{\mathbb{Z}/2} \ar[d]_ -{f_{\dR}^*} \\ 
\mathrm{Mod}^{\mathrm{dual}}_{\QCoh(Y)} \ar[r]^-{\Ch}  & \QCoh(\cL Y)^{S^1} \ar[r]^-{\Upsilon} & \mathrm{IndCoh}(\cL Y)^{S^1} \ar[r]^-{\Theta} & \mathrm{IndCoh}(\cL Y)^{\mathrm{tTate}} \ar[r]^-{\simeq}  & 
\mathrm{QCoh}(Y_\dR)_{\mathbb{Z}/2}.
}
\end{equation}
(For the last square, recall that the horizontal equivalences are inverse to $\pi^!\circ\Upsilon$.)
The 2-cell is invertible if $f$ is proper, since in this case $\cL f^!$ admits a coherent left adjoint.
In fact, each of the component functors of $\Ch^\dR$ is natural on $\Sch^\op$ (for $\Upsilon$, see \cite[II.3.3.2.5]{GR}), except $\Theta$ which is left-lax natural. Hence, $\Ch^\dR$ can be promoted to a left-lax natural transformation
\begin{equation*}\label{eqn:ChdR}
\Ch^\dR\colon \Mod^{\dual}_{\QCoh(-)} \Rightarrow \QCoh((-)_\dR)_{\Z/2}\colon \Sch^\op \to \Cat_{(\infty,1)},
\end{equation*}
which is strictly natural for proper morphisms, and whose restriction to fully dualizable modules is strictly natural.
%where $\CAlg_{\mathrm{rlax}}(\Cat_{(\infty,1)})$ is the $\infty$-category of symmetric monoidal $\infty$-categories and right-lax symmetric monoidal functors.
For any morphism $f\colon Y\to X$, we obtain by passing to right adjoints a canonical transformation
\[
\Ch^\dR \circ f_* \Rightarrow f_{\dR,*} \circ \Ch^\dR.
\]
Following \cite[Definition 7.3.2]{gaitsgory2011ind} we say that a morphism $f\colon Y\to X$ is \defterm{Gorenstein} if it is eventually coconnective, locally almost of finite type, and the image of $\mathcal{O}_X$ under 
$$f^!: \QCoh(X) \to \QCoh(Y)$$ is a graded line bundle. 
\begin{lemma}\label{lem:Gorenstein}
	Let $f\colon Y\to X$ be a smooth morphism of derived schemes. Then the morphism $\cL f\colon \cL Y\to \cL X$ is quasi-smooth and in particular Gorenstein.
\end{lemma}

\begin{proof}
	The morphism $\cL f$ can be factored as
	\[
	\cL Y \to \cL X\times_{X} Y \to \cL X,
	\]
	where the first morphism is a base change of the diagonal $f$ and the second is a base change of $f$. If $f$ is smooth, both $f$ and its diagonal are quasi-smooth, and the result follows.
\end{proof}

\begin{theorem}\label{deRhamGRR}
	Suppose $f\colon Y\to X$ is a \emph{smooth} morphism of derived schemes. Then the diagram
	\begin{equation*}
	\xymatrix{
	\Mod^\dual_{\QCoh(Y)} \ar^{f_*}[d] \ar^-{\Ch^\dR}[r] & \QCoh(Y_\dR)_{\Z/2} \ar^{f_{\dR,*}}[d] \\
	\Mod^\dual_{\QCoh(X)} \ar^-{\Ch^\dR}[r] \twocell{ur} & \QCoh(X_\dR)_{\Z/2}
	}
	\end{equation*}
	commutes strictly.
\end{theorem}

\begin{proof}
	We show that each square in~\eqref{ChdR-pullbacks} is right-adjointable. For the first square, this is Theorem~\ref{thm:GRR}. The assumption that $f$ is smooth implies that $\cL f$ is Gorenstein (Lemma~\ref{lem:Gorenstein}). 
	By \cite[Proposition 7.3.8]{gaitsgory2013ind}, it follows that the functor $\cL f^!$ admits a right adjoint given by
	\begin{equation}\label{eqn:Gorenstein-adj}
	\cF \mapsto \cL f_*(\cK_{\cL f}^{-1}\otimes \cF),
	\end{equation}
	where $\cK_{\cL f}\in \QCoh(\cL Y)$ is the relative dualizing sheaf. 
	% In particular, we have
% 	\[
% 	\cK_{\cL f} \otimes \cL f^*(\omega_{\cL X}) \simeq \omega_{\cL Y}
% 	\]
% 	in $\IndCoh(\cL Y)$, where $\cL f^*$ is the left adjoint to $\cL f_*$ on $\IndCoh$.
	 The right adjointability of the second square is thus the statement that the canonical map
	\[
	\Upsilon (\cL f_*(\cF)) \to \cL f_*(\cK_{\cL f}^{-1}\otimes \Upsilon(\cF))
	\]
	is an equivalence for every $\cF\in \QCoh(\cL Y)$. Since $\omega_{\cL Y} \simeq \cK_{\cL f}\otimes \cL f^*(\omega_{\cL X})$ in $\IndCoh(\cL Y)$, we can identify this map with the canonical map
	\[
	\cL f_*(\cF) \otimes \omega_{\cL X} \to \cL f_*(\cF \otimes \cL f^*(\omega_{\cL X})),
	\]
	which is indeed an equivalence by \cite[Proposition II.1.3.3.7]{GR}.
	
	For the third square in~\eqref{ChdR-pullbacks}, first note that the functor~\eqref{eqn:Gorenstein-adj} preserves colimits and is left t-exact up to a shift, so that it induces a functor between the tTate constructions which is right adjoint to $(\cL f^!)^\tTate$. Moreover, since the pullback functors commute strictly with the functors $\Ind(\Coh(\cL(-))^{S^1}) \to \IndCoh(\cL(-))^{S^1}$, their right adjoints commute strictly with $\Theta$.
	 Finally, the last square is trivially right-adjointable, since its horizontal maps are equivalences.
\end{proof}

Next, we show that the categorified de Rham Chern character is $\A^1$-homotopy invariant. We start with a lemma generalizing the homotopy invariance of periodic cyclic homology, first proved by Kassel \cite[\S3]{Kassel}.

\begin{lemma}\label{lem:HP-invariance}
	Let $\cC$ be a $k$-linear symmetric monoidal $\infty$-category. For any $E\in \cC^{S^1}$, the map \[E\to E\otimes_k \HH(k[t]/k)\] induces an equivalence on Tate $S^1$-fixed points.
\end{lemma}

\begin{proof}
	Since Tate fixed points vanish on the image of the left adjoint to the forgetful functor $\cC^{S^1} \to \cC$ \cite[Corollary 10.2]{Klein}, it will suffice to show that the cofiber of $k\to \HH(k[t]/k)$ is induced from the trivial subgroup of $S^1$.
		 Since $k$ has characteristic zero, $\HH(k[t]/k)\in \Mod_k^{S^1}$ is the free simplicial commutative $k$-algebra on the $S^1$-equivariant $k$-module $k[S^1]$:
		 \[
		 \HH(k[t]/k) \simeq \bigoplus_{n\geq 0}\Sym^n_k(k[S^1]) \simeq \bigoplus_{n\geq 0} k[\Sym^n S^1].
		 \]
		 Here, $\Sym^n_k$ is the symmetric power defined in \cite[\S25.2.2]{SAG}, and $\Sym^n$ is the ``strict'' symmetric power of spaces. The second equivalence holds because $\Sym^n_k$ and $\Sym^n$ are left Kan extended from their restrictions to finite free $k$-modules and finite sets, respectively (for the latter, see \cite[\S2]{HoyoisKunneth}).
		  Thus, it suffices to show that $k[\Sym^n S^1]$ is induced for all $n\geq 1$. It is easy to check that the $S^1$-equivariant map $\Sym^n S^1\to S^1/C_n$, $(z_1,\dotsc,z_n)\mapsto \sqrt[n]{z_1\dots z_n}$, is an equivalence \cite{Morton}. Using again that $k$ has characteristic zero, the map $k[S^1] \to k[S^1/C_n]$ is an equivalence for all $n$. This concludes the proof.
\end{proof}

\begin{proposition}[Homotopy invariance]
	 \label{prop:A1-invariance}
	Let $X$ be a derived scheme. For every dualizable $\QCoh(X)$-module $\cM$, the map
	 \[
	 \Ch^\dR(\cM) \to \Ch^\dR(\cM\otimes_{\QCoh(X)}\QCoh(\A^1_X))
	 \]
	 induced by the projection $\A^1_X\to X$ is an equivalence in $\QCoh(X_\dR)_{\Z/2}$.
\end{proposition}

\begin{proof}
	 Since crystals satisfy h-descent \cite[Proposition 3.2.2]{GR2} and $\Ch^\dR$ is strictly natural with respect to proper maps, we can assume that $X$ is a smooth scheme. By Lemma~\ref{lem:Ch=HP} and the conservativity of the forgetful functor $\QCoh(X_\dR)\to \QCoh(X)$ \cite[Lemma 2.2.6]{GR2}, we are reduced to proving that $\mathrm{HP}(-/X)\colon \Mod_{\QCoh(X)}^\dual \to \QCoh(X)$ is homotopy invariant. This is a special case of Lemma~\ref{lem:HP-invariance}.
\end{proof}

 \subsection{de Rham  Chern character and de Rham realization}
\label{deRhamComparison}

In this section we explain how 
Theorem~\ref{deRhamGRR} implies a comparison between the de Rham Chern character and the classical de Rham realization.
We first explain what we mean by the latter.

The functor $\QCoh((-)_\dR)\colon \Sch^\op\to \Cat_{(\infty,1)}$ classifies a coCartesian fibration
\[
\cQ\to \Sch^\op.
\]
Since the pullbacks $f_{\dR}^*$ admit right adjoints, it is also a \emph{Cartesian} fibration over $\Sch^\op$. For $X\in\Sch$, let $\cQ_{/X}$ denote the restriction of $\cQ$ to $\Sch_X^\op$. Since $X$ is an initial object in $\Sch_X^\op$, the inclusion of the fiber over $X$
\[
 \QCoh(X_\dR) \hookrightarrow \cQ_{/X}
\]
is fully faithful and admits a right adjoint sending any object to its pushforward to $X$.

\begin{definition}
	Let $X$ be a derived scheme.
	The \defterm{de Rham realization} 
	\[
	\dR_X\colon \Sch_X^\op \to \QCoh(X_\dR)
	\]
	is the composition of the unit section $\Sch_X^\op \to \cQ_{/X}$ and the right adjoint to the inclusion.
	% More informally,
% 	\[
% 	\dR_X(f\colon Y\to X) = f_{\dR,*}(\cO_{Y_\dR}) \in \QCoh(X_\dR).
% 	\]
\end{definition}

By definition, the de Rham realization $\dR_X$ sends an $X$-scheme $f\colon Y\to X$ to the crystal \[f_{\dR,*}(\cO_{Y_\dR}) \in \QCoh(X_\dR).\]

Similarly, consider the coCartesian fibrations
\[
\cM\to \Sch^\op\quad\text{and}\quad \cQ_{\Z/2}\to \Sch^\op
\]
classified by the functors $\Mod^\dual_{\QCoh(-)}$ and $\QCoh((-)_\dR)_{\Z/2}$. Both are also Cartesian fibrations, and hence the inclusions
\[
\Mod^\dual_{\QCoh(X)} \hookrightarrow \cM_{/X} \quad\text{and}\quad \QCoh(X_\dR)_{\Z/2} \hookrightarrow (\cQ_{\Z/2})_{/X}
\]
admit right adjoints. The left-lax natural transformation $\Ch^\dR$ classifies a morphism $\cM \to\< \cQ_{\Z/2}$ over $\Sch^\op$.
The commutative square
\[
\xymatrix{
\cM_{/X}  \ar^-{\Ch^\dR}[r] & (\cQ_{\Z/2})_{/X}  \\
\Mod^\dual_{\QCoh(X)} \ar[u] \ar^-{\Ch^\dR}[r] & \QCoh(X_\dR)_{\Z/2} \ar[u]
}
\]
induces by adjunction a 2-cell
\begin{equation}\label{deRham2cell}
\xymatrix{
\cM_{/X} \ar[d] \ar^-{\Ch^\dR}[r] & (\cQ_{\Z/2})_{/X} \ar[d] \\
\Mod^\dual_{\QCoh(X)} \ar^-{\Ch^\dR}[r] \twocell{ur} & \QCoh(X_\dR)_{\Z/2}.
}
\end{equation}
Given $f\colon Y\to X$ and $\cC\in \Mod^\dual_{\QCoh(Y)}$, the component of this 2-cell at $\cC$ is the canonical map
\[
\Ch^\dR(f_*\cC) \to f_{\dR,*}\Ch^\dR(\cC).
\]
In particular, by Theorem~\ref{deRhamGRR}, it is an equivalence if $f$ is smooth. 

Precomposing the 2-cell~\eqref{deRham2cell} with the unit section
\[
\Sch_X^\op \to \cM_{/X},\quad Y\mapsto \QCoh(Y) \in \Mod^\dual_{\QCoh(Y)},
\]
we obtain a natural transformation
\[
\Ch^\dR \circ \QCoh_X \Rightarrow \dR_X \colon \Sch_X^\op \to \QCoh(X_\dR)_{\Z/2}
\]
comparing the categorified de Rham Chern character with the $\Z/2$-folding of the classical de Rham realization.
By Theorem~\ref{deRhamGRR}, it restricts to an equivalence
\[
\Ch^\dR \circ \QCoh_X \simeq \dR_X \colon \Sm_{X}^\op \to \QCoh(X_\dR)_{\Z/2}
\]
on the category of smooth $X$-schemes. We state this as the next result:

\begin{theorem}
\label{Ch-dR-comparison}
Let $X$ be a derived scheme. Then there is a commutative diagram
\[
\xymatrix{
\mathrm{Sm}_{X}^{\op}  \ar[r]^-{\dR_X} \ar[d]_-{\QCoh_X} & \QCoh(X_\dR)_{\mathbb{Z}/2}.
  \\ 
\mathrm{Mod}^{\mathrm{dual}  }_{\QCoh(X)} \ar[ur]_ - {\Ch^\dR} &
}
\]
\end{theorem} 

%This theorem implies in particular that the de Rham realization $\dR_X$ is a symmetric monoidal functor on $\Sm_X^\op$, since $\QCoh_X$ is \cite[Corollary 9.4.3.8]{SAG}.

Suppose now that $X$ is a smooth scheme. In this case, we can rephrase Theorem~\ref{Ch-dR-comparison} in the more classical language of D-modules. Let $\cD_X$ denote the quasi-coherent sheaf of differential operators from $\cO_X$ to $\cO_X$, viewed as an algebra object of $\QCoh(X)^\heartsuit$ under composition $\circ$. 
Let $\pi\colon X\to X_\dR$ be the canonical map. As shown in \cite[\S5.4]{GR2}, the forgetful functor
\[
\pi^*\colon \QCoh(X_\dR) \to \QCoh(X)
\]
is monadic and the corresponding monad on $\QCoh(X)$ can be identified with $\cD_X\otimes(-)$. We therefore have an equivalence
\[
\QCoh(X_\dR) \simeq \cD_X\dmod,
\]
where $\cD_X\dmod$ is the $\infty$-category of left $\cD_X$-modules in $\QCoh(X)$. By \cite[Lemma III.4.4.1.6]{GR}, for any morphism $f\colon Y\to X$ of smooth schemes, there is a commutative square
\[
\xymatrix{
\QCoh(X_\dR)  \ar^{f_{\dR}^*}[r] \ar_{\simeq}[d] & \QCoh(Y_\dR) \ar^{\simeq}[d]  \\
\cD_X\dmod \ar^{f^\circ}[r] & \cD_Y\dmod,
}
\]
where $f^\circ$ is the naive pullback of left D-modules \cite[VI \S4.1]{Borel}.
By adjunction, there is a commutative square
\[
\xymatrix{
\QCoh(Y_\dR)  \ar^{f_{\dR,*}}[r] \ar_{\simeq}[d] & \QCoh(X_\dR) \ar^{\simeq}[d]  \\
\cD_Y\dmod \ar^{f_\circ}[r] & \cD_X\dmod,
}
\]
where $f_\circ$ is right adjoint to $f^\circ$. If $f$ is smooth and $\cM\in\cD_Y\dmod$, $f_\circ(\cM)$ is the quasi-coherent sheaf
\[
f_*(\cM\otimes_{\cO_Y}\Omega_{Y/X}^{-\bullet})
\]
equipped with the \emph{Gauss--Manin connection} (here, $\Omega_{Y/X}^{-\bullet}$ is the relative de Rham complex, viewed as an object in $\QCoh(Y)_{\leq 0}$).\footnote{If $f$ is smooth of relative dimension $d$, the usual pushforward of left D-modules, denoted by $f_+$ in \cite[VI, \S5]{Borel}, by $\int_f$ in \cite{hotta2007d}, and by $f_*$ in \cite[\S3]{drew2013realisations}, sends $\cM$ to $f_*(\cM\otimes_{\cO_Y}\Omega_{Y/X}^{-\bullet})[d]$, but it is right adjoint to $f^\circ[-d]$.} In particular, the de Rham realization $\dR_X$ sends a smooth morphism $f\colon Y\to X$ to its relative de Rham cohomology $f_*(\Omega_{Y/X}^{-\bullet})$ equipped with the Gauss--Manin connection.

\begin{remark}
Theorem \ref{Ch-dR-comparison} implies in particular that, up to $\mathbb{Z}/2$-folding, the Gauss--Manin connection on the cohomology of the fibers of a smooth morphism $f\colon Y \to X$ is of \emph{non-commutative origin}. That is, it only depends on $\QCoh(Y)$ and its $\QCoh(X)$-linear structure. 
\end{remark}

\begin{remark}
	Let $X$ be a smooth scheme. Then the categorified de Rham Chern character
	\[\Ch^\dR\colon \Mot(X)\to \cD_X\dmod_{\Z/2}\] (see Remark~\ref{rem:ChdR-Mot}) is a categorification of the classical Chern character with values in de Rham cohomology. Indeed, on endomorphisms of the unit objects, it gives a morphism of $\bE_\infty$ ring spectra
	\[
	\ch^\dR\colon\mathbb K(X) \to \prod_{n\in \Z} H^{*+2n}_\dR(X),
	\]
	which is the composition of the Dennis trace map $\mathbb K(X) \to \HH(X/k)^{S^1}$ (see  \cite[Remark 6.12]{HSS}) and the canonical map $\HH(X/k)^{S^1} \to \HH(X/k)^{tS^1} \simeq \mathrm{HP}(X/k)$.
\end{remark}

 \begin{remark}
We state explicitly an important  special case  of  the categorified de Rham Chern character.  
Let $X=\mathrm{Spec}(R)$ be a smooth affine scheme.  
Recall from Lemma~\ref{lem:Ch=HP} that the composite
\begin{equation*}
\label{factthroughch}
\mathrm{Mod}^{\mathrm{dual}}_{\mathrm{QCoh}(X)} \xrightarrow{\mathrm{Ch}^{\mathrm{dR}}} \cD_X\dmod_{\mathbb{Z}/2} \xrightarrow{\mathrm{forget}} \mathrm{QCoh}(X)_{\mathbb{Z}/2}
\end{equation*} 
maps $\cM \in \mathrm{Mod}^{\mathrm{dual}}_{\mathrm{QCoh}(X)}$ to its relative periodic cyclic homology  $\mathrm{HP}(\cM/R)$ viewed as an $R$-module. The fact that $\mathrm{HP}(-/R)$ factors  through the $\infty$-category of D-modules implies that $\mathrm{HP}(\cM/R)$ carries a flat connection, which is a non-commutative analog of the Gauss--Manin connection. 
 
If $A$ is an $R$-algebra, a construction of the Gauss--Manin  connection on $\mathrm{HP}(A/R)$ was proposed by Getzler in \cite{getzler1993cartan}. Our construction has 
the  advantage that it applies to all dualizable sheaves of categories over $X,$ and not just to modules over a sheaf of algebras $A$ over $X$.  We believe, but do not prove, that the categorified de Rham Chern character matches  Getzler's construction in the cases where they overlap. We will return to this question in future work. 
 \end{remark}
 
% \begin{remark}
% 	Let $X$ be a derived scheme. By Proposition~\ref{prop:A1-invariance}, the categorified de Rham Chern character $\Ch^\dR\colon \Mot(X) \to \QCoh(X_\dR)_{\Z/2}$ factors through the $\A^1$-localization:
% 	\[
% 	\Ch^\dR\colon \Mot_{\A^1}(X) \to \QCoh(X_\dR)_{\Z/2}.
% 	\]
% 	Let $\SH(X)$ be Voevodsky's stable $\infty$-category of motivic spectra over $X$ (see \cite{KhanThesis} for the definition in the derived context), and let $\SH(X)^\vee=\Ind(\SH(X)^{\omega,\op})$ be its dual in $\Prst$.
% 	By the universal property of $\SH(X)^\vee$ (which is dual to \cite[Corollary 2.39]{robalo}), the symmetric monoidal functor \cite[Corollary 9.4.3.8]{SAG}
% 		\[
% 		\Sm_X^\op \xrightarrow{\QCoh_X} \Mot_{\A^1}(X)
% 		\]
% 	extends uniquely to a colimit-preserving symmetric monoidal functor
% 	\[
% 	\SH(X)^\vee \to \Mot_{\A^1}(X).
% 	\]
% 	By Theorem~\ref{Ch-dR-comparison}, the composition
% 	\[
% 	\SH(X)^\vee \to \Mot_{\A^1}(X) \xrightarrow{\Ch^\dR} \QCoh(X_\dR)_{\Z/2}
% 	\]
% 	extends the de Rham realization $\dR_X$. This justifies and generalizes \cite[Remark 1.6]{HSS} to non-smooth schemes.
% \end{remark}

\begin{remark}
	Suppose $X$ smooth and quasi-projective over $k$. In \cite[Theorem 3.3.9]{drew2013realisations}, Drew constructs a de Rham realization functor
	\[
	\rho_\dR\colon \mathrm{SH}(X) \to \cD_X^\mathrm{h}\dmod
	\]
	where 
	$\mathrm{SH}(X)$ is the stable motivic homotopy $\infty$-category over $X$ and 
	$\cD_X^\mathrm{h}\dmod\subset \cD_X\dmod$ is the full subcategory of holonomic left $\cD_X$-modules. Consider the composite functor
	\[
	\dR_X'\colon \Sch^\mathrm{cl}_{X} \xrightarrow{\mathrm M} \mathrm{SH}(X)^\omega \xrightarrow{\rho_\dR} \cD_X\dmod^\omega \stackrel{\mathbb D}\simeq (\cD_X\dmod^\omega)^\op ,
	\]
	where $\Sch^\mathrm{cl}_X$ is the category of classical $X$-schemes of finite type, $\mathrm M(f\colon Y\to X) = f_!f^!(\mathbf 1_X)$, and $\mathbb D$ is Verdier duality. Using the compatibility of $\rho_\dR$ with the six operations proved by Drew, one can show that, if $Y$ is smooth quasi-projective and $f\colon Y\to X$ is arbitrary,
	\[
	\dR_X'(f\colon Y\to X) \simeq \dR_X(f\colon Y\to X)[\dim(X)].
	\]
	By inspecting the definition of $\rho_\dR$, it is not difficult to show that in fact there is an equivalence of functors $\dR_X'\simeq \dR_X[\dim(X)]\colon \Sm_X^\op\to \cD_X\dmod$. In other words, our de Rham realization is, up to a shift, Verdier dual to Drew's de Rham realization.
\end{remark}

\bibliographystyle{alphamod}

\bibliography{chern}

\providecommand{\bysame}{\leavevmode\hbox to3em{\hrulefill}\thinspace}
\providecommand{\MR}{\relax\ifhmode\unskip\space\fi MR }
\providecommand{\MRhref}[2]{%
  \href{http://www.ams.org/mathscinet-getitem?mr=#1}{#2}
}
\providecommand{\href}[2]{#2}
\begin{thebibliography}{BZN13b}
\providecommand{\url}[1]{\href{#1}{{\def~{\textasciitilde}\tt #1}}}

\bibitem[Bar15]{barwick2015exact}
C.~Barwick, \emph{On exact {$\infty$}-categories and the theorem of the heart},
  Compos. Math. \textbf{151} (2015), no.~11, pp.~2160--2186

\bibitem[BDR04]{BDR}
N.~A. Baas, B.~I. Dundas, and J.~Rognes, \emph{Two-vector bundles and forms of
  elliptic cohomology}, Topology, Geometry and Quantum Field Theory
  (U.~Tillmann, ed.), Cambridge University Press, 2004, pp.~18--45

\bibitem[BGT13]{BGT}
A.~Blumberg, D.~Gepner, and G.~Tabuada, \emph{A universal characterization of
  higher algebraic K-theory}, Geom. Topol. \textbf{17} (2013), no.~2,
  pp.~733--838

\bibitem[Bit04]{bittner2004universal}
F.~Bittner, \emph{The universal Euler characteristic for varieties of
  characteristic zero}, Compositio Mathematica \textbf{140} (2004), no.~04,
  pp.~1011--1032

\bibitem[BK90]{bondal1990representable}
A.~I. Bondal and M.~M. Kapranov, \emph{Representable functors, Serre functors,
  and mutations}, Mathematics of the USSR-Izvestiya \textbf{35} (1990), no.~3,
  p.~519

\bibitem[BLL04]{BLL}
A.~Bondal, M.~Larsen, and V.~Lunts, \emph{Grothendieck ring of pretriangulated
  categories}, Int. Math. Res. Notices \textbf{29} (2004), pp.~1461--1495

\bibitem[Bor87]{Borel}
A.~Borel, \emph{Algebraic D-modules}, Perspectives in mathematics, Academic
  Press, 1987

\bibitem[BSY10]{brasselet2005hirzebruch}
J.-P. Brasselet, J.~Sch\"urmann, and S.~Yokura, \emph{Hirzebruch classes and
  motivic {C}hern classes for singular spaces}, J. Topol. Anal. \textbf{2}
  (2010), no.~1, pp.~1--55

\bibitem[BZN12]{BNlsc}
D.~Ben-Zvi and D.~Nadler, \emph{Loop spaces and connections}, J. Topol.
  \textbf{5} (2012), no.~2, pp.~377--430

\bibitem[BZN13a]{BN3}
D.~Ben-Zvi and D.~Nadler, \emph{Nonlinear traces}, 2013,
  \href{http://arxiv.org/abs/1305.7175v3}{{\sf arXiv:1305.7175v3}}

\bibitem[BZN13b]{BN2}
\bysame, \emph{Secondary traces}, 2013,
  \href{http://arxiv.org/abs/1305.7177v3}{{\sf arXiv:1305.7177v3}}

\bibitem[CS91]{cappell1991stratifiable}
S.~E. Cappell and J.~L. Shaneson, \emph{Stratifiable maps and topological
  invariants}, Journal of the American Mathematical Society \textbf{4} (1991),
  no.~3, pp.~521--551

\bibitem[CT12]{CT}
D.-C. Cisinski and G.~Tabuada, \emph{Symmetric monoidal structures in
  non-commutative motives}, J. K-Theory \textbf{9} (2012), pp.~201--268

\bibitem[Dre13]{drew2013realisations}
B.~Drew, \emph{R{\'e}alisations tannakiennes des motifs mixtes triangul{\'e}s},
  Th{\`e}se de Doctorat, U. Paris 13 (2013)

\bibitem[Gai11]{gaitsgory2011ind}
D.~Gaitsgory, \emph{Ind-coherent sheaves}, arXiv preprint arXiv:1105.4857
  (2011)

\bibitem[Gai13]{gaitsgory2013ind}
\bysame, \emph{Ind-Coherent Sheaves}, Moscow Mathematical Journal \textbf{13}
  (2013), no.~3, pp.~399--528

\bibitem[Gai15]{Ga1}
\bysame, \emph{Sheaves of categories and the notion of 1-affineness}, Stacks
  and categories in geometry, topology, and algebra, Contemp. Math., vol. 643,
  Amer. Math. Soc., Providence, RI, 2015, pp.~127--225

\bibitem[Get93]{getzler1993cartan}
E.~Getzler, \emph{Cartan homotopy formulas and the Gauss--Manin connection in
  cyclic homology}, Israel Math. Conf. Proc, vol.~7, 1993, pp.~65--78

\bibitem[GK08]{Ganter20082268}
N.~Ganter and M.~Kapranov, \emph{Representation and character theory in
  2-categories}, Adv. Math. \textbf{217} (2008), no.~5, pp.~2268 -- 2300

\bibitem[GR14]{GR2}
D.~Gaitsgory and N.~Rozenblyum, \emph{Crystals and D-modules}, Pure Appl. Math.
  Q. \textbf{10} (2014), no.~1, pp.~57--154

\bibitem[GR16]{GR}
\bysame, \emph{A study in derived algebraic geometry}, book project, 2016,
  \url{http://www.math.harvard.edu/~gaitsgde/GL/}

\bibitem[Gur12]{Gurski}
N.~Gurski, \emph{Biequivalences in tricategories}, Theory Appl. Categ.
  \textbf{26} (2012), no.~14, pp.~349--384

\bibitem[HKR92]{hopkins1992morava}
M.~J. Hopkins, N.~J. Kuhn, and D.~C. Ravenel, \emph{Morava K-theories of
  classifying spaces and generalized characters for finite groups}, Algebraic
  Topology Homotopy and Group Cohomology, Springer Berlin Heidelberg, 1992,
  pp.~186--209

\bibitem[HNR19]{hall2019one}
J.~Hall, A.~Neeman, and D.~Rydh, \emph{One positive and two negative results
  for derived categories of algebraic stacks}, Journal of the Institute of
  Mathematics of Jussieu \textbf{18} (2019), no.~5, pp.~1087--1111

\bibitem[Hoy18]{HoyoisKunneth}
M.~Hoyois, \emph{The {\'e}tale symmetric K{\"u}nneth theorem}, 2018,
  \href{http://arxiv.org/abs/1810.00351v2}{{\sf arXiv:1810.00351v2}}

\bibitem[HSS17]{HSS}
M.~Hoyois, S.~Scherotzke, and N.~Sibilla, \emph{Higher traces, noncommutative
  motives, and the categorified {C}hern character}, Adv. Math. \textbf{309}
  (2017), pp.~97--154

\bibitem[HTT08]{hotta2007d}
R.~Hotta, K.~Takeuchi, and T.~Tanisaki, \emph{{$D$}-modules, perverse sheaves,
  and representation theory}, Progress in Mathematics, vol. 236, Birkh\"auser
  Boston, Inc., Boston, MA, 2008

\bibitem[JFS15]{JFS}
T.~Johnson-Freyd and C.~Scheimbauer, \emph{(Op)lax natural transformations,
  relative field theories, and the ``even higher'' Morita category of
  $E_d$-algebras}, 2015, \href{http://arxiv.org/abs/1502.06526v2}{{\sf
  arXiv:1502.06526v2}}

\bibitem[Kas87]{Kassel}
C.~Kassel, \emph{Cyclic Homology, Comodules, and Mixed Complexes}, J. Algebra
  \textbf{107} (1987), pp.~195--216

\bibitem[Kel99]{keller1999cyclic}
B.~Keller, \emph{On the cyclic homology of exact categories}, Journal of Pure
  and Applied Algebra \textbf{136} (1999), no.~1, pp.~1--56

\bibitem[KL09]{khovanov2009diagrammatic}
M.~Khovanov and A.~Lauda, \emph{A diagrammatic approach to categorification of
  quantum groups I}, Representation Theory of the American Mathematical Society
  \textbf{13} (2009), no.~14, pp.~309--347

\bibitem[KL11]{khovanov2011diagrammatic}
\bysame, \emph{A diagrammatic approach to categorification of quantum groups
  II}, Transactions of the American Mathematical Society \textbf{363} (2011),
  no.~5, pp.~2685--2700

\bibitem[Kle01]{Klein}
J.~R. Klein, \emph{The dualizing spectrum of a topological group}, Math. Ann.
  \textbf{319} (2001), no.~3, pp.~421--456

\bibitem[Lur17a]{HA}
J.~Lurie, \emph{Higher Algebra}, September 2017,
  \url{http://www.math.harvard.edu/~lurie/papers/HA.pdf}

\bibitem[Lur17b]{HTT}
\bysame, \emph{Higher Topos Theory}, April 2017,
  \url{http://www.math.harvard.edu/~lurie/papers/HTT.pdf}

\bibitem[Lur18]{SAG}
\bysame, \emph{Spectral Algebraic Geometry}, February 2018,
  \url{http://www.math.harvard.edu/~lurie/papers/SAG-rootfile.pdf}

\bibitem[Mac74]{macpherson1974chern}
R.~D. MacPherson, \emph{Chern classes for singular algebraic varieties}, Ann.
  of Math. (2) \textbf{100} (1974), pp.~423--432

\bibitem[Mar09]{M}
N.~Markarian, \emph{The Atiyah class, Hochschild cohomology and the
  Riemann-Roch theorem.}, J. Lond. Math. Soc. (2) 79 (2009), no. 1 (2009),
  pp.~129--143

\bibitem[McC94]{mccarthy}
R.~McCarthy, \emph{The cyclic homology of an exact category}, Journal of pure
  and applied algebra \textbf{93} (1994), no.~3, pp.~251--296

\bibitem[Mor67]{Morton}
H.~R. Morton, \emph{Symmetric products of the circle}, Proc. Camb. Phil. Soc.
  \textbf{63} (1967), no.~2, pp.~349--352

\bibitem[Orl93]{orlov1993projective}
D.~O. Orlov, \emph{Projective bundles, monoidal transformations, and derived
  categories of coherent sheaves}, Russian Academy of Sciences. Izvestiya
  Mathematics \textbf{41} (1993), no.~1, p.~133

\bibitem[Pre15]{preygel2014ind}
A.~Preygel, \emph{Ind-coherent complexes on loop spaces and connections},
  Stacks and categories in geometry, topology, and algebra, Contemp. Math.,
  vol. 643, Amer. Math. Soc., Providence, RI, 2015, pp.~289--323

\bibitem[{Pst}14]{Pstragowski}
P.~{Pstr{\c a}gowski}, \emph{{On dualizable objects in monoidal bicategories,
  framed surfaces and the Cobordism Hypothesis}}, Master's thesis, November
  2014, \href{http://arxiv.org/abs/1411.6691}{{\sf arXiv:1411.6691}}

\bibitem[Rob15]{robalo}
M.~Robalo, \emph{$K$-theory and the bridge from motives to noncommutative
  motives}, Adv. Math. \textbf{269} (2015), pp.~399--550

\bibitem[Rou08]{rouquier20082}
R.~Rouquier, \emph{2-Kac-Moody algebras}, arXiv preprint arXiv:0812.5023 (2008)

\bibitem[RV16]{RiehlVerity}
E.~Riehl and D.~Verity, \emph{Homotopy coherent adjunction's and the formal
  theory of monads}, Adv. Math. \textbf{286} (2016), pp.~802--888

\bibitem[Sch09]{schuermann2009specialization}
J.~Sch\"urmann, \emph{Specialization of motivic Hodge-Chern classes}, arXiv
  preprint arXiv:0909.3478 (2009)

\bibitem[TV09]{TV1}
B.~To{\"e}n and G.~Vezzosi, \emph{Chern Character, Loop Spaces and Derived
  Algebraic Geometry}, Algebraic Topology (N.~Baas, E.~M. Friedlander,
  B.~Jahren, and P.~A. {\O}stv{\ae}r, eds.), Abel Symposia, vol.~4, Springer
  Berlin Heidelberg, 2009, pp.~331--354

\bibitem[TV15]{TV2}
B.~To{\"e}n and G.~Vezzosi, \emph{Caract{\`e}res de Chern, traces
  {\'e}quivariantes et g{\'e}om{\'e}trie alg{\'e}brique d{\'e}riv{\'e}e},
  Selecta Math. \textbf{21} (2015), no.~2, pp.~449--554

\end{thebibliography}

\end{document}